\documentclass[11pt, reqno, a4paper]{amsart}
\usepackage{a4, amsmath, bm, amssymb}
\usepackage{setspace}
\usepackage{cite}
\usepackage{enumerate}
\usepackage{url}
\usepackage{xcolor}
\usepackage{mdframed}

\usepackage{dsfont}
\usepackage{amsthm}

\makeatletter
\newtheorem*{rep@theorem}{\rep@title}
\newcommand{\newreptheorem}[2]{%
\newenvironment{rep#1}[1]{%
 \def\rep@title{#2 \ref{##1}}%
 \begin{rep@theorem}}%
 {\end{rep@theorem}}}
\makeatother

\theoremstyle{plain}
\newtheorem{theorem}{Theorem}
\newtheorem{lemma}[theorem]{Lemma}
\newtheorem{proposition}[theorem]{Proposition}
\newtheorem{corollary}[theorem]{Corollary}

\newtheorem{assumption}[theorem]{Standing Assumption}
\theoremstyle{definition}
\newtheorem{definition}[theorem]{Definition}

\newtheorem{remark}[theorem]{Remark}

\numberwithin{equation}{section}
\numberwithin{theorem}{section}

\begin{document}
  
\normalsize

\title[Expanders and decay of correlations]{Relative ($\tau$), expanders, and decay of correlations for certain expanding maps}

\author{Rhiannon \textsc{Dougall}}

\address{Department of Mathematical Sciences,
Durham University,
Upper Mountjoy,
Durham DH1 3LE}
\email{rhiannon.dougall@durham.ac.uk}

\begin{abstract}
Relative ($\tau$) is equivalent to a statement that the sequence of Cayley graphs associated to group quotients $\Gamma_q=\Gamma/N_q$, $q\in\mathbb{N}$, form an expander family. There is a philosophy that expander graphs give rise to good mixing; for instance, one has exponential mixing for the geodesic flow uniformly the along a family of congruence covers of the modular surface, stemming from the symmetry in the $\mathrm{SL}(2,\mathbb{R})$ action and the uniform spectral gap for the Laplacian. Do we see similar phenomena in less structured settings? We investigate this question for a tower of finite sheeted covers of certain expanding dynamical systems. We use transfer operator machinery that applies in particular in the cases of subshifts of finite type and expanding interval maps. We make fruitful connections with KMS states of Cuntz--Krieger algebras.
\end{abstract}

\maketitle

\section{Introduction}\label{section:intro}
There is a philosophy that expander graphs give rise to good mixing. It is known that expander graphs have Cheeger constant uniformly bounded away from $0$ and that this in turn implies fast mixing for the simple walk on the graph. (See for instance \cite{Lubotzky}.) A well-known example is constructed from the quotients of $\Gamma=\mathrm{SL}(2,\mathbb{Z})$ by certain principle congruence subgroups $N_q$, $q\in\mathbb{N}$. (A principle congruence subgroup is the kernel of the map corresponding to reduction modulo a natural number, and the expander property for the sequence of reductions modulo primes is an observation of Margulis using Selberg's $3/16$ theorem.) One also says that $\mathrm{SL}(2,\mathbb{Z})$ has Property ($\tau$) relative to those $N_q$ -- this is a property of related unitary representations (we make the definition precise in Section \ref{section:tau}; more generally we refer to \cite{bekka}.) The quantifiers in the exponential mixing property for the geodesic flow on $\mathrm{SL}(2,\mathbb{R})/N_q$ are uniform in the given family of $N_q$, $q\in\mathbb{N}$. The aforementioned mixing result relies on a representation-theoretic approach for the dynamics of the geodesic flow via the $\mathrm{SL}(2,\mathbb{R})$ action (see \cite{EW} for an introductory exposition) in conjunction with a spectral gap for the Laplacian coming from the expander condition. The last 50 years has seen the development of tools to study hyperbolic flows, which recovers certain dynamical phenomema for geodesic flows but without the repesentation theoretic input. Does one still expect good mixing behaviour for analogous extensions of more general maps with ``hyperbolic" behaviour? (We won't study the geodesic flow case but discuss its state of the art later.)

A sequence of $k$-regular graphs $\mathcal{G}_q$, $q\in\mathbb{N}$, is said to form an expander family if their Cheeger constants are uniformly bounded away from zero. One such model is given by Cayley graphs, namely one can suppose that $\mathcal{G}_q=\mathrm{Cay}(\Gamma_q, S_q)$ are a tower of Cayley graphs of given by quotients $\Gamma_q=\Gamma/N_q$ of a group $\Gamma$, with edges given by the image $S_q$ of a fixed finite generating set $S$ of $\Gamma$. For brevity, we say that $(\Gamma_q)_q$ are an expander family if the previously given Cayley graphs are expanders for some (any) generating $S$. One expects a kind of rigidity for these objects: Lubotszky and Weiss \cite{LW} pose the question of whether there are finite groups $\Gamma_q$ and generating sets $S_q,S^\prime_q$ of the same constant cardinality and with $\mathrm{Cay}(\Gamma_q,S_q)$ being an expander family but $\mathrm{Cay}(\Gamma_q,S^\prime_q)$ failing to be an expander family. The questions asks, for example, whether it is sufficient that the projection of $S$ to $\Gamma/N_q$ is generating in the previous criterion. We find this perspective in work of Shalom \cite{Shalom} and make use of his criterion.

We investigate the relationship between dynamics and the expander hypothesis for certain expanding maps. Our results are demonstrated to apply when $(X,T,\mu)$ is one of the following two examples: a subshift of finite type, or an expanding map of an interval. (In Section \ref{section:vectorspaces} we give a broader abstract framework, referred to as \textnormal{(Base)}, in terms of ``weak" and ``strong" norms $\|\cdot\|_w, \|\cdot\|_s$.)

\begin{definition}[Lipschitz observables when base dynamics are a subshifts finite type.]\label{exampleA} See \cite{PP}. We say that $(X,T)$ is a subshift of finite type if $X$ is the subset of $\left\{ 1,\ldots, k\right\}^{\mathbb{N}_0}$ given by disallowing finitely many transitions $x_i x_{i+1}$, and $T$ is the left shift given by $(Tx)_i= x_{i+1}$, $i\in \mathbb{N}_0$. We let $\mu$ be an equilibrium measure for a H\"{o}lder continuous potential $h=\log p$. 
We let $\mathcal{F}_w(T)=C(X,\mathbb{R})$ equipped with the supremum norm $\|\cdot\|_w=\|\cdot\|_\infty$. The measure $\mu$ is assumed to be an equilibrium state for a function $h=\log p$ which is Lipschitz in the metric 
$$
d(x,y)=\exp ( \min \left\{i\in\mathbb{N}_0 : x_i\ne y_i\right\}\log \theta),
$$
with $\theta<1$. We let $\mathcal{F}_s(T)$ be the Lipschitz functions in this metric equipped with the norm $\|f\|_s = |f|_L + \|f\|_{w}$, where
$$
|f|_L := \sup_{x\ne y, x,y\in X} |f(x)-f(y)| /d(x,y).
$$ 
A decay of correlations result is observed in \cite{PP} under the assumption that $(X,T,\mu)$ is weak-mixing (which is implied by the transition matrix being primitive); namely there are $C,\beta>0$ with
$$
\left| \int  f \cdot g\circ T^n \,d\mu -\int f\, d\mu\int g\,d\mu\right|\le Ce^{ -n\beta}\|f\|_s\|g\|_{L^{1}_\mu(X,\mathbb{C})} ,
$$
for $f\in \mathcal{F}_s(T)$ and $g\in L^{1}_\mu(X,\mathbb{C})$.
\end{definition}
\begin{definition}[Bounded variation observables when base dynamics are expanding maps of the interval.]\label{exampleB} For general results see \cite{Wong}, \cite{KellerH}. The general set-up is as follows. The phase space is $X=[0,1]$, and we consider an expanding piecewise nearly $C^1$ map $T:X\to X$; namely there are $0=x_1<\ldots <x_{k+1}=1$ such that on the interior of $[i]:=[x_i, x_{i+1}]$ we have that $T$ is $C^1$ and has derivative $|T^\prime (x)|>1$ uniformly. (The $x_i$ need not be the smallest/largest possible.)

We let $\mu$ be a $T$-invariant absolutely continuous (with respect to Lebesgue) probability measure. We assume that the density $h$ has bounded variation, and we shall further assume that $h$ only vanishes on a set of zero Lebesgue measure.
We let $\mathcal{F}_w(T) =L^1(X,\mathbb{R},\mu)$ and $\|\cdot\|_w=\|\cdot\|_{L^1}$. We let $\mathcal{F}_s(T)\subseteq \mathcal{F}_w(T)$ be the functions of bounded variation, namely for $f:X\to \mathbb{R}$ we define,
$$
|f|_{\mathrm{BV}}:=\sup\left\{ \sum_{j=1}^n |f(y_{j+1}) - f(y_{j})| : y_1=0\le  \cdots \le y_{n+1} = 1 \right\} <\infty.
$$
and then denote $|\cdot|_{\mathrm{BV}}$ the function on $L^1(X,\mathbb{R},\mu)$ given by taking the infimum over representatives.
We equip $\mathcal{F}_s(T)$ with the norm $\|\cdot \|_s =|\cdot|_{\mathrm{BV}}+ \|\cdot\|_{L^1}$. 
In general $h$ may fail to be bounded away from zero and this affects the resulting decay of correlations statement. In \cite{Liverani}, under a weak covering and covering hypothesis (which in pariticular implies that $h$ is bounded away from $0$), one has the existence of $C,\beta>0$ with
$$
\left| \int  f \cdot g\circ T^n \,d\mu -\int f\, d\mu\int g\,d\mu\right|\le Ce^{ -n\beta}\|f\|_s\|g\|_{L^{1}_\mu(X,\mathbb{C})},
$$
for $f\in \mathcal{F}_s(T)$ and $g\in L^{1}_\mu(X,\mathbb{C})$. More generally, provided $(X,T,\mu)$ is weak mixing, one has
the existence of $C,\beta>0$ with
$$
\left| \int  f \cdot g\circ T^n \,d\mu -\int f\, d\mu\int g\,d\mu\right|\le Ce^{ -n\beta}\|f\|_s\|g\|_{L^{\infty}_\mu(X,\mathbb{C})} ,
$$
for $f\in \mathcal{F}_s(T)$ and $g\in L^{\infty}_\mu(X,\mathbb{C})$. Having an exponential decay of correlations is stronger than knowing a central limit theorem for observables in $\mathcal{F}_s(T)$. The central limit theorem in this case states that the random variables
$$
\frac{1}{\sqrt{N}} \left(\sum_{n=0}^{N-1} f\circ T^n -\int f\,d\mu \right)
$$
converge in distribution to a normal distribution with variance 
$$
\sigma_{T}(f) = \lim_{n\to\infty} \frac{1}{n}\int \left(\sum_{k=0}^{n-1} f\circ T^k - \int f\,d\mu\right)^2 d\mu_.
$$ 
Moreover, the rate $\sqrt{N}$ is optimal in the sense that the only functions vanishing variance (a trivial limit distribution) are those that satisfy the coboundary equation
$$
f = g\circ T - g +c
$$
for some $g\in \mathcal{F}_s$ and $c\in\mathbb{R}$.
\end{definition}

How does the statistical behaviour change for $(X_q,T_q,\mu_q)$ that are ``locally the same as" $(X,T,\mu)$? We construct a family of extensions $(X_q,T_q,\mu_q)$ of $(X,T,\mu)$ associated to the data of a countable group $\Gamma$, a sequence of finite quotients $\Gamma_q$ ($q\in \mathbb{N}$), and a map $\psi:X\to \Gamma$ which is assumed constant mod $\mu$ on each $[i]\cap T^{-1}[j]$, $[i],[j]\in\alpha$. We construct $X_q$ as the product $X\times \Gamma_q$ together with skew product dynamics $T_q(x,g) = (Tx,g\psi_q(x))$ with $\psi_q$ the composition of $\psi$ with the projection $\Gamma\to\Gamma_q$, and measure $\mu_q$ the product of $\mu$ with counting measure on $\Gamma_q$. (We also refer to the probability measure $\overline{\mu}_q=(\#\Gamma_q)^{-1} \mu_q$.) 
We say that such a $(X_q,T_q,\mu_q)$ is a \emph{finite sheeted cover} of $(X,T,\mu)$. We will say that a sequence $(X_q,T_q,\mu_q)_{q\in\mathbb{N}}$ is a \emph{tower} of finite sheeted covers if each $(X_q,T_q,\mu_q)$ is a finite sheeted cover of $(X_{q-1},T_{q-1},\mu_{q-1})$, for any $q> 1$; equivalently, $\Gamma_{q}$ is a quotient of $\Gamma_{q-1}$ for each $q>1$. (More generally, given any quotient $G$ of $\Gamma$ we denote $T_{G}$ the skew product dynamics $(x,g)\mapsto (Tx,g\psi_G(x))$ with phase space $X\times G$ and with $\psi_G(x)$ the composition of $\psi$ with a fixed epimorphism $\Gamma\to G$.)
\begin{remark}
In each of our examples for $(X,T,\mu)$ we can see $(X_q,T_q,\mu_q)$ as being of the same ``class" as $(X,T,\mu)$. The group extension of a finite type is again a subshift of finite type, (with alphabet $\left\{1,\ldots, k\right\}\times \Gamma_q$), and $\mu_q$ is again a Gibbs measure. The case of interval maps is similar, where we realise the phase space as $[0,\#\Gamma_q]$ and $T_q$ as a particular map from $\#\Gamma_q$-many copies of each branch of $T$.
\end{remark}

\begin{definition}\label{def:exampleA2}
In the case of a base dynamics as in Definition \ref{exampleA} we take $\mathcal{F}_w(T_q)$ as the functions $f$ defined on $X\times \Gamma_q$ for which 
$$
\|f\|_{\mathcal{F}_w(T)}:= \left\| x \mapsto \sum_{\gamma\in \Gamma_q} |f(x, \gamma)|^2\right\|_{\infty} <\infty,
$$ 
and $\mathcal{F}_s(T_q)$ the functions defined on $X\times \Gamma_q$ for which 
\begin{equation}\label{eq:lipschitzgpext}
|f|_{L,\Gamma_q}: = \sup_{x\ne y, x,y\in X} \frac{1}{d(x,y)}\sqrt{\sum_{\gamma \in \Gamma_q} |f(x,\gamma)-f(y,\gamma)|^2}<\infty,
\end{equation}
and set $\|f\|_{\mathcal{F}_s(T_q)} = |f|_{L,\Gamma_q} + \|f\|_{\mathcal{F}_q(T_q)}$.
\end{definition}
\begin{definition}\label{def:exampleB2}
For In the case of a base dynamics as in Definition \ref{exampleB} we take $\mathcal{F}_w(T_q)= L^1(X\times \Gamma_q,\mathbb{R}, \mu)$, 
$$
\|f\|_{\mathcal{F}_w(T_q)} = \int \sqrt{\sum_{\gamma \in \Gamma_q} |f(x,\gamma)|^2}\,d\mu(x)
$$ 
and
$\mathcal{F}_s(T_q)$ the functions $f:X\times \Gamma_q\to\mathbb{R}$ for which
\begin{equation}\label{eq:bvgpext}
|f|_{\mathrm{BV},\Gamma_q}:=\sup\left\{ \sum_{j=1}^n \sqrt{\sum_{\gamma\in \Gamma_q}|f(y_{j+1},\gamma) - f(y_{j},\gamma)|^2} : y_1=0\le  \cdots \le y_{n+1} = 1 \right\} ,
\end{equation}
is finite, and then denote $|\cdot|_{\mathrm{BV},\Gamma_q}$ the function on $L^1(X\times \Gamma_q,\mathbb{R}, \mu)$ given by taking the infimum over representatives.
\end{definition}
These examples satisfy a broader abstract framework, referred to as \textnormal{($\mathbb{C}$-Ext)}. In particular these include a uniform Doeblin--Fortet inequality and relative pre-compactness for the inclusions of associated strong and weak spaces.
Once we take an isomorphism with vector valued functions the norms in Definitions \ref{def:exampleA2} and \ref{def:exampleB2} are respectively, the supremum norm, the Lipschitz semi-norm, the $L^1$ norm, the BV semi-norm, for functions with values in $\mathbb{C}^{\Gamma_q}$, the representation space for $\Gamma_q$. (See Propositions \ref{prop:lipequivalent} and \ref{prop:bvequivalent}.)
Further, for this natural permutation action of $\Gamma_q$ on $\mathbb{C}^{\Gamma_q}$, if one has a decomposition $\mathbb{C}^{\Gamma_q} = V_q\oplus \mathcal{H}_q$, with both $V_q$, $\mathcal{H}_q$ invariant by $\Gamma_q$, then one has an analogous decomposition $\mathcal{F}_\bullet(T_q)=\mathcal{F}_\bullet(T_q;V_q)\oplus \mathcal{F}_\bullet(T_q;\mathcal{H}_q)$ for each of $\bullet=s,w$.

Recall that decay of correlations for our example $(X,T,\mu)$ follows under a weak mixing hypothesis. One might think to ask that each $T_q$ is weak-mixing, but even then we run into the (hard) problem of whether the expanders property persists relative to changing generating sets.
We impose some extra structure concerning the monodromy of orbits (one can see this as a weakening of transitivity for the $\Gamma$-extension). Given $\alpha$, we make reference to the process of an orbit hitting elements of $\alpha$. Namely, given $x\in X$ its \emph{hitting sequence} is any $i_0,\ldots, i_n$ with $T^k x\in [i_k]$ for all $0\le k< n$, $n\in \mathbb{N}$. 
Equivalently, $x\in [i_0]\cap \cdots \cap T^{-n}[i_n]$. (Since $\alpha$ is a partition mod $\mu$ almost every $x$ has a unique hitting sequence.) We denote $\Sigma_{n+1}$ the collection of all hitting sequences $i_0,\ldots, i_n$ with $\mu([i_0]\cap \cdots \cap T^{-n}[i_n])>0$.
\begin{definition}\label{def:monodromy1}
Given $x\in X$, we refer to $\gamma=\psi(x)\cdots \psi(T^{n-1}x)$ as the \emph{monodromy of the orbit} determined by the point $x\in X$. Given $j_0,\ldots, j_n\in \Sigma_n$, we let $\psi(j_0,\ldots, j_n)$ be the almost sure value of the monodromy on $[j_0]\cap \cdots \cap T^{-n}[j_n]$, and refer to all possible elements as \emph{the monodromy set}. 
Denote $H(i_0,\ldots, i_k)$ the collection of all $\psi(j_0,\ldots, j_n)$, for $j_0,\ldots, j_n\in \Sigma_n$, $n\in\mathbb{N}$, with common prefix $i_0,\ldots, i_k$. We call this $H(i_0,\ldots, i_k)$ \emph{based monodromy set}. 

We say that a finite set $F\subset \Gamma$ \emph{embeds} into $H(i_0,\ldots ,i_k)$ if there are $g_-,g_+\in \Gamma$ such that for each $\gamma\in F$ there are $i^\gamma_0,\ldots, i^\gamma_{n_\gamma}$ with common prefix $i_0,\ldots, i_k$, and with $\psi(i^\gamma_0,\ldots, i^\gamma_{n_\gamma}) = g_-\gamma g_+$. 
We write $\Sigma_k(F)$ for those $i_0,\ldots, i_k$ such that an embedding exists.
\end{definition}

It is natural to wonder whether the based monodromy sets have any additional structure. In the best possible case $H(i_0)$ is a semigroup that coincides with $\Gamma$; failing that we might hope that $H(i_0)$ contains a semigroup that coincides with $\Gamma^\prime$ for $\Gamma^\prime$ sufficiently large in $\Gamma$. The latter case occurs (with $\Gamma^\prime$ finite index in $\Gamma$) when considering the Markov map coding of certain geometric problems \cite{Coulon}. 
In general it is not clear whether one could find a non-trivial subsemigroup of $H(i_0)$. On the other hand, it is natural to ask that we capture something of the group structure, and so we shall impose as an asssumption the existence of embeddings into some $H(i_0,\ldots, i_k)$. 
Another obstacle to relating the unitary structure of the group and measure dynamics is the distribution of the measure $\mu$ on cylinder sets. Consider the case of an expanding interval map as in Definition \ref{exampleB}. When $\alpha$ is a Markov partition, there is some $t_0>0$ with $\mu(T^n( [i_0]\cap \cdots \cap T^{-n}[i_n]))\ge t_0$ uniformly in all those $i_0,\ldots, i_n$ with $[i_0]\cap \cdots \cap T^{-n}[i_n]$ non-null. Denote, in the general case,
$$
d(n) =\min \left\{ \mu(T^n( [i_0]\cap \cdots \cap T^{-n}[i_n])): [i_0]\cap \cdots \cap T^{-n}[i_n]\ne \emptyset \;\mathrm{mod}\,\mu\right\}.
$$
There are non-Markov examples where one has $d(n)\to 0$ as $n\to\infty$ and with different speeds; for instance, for a $\beta$-expanding map (see \cite{BW}) one has that $d(n)$ is related to the number of zeros in the $\beta$-expansion of the number $1$.

We explain our requirement on the monodromy in the following way, permitting ourselves to first give a proposition using the subsequent definition. 
 \begin{proposition}[Sufficient monodromy hypothesis]\label{prop:monodromy}
Let $d\mu = h\,dm$, where $m$ is a probability measure, and let $p\in [1,\infty]$. Given a finite $S\subseteq \Gamma$ and given $\kappa^\prime>0$ there exists some $n,k$ such that the monodromy \emph{$(k,C,L^p,(\kappa^\prime)^n)$-sees $S^n$} if $t_{n,k}$, the minimum $m$-measure of 
$$
T^k([i_0]\cap \ldots \cap T^{-k}[i_{k}]) \cap \bigcap T^{n_\gamma}([i^\gamma_0]\cap \ldots \cap T^{-n_\gamma}[i^\gamma_{n_{\gamma}}])
$$
where $i_0,\ldots, i_{k}\in \Sigma_k(S^n)$ and $i^\gamma_0,\ldots, i^\gamma_{n_\gamma}$ are given associated embeddings,
satisfies that 
$$
t_{n,k} \inf_{x\in X} h(x) > (\kappa^\prime)^n + \frac{C}{\#\Sigma_k(S^n)},
$$ 
for some $n,k$.
\end{proposition}

\begin{remark}
Note that the hypotheses of the proposition can only be met when $h$ is bounded away from zero. We give a couple of examples in Section \ref{section:monodromyexamples}. Firstly,
Markov maps easily satisfy the definition when $\alpha$ is the Markov partition, and this extends to other partitions with a constraint on the dynamics of boundary points. Secondly, we may be lucky and have that the embeddings always manage to miss any bad sets, thus having $t_{n,k}$ uniformly bounded (such a map is given in example Section \ref{section:monodromyexamples}). 
\end{remark}

\begin{definition}[General monodromy hypothesis]\label{def:monodromy3}
Let $p\in [1,\infty]$. Given a finite $S\subseteq \Gamma$ and given $n,k\in\mathbb{N}$ and $\epsilon>0$, the monodromy \emph{$(k,C,L^p,\epsilon)$-sees $S^n$} if we have the inequality 
$$
t_n>\epsilon + C t_n^{\prime\prime}
$$ 
where $t_n,t^{\prime\prime}_n$ are a convex combination of $t_n(\gamma), t_n^{\prime\prime}(\gamma)$ given by, respectively,
$$
t_n(\gamma)=\int_{E\cap E_\gamma} g \, h\circ (T^\dag_{i_{0},i_1}\circ \cdots T^\dag_{i_{n_\gamma-1},i_{n_\gamma}})\,d m,
$$
and
$$
 t_n^{\prime\prime}(\gamma)=\left(\int_{E\cap E_\gamma} g^p\,dm \right)^{1/p}(\#\Sigma_k(S^n))^{-1}
+\left(\int_{E \Delta E_\gamma} g^p\,dm \right)^{1/p},
$$
with $E = T^k([i_0]\cap \ldots \cap T^{-k}[i_{k}])$, $E_\gamma = T^{n_\gamma}([i^\gamma_0]\cap \ldots \cap T^{-n_\gamma}[i^\gamma_{n_{\gamma}}])$, for some at most unit norm $g\in L^p_m(X,\mathbb{C})$ when $p<\infty$, and otherwise we assume that such $g$ is a limit point of finite linear combinations and products of functions supported on cylinder sets. 
\end{definition}
\begin{remark}
Note that the final term conveniently vanishes for $g$ the indicator function on the intersection of $E$ and all $E_\gamma$, whence we see Proposition \ref{prop:monodromy}.
\end{remark}

If $(\Gamma_q)_q$ are an expander family then one sees a contraction for the permutation actions $\lambda_q$ on $\mathbb{C}^{\Gamma_q}$: there is a finite $S\subset \Gamma$ such that $(\# S)^{-1}\|\lambda_q(S)\|_{\mathbb{C}^{\Gamma_q}}<1$ uniformly in $q\in \mathbb{N}$. In certain geometric examples (such as Markov coding of hyperbolic group actions \cite{Coulon}) one only hopes for the monodromy to see a finite index subgroup. Notice that any finite index subgroup $\Gamma^\prime$ of $\Gamma$ in turn contains the finite index normal subgroup $\mathrm{Core}_\Gamma(\Gamma^\prime)=\cap_{\gamma\in \Gamma} \gamma^{-1}\Gamma^\prime \gamma$. Using work of Shalom \cite{Shalom} it is natural to try to find a normal subgroup $\Gamma^\prime$ that is co-amenable in $\Gamma$ (meaning $\Gamma/\Gamma^\prime$ is an amenable group) because the expander machinery passes to these. More precisely, assuming that $(\Gamma_q)_q$ are an expander family, there is a decomposition $\mathbb{C}^{\Gamma_q}=V_q\oplus \mathcal{H}_q$, $\lambda_q=\theta_q+\rho_q$ (see Theorem \ref{thm:tau}) such that the dimension of $V_q$ is uniformly bounded in $q\in\mathbb{N}$ and such that $(\# S)^{-1}\|\rho_q(S)\|_{\mathcal{H}_q}\le \kappa^\prime(S,(\mathcal{H}_q)_q)<1$ for some finite $S\subset \Gamma^\prime$. 
(In fact, these $V_q$ satisfy that $\mathcal{F}_\bullet(T_q;V_q)\cong \mathcal{F}_\bullet(T_G)$ for the group extension associated to a finite quotient $G$ of $\Gamma_q$, and by construction $V_q$ will always contain the constant function whence any $f_q\in \mathcal{F}_s(T_q;\mathcal{H}_q)$ has $\int f_q\,d\mu_q=0$.) In this way we only try to see statistical results modulo a subspace that is ``relatively bounded", loosely speaking.


\subsection{Decay of correlations along the tower}\label{subsection:decay}
\begin{theorem}\label{theorem:one}
Assume that $(X,T,\alpha,\mu,m)$ and $\mathcal{F}_s(T)$ satisfies \textnormal{(Base)} with $p\in[1,\infty]$ and ACIP density $h$ bounded away from $0$. Let $(X_q,T_q,\mu_q)_{q\in\mathbb{N}}$ be a tower of finite sheeted covers and assume that $\mathcal{F}_s(T_q)$ is equipped with a norm satisfying Assumption \textnormal{($\mathbb{C}$-Ext)} with uniform Doeblin--Fortet constant $C_{DF}$.

Fix a co-amenable normal subgroup $\Gamma^\prime$ of $\Gamma$ and corresponding decomposition $\mathbb{C}^{\Gamma_q}=V_q\oplus \mathcal{H}_q$.
Assume that $(\Gamma_q)_q$ are an expander family and let $S\subseteq \Gamma^\prime$ be a finite subset with $\kappa^\prime(S,(\mathcal{H}_q)_{q\in\mathbb{N}})<1$.
Assume that the monodromy $(C_{DF},\epsilon,L^p)$-sees $S^n$ for $n\ge N$ with $C_{DF} \kappa^\prime(S,(\mathcal{H}_q)_{q\in\mathbb{N}})^N\le \epsilon$.

There is a constant $\beta>0$ such for all $q\in\mathbb{N}$ there are constants $C_q$ with,
$$
\left|\int f_q \cdot g_q\circ T^n_q \,d\mu_q \right|\le C_qe^{ -n\beta}\|f_q\|_{\mathcal{F}_s(T_q)}\|g_q\|_{L^{1}_{\mu_q}(X_q,\mathbb{C})} ,
$$
for any $f_q\in  \mathcal{F}_s(T_q;\mathcal{H}_q)$ and $g_q\in L^{1}_{\mu_q}(X_q,\mathbb{C})$, provided $h$ is uniformly bounded away from $0$; and otherwise
$$
\left|\int f_q \cdot g_q\circ T^n_q \,d\mu_q  \right|\le C_q\#\Gamma_q e^{ -n\beta} \|f_q\|_{\mathcal{F}_s(T_q)}\|g_q\|_{L^{\infty}_{\mu_q}(X_q,\mathbb{C})} .
$$
for any $f_q\in  \mathcal{F}_s(T_q;\mathcal{H}_q)$ and $g_q\in L^{\infty}_{\mu_q}(X_q,\mathbb{C})$.
\end{theorem}

In the case of a Markov map, the conclusion of Theorem \ref{theorem:one} is foreshadowed by various works. 
In the case that one looks for a finite index normal subgroup (rather than co-amenable), the statement of the previous theorem in the context of subshifts of finite type and Lipschitz functions (Def \ref{exampleA}) was within the reach of the methods of \cite{Coulon} once one checks the essential spectral radius of the transfer operator. (Moreover, the argument has some core aspects in common, going via \cite{CoulonDougallSchapiraTapie}.) 

Indeed, the conclusion will not be a surprise given what is known for hyperbolic surfaces. There are various works that study convex cocompact Fuchsian groups using a thermodynamic approach and estimating the norms of iterates of the transfer operator: a statement on exponential mixing is given in \cite{OhWinter}, but crucially uses the machinery of congruence subgroups of $\mathrm{SL}(2,\mathbb{Z})$; this dependency is removed in a recent preprint which concerns the expander hypothesis and resonance free regions \cite{Soares}. 

We also like to mention some not-directly-related work that are also on this question of global structure and mixing. In the context of expanding interval maps \cite{holland} investigate the question of mixing properties for random permutations of the branches of the map. For certain subshifts of finite type and Markov measures one may find relations between mixing rates and the Ramanujan property in \cite{alina}.


\subsection{Central limit theorem and variance for the inverse limit}\label{subsection:variance}
There is an interplay between properties of the tower and properties of the inverse limit. 
Given a nested sequence of quotients $\Gamma_q$ we may form their inverse limit $\Gamma_\infty$ (a more standard notation for $\Gamma_\infty$ is $\varprojlim \Gamma_q$) --- this is a compact topological space with an action of $\Gamma$ preserving a probability measure which we denote as $\mathrm{Haar}$. The elements of $\Gamma_\infty$ can be realised as sequences $(g_q)_{q\in\mathbb{N}}$ in the direct product $\prod_{q=1}^\infty \Gamma_q$ for which each $g_{q+1}$ maps onto $g_q$ under the factor map $\Gamma_{q+1}\to \Gamma_q$. The key property use use is that there are natural factor maps $\Gamma_\infty \to \Gamma_q$ such that the pushforward of $\mathrm{Haar}$ is the normalized counting measure on $\Gamma_q$, and conversely, there are nested copies of $\Gamma_q$ data in $\Gamma_\infty$ data, and the latter is exhausted by the collection of the former (we make this more precise later).

Given dynamical systems on a compact phase space one may consider the inverse limit dynamical system. There has been interest in obtaining categorical-type statements for inverse limit dynamical systems \cite{Brown}; we single out the following: the inverse limit dynamics are mixing precisely when all those in the defining sequence are mixing, the inverse limit dynamics has a unique measure of maximal entropy precisely when each of the defining sequence do. We won't define inverse limit dynamical systems since the context is more general than ours. We assert that the inverse limit of $(X_q,T_q,\mu_q)$ is $(X_\infty, T_\infty,\mu_\infty)$ where $X_\infty = X\times \Gamma_\infty$, $T_\infty (x,\xi) = (Tx, \psi_\infty (x) \xi)$ where $\psi_\infty (x)$ is the composition of $\phi$ with the action on $\Gamma_\infty$, and $\overline{\mu}_\infty$ is the product of $\mu$ and $\mathrm{Haar}$.
  
We consider observables that \textit{a priori} belong to $L^2_{\mu_{\infty}}(X_\infty,\mathbb{C})$. We make the convention that $\mathcal{F}_s(T_q)$ are vector spaces and that $\mathcal{B}_s(T_q)=\mathcal{F}_s(T_q)$ is understood as a Banach space (this distinction is useful for extending to $q=\infty$). The nesting property of $\Gamma_q$ gives rise to a nesting property for embeddings of $\mathcal{F}_s(T_q)$ in $L^2_{\mu_{\infty}}(X_\infty,\mathbb{C})$. In this way one may construct $(\mathcal{B}_s(T_\infty),\|\cdot\|_{\mathcal{B}_{s}(T_\infty)})$ such that their projections to $(\mathcal{F}_s(T_q),\|\cdot\|_{\mathcal{F}_s(T_q)})$ are surjective of norm $1$ and equivariant with repsect to $T_\infty,T_q$; and such that the linear span of embeddings of $(\mathcal{F}_\bullet(T_q),\|\cdot\|_{\mathcal{F}_\bullet(T_q)})$ in  $(\mathcal{B}_s(T_\infty),\|\cdot\|_{\mathcal{B}_{s}(T_\infty)})$ are dense. In addition, these properties pass to the restrictions $\mathcal{F}_s(T_q;V_q),\mathcal{F}_s(T_q;\mathcal{H}_q)$. We make these constructions precise in Section \ref{section:proofmain}. 

Focussing on the (nicer) case that the ACIP density $h$ is bounded away from zero, the conclusions of Theorem \ref{theorem:one} only imply exponential decay of correlations with respect to $T_\infty$ when the constants $C_q$ are uniformly bounded. Nevertheless, a good understanding of variance persists, and this is what we describe next. 

For $q\in\mathbb{N}\cup\left\{\infty\right\}$, and a function $f\in\mathcal{B}_s(T_q)$ with $\int f\,d\overline{\mu}_q=0$, the non-negative number
\begin{equation}\label{eq:variance}
\sigma_{T_q}(f) = \lim_{n\to\infty} \frac{1}{n}\int \left(\sum_{k=0}^{n-1} f\circ T^k_q\right)^2 d\overline{\mu}_q = \left\|n^{-1/2}\sum_{k=0}^{n-1}f\circ T^k_q\right\|_{L^2_{\overline{\mu}_q}(X_q,\mathbb{C})}^2,
\end{equation}
is called the variance of $f$ (provided the limit exists). In the case that $q<\infty$ and assuming $T_q$ is mixing, one can obtain a central limit theorem for the Birkhoff sums with scaling $\sqrt{n}$, for which the limiting normal distribution has variance precisely $\sigma_{T_q}(f)$. Moreover,  one can observe that variance is continuous in $\mathcal{F}_s(T_q)$ and one can characterize that $f$ has $\sigma_{T_q}(f) =0$ precisely when it is a coboundary, i.e.\ when it satisfies the coboundary equation
$$
f = g-g\circ T_q,
$$
for some $g\in \mathcal{F}_s(T_q)$.

\begin{theorem}\label{theorem:three}
Assume that  $(X,T,\alpha,\mu,m)$ and $\mathcal{F}_s(T)$ satisfies \textnormal{(Base)} with $p\in[1,\infty]$. Let $(X_q,T_q,\mu_q)_{q\in\mathbb{N}}$ be a tower of finite sheeted covers and assume that $\mathcal{F}_s(T_q)$ is equipped with a norm satisfying Assumption \textnormal{($\mathbb{C}$-Ext)} with uniform Doeblin--Fortet constant $C_{DF}$.

Fix a co-amenable normal subgroup $\Gamma^\prime$ of $\Gamma$ and corresponding decomposition $\mathbb{C}^{\Gamma_q}=V_q\oplus \mathcal{H}_q$.
Assume that $(\Gamma_q)_q$ are an expander family and let $S\subseteq \Gamma^\prime$ be a finite subset with $\kappa^\prime(S,(\mathcal{H}_q)_{q\in\mathbb{N}})<1$.
Assume that the monodromy $(k,C_{DF},L^p, \epsilon)$-sees $S^n$ for $n\ge N$ with $C_{DF} \kappa^\prime(S,(\mathcal{H}_q)_{q\in\mathbb{N}})^N\le \epsilon$.

Assume that $h$ is bounded away from zero. For each $q\in\mathbb{N}\cup\left\{ \infty\right\}$, the limit defining for $\sigma_{T_q}(f)$ in \ref{eq:variance} exists for any $f\in \mathcal{B}_s(T_q;\mathcal{H}_q)$. We have $\sigma_{T_q}: \mathcal{B}_s(T_q;\mathcal{H}_q)\to [0,\infty)$ is continuous in the $\mathcal{B}_s(T_q;\mathcal{H}_q)$ topology, and vanishes precisely on those functions that are coboundaries.
\end{theorem}

The conclusions of Theorem \ref{theorem:three} upgrades a central limit theorem from $q<\infty$ to $q=\infty$. The key step is, for any fixed $a\in\mathbb{R}$,
\begin{align*}
&\left| \overline{\mu}_\infty \left(\frac{\sum_{k=0}^{n-1} f\circ T_\infty^k}{\sqrt{n}} \in [-a,a]\right)- \overline{\mu_q} \left(\frac{\sum_{k=0}^{n-1} f_q\circ T_q^k}{\sqrt{n}} \in [-a,a]\right) \right|
\\
&\le \left\|n^{-1/2}\sum_{k=0}^{n-1} g_q \circ T_\infty^k \right\|_{L^2_{\overline{\mu}_\infty}(X_q,\mathbb{C})},
\end{align*}
for associated $f_q\in\mathcal{B}_{s}(T_q;\mathcal{H}_q)$ with $\lim_{q\to\infty}\sigma_{T_q}(f_q)=\sigma_{T_\infty}(f_\infty)$, and $g_q\in\mathcal{B}_{s}(T_\infty;\mathcal{H}_\infty)$ with $\lim_{q\to\infty} g_q= 0$ in $\mathcal{B}_{s}(T_\infty;\mathcal{H}_\infty)$. 

The Delorme--Guichardet theorem characterizes Property (T) for locally compact groups in terms of vanishing first cohomology (which colloquially says ``all $1$-cocycles are $1$-coboundaries") \cite{bekka}. There is an analogous characterization for Property ($\tau$) and Relative ($\tau$) (however one has to beware of possible homomorphisms to $\mathbb{R}$; see for instance \cite{LubZuk}). In each case, an important step is the reduction of bounded $1$-cocycles to $1$-coboundaries. The $1$-cocyles are analogous to Birkhoff sums in our setting, and the na\"{i}ve variance formula compares their growth to $\sqrt{n}$. The parallelism we make is to suggest that, in our dynamical context, the sharpness of the $\sqrt{n}$ deviations (sharpness in the sense that \emph{only} the coboundaries vanish) could be valid more widely than the spectral gap picture.


\section{Relative ($\tau$), non-generating subsets, and the inverse limit}\label{section:tau}
We begin with a few general definitions and conventions pertaining to unitary representations. A unitary representation $(\rho,\mathcal{H})$ (or just $\rho$) of a group $G$ is a homomorphism $\rho$ from $G$ to the group of unitary operators of a (complex) Hilbert space $\mathcal{H}$. 
Two unitary representations $(\rho_1,\mathcal{H}_1)$, $(\rho_2,\mathcal{H}_2)$ are equivalent (we simply write $\rho_1=\rho_2$) if there is a surjective isometry from $ a:\mathcal{H}_1\to \mathcal{H}_2$ that is equivariant for $\rho_1$, $\rho_2$, i.e. $a\rho_1(\gamma) = \rho_2(\gamma)a$ for every $\gamma\in G$. We say that $\rho$ decomposes as $\rho = \rho_1 + \rho_2$ if there are closed invariant subspaces $\mathcal{H}_1, \mathcal{H}_2$ such that the restriction of $\rho$ to $\mathcal{H}_i$ is equivalent to $\rho_i$ for each $i=1,2$. Another notation of ``same" comes from factoring and lifting. If $H\unlhd G$ and the representation $\rho$ of $G$ is trivial on $H$ then $\rho$ factors as a representation of $G/H$. Conversely, given a unitary representation of $G/H$ one can lift this to a representation of $G$.

Any group $G$ has the following unitary representations. We write $\mathds{1}_G$ for the trivial representation, defined by letting $\mathcal{H}=\mathbb{C}$ and $\mathds{1}_G(\gamma)=\mathrm{Id}$ for all $\gamma\in G$. We write $\lambda_G$ for the left regular representation of $G$, defined by letting $\mathcal{H}=\ell^2(G)$ and $\lambda_G(\gamma) f(x)=f(\gamma^{-1}x)$. When $G$ is finite we always have the decomposition $\lambda_G = \mathds{1}_G + \lambda^\prime_G$ where $\mathds{1}_G$ corresponds to restriction to the span of the constant vector, and $\lambda^\prime_G$ restriction to vectors with zero average. 

Given a unitary representation on a subgroup $H\le G$, the induced representation $\mathrm{Ind}_\rho^G$ is a representation of $G$ with Hilbert space 
$$
\left\{ f: G\to \mathcal{H} : \forall h\in H \, f(gh^{-1}) = \rho(h)f(g) \right\}
$$
and action $\mathrm{Ind}_\rho^G(\gamma) f(x)= f(\gamma^{-1}x)$. In fact we will only consider the case where $H$ is a finite index normal subgroup of $G$. In this case one can choose coset representatives $\zeta_1,\ldots ,\zeta_r\in G$ of $G/H$ whence the Hilbert space is isomorphic to $\oplus_{i=1}^r \mathcal{H}$, and the action of $H$ preserves the direct sum decomposition with $\mathrm{Ind}_\rho^G(h) = \oplus_{i=1}^r \rho(\zeta^{-1}_i h \zeta_i)$. 
In the case that $\rho=\mathds{1}_H$ one has that $\mathrm{Ind}_{\mathds{1}_H}^G$ is trivial on $H$ and so factors as representation of $G/H$, which can be seen is equivalent to $\lambda_{G/H}$. (Conversely, lifting $\lambda_{G/H}$ produces  $\mathrm{Ind}_{\mathds{1}_H}^G$.)  For brevity, we write $\pi_H= \mathrm{Ind}_{\mathds{1}_H}^G$. In this way, when $G$ is finite, one also has the decomposition $\pi_H = \mathds{1}_G + \pi_H^\prime$.

The expander hypothesis for Cayley graphs can be rephrased as \emph{relative ($\tau$)}. A group $\Gamma$ has property ($\tau$) \emph{relative to the normal subgroups $\left\{ N_q: q\in\mathbb{N}\right\}$} if any unitary representation $\rho$ of $\Gamma$ that factors through some $\Gamma_q=\Gamma/N_q$ either contains the trivial representation or is ``bounded away from it" (does not weakly contain it). Equivalently, $\Gamma$ has property ($\tau$) relative to the subgroups $\left\{ N_q: q\in\mathbb{N}\right\}$ precisely when, for some finite $S\subset \Gamma$, there is $\kappa>0$ such that for any $N_q$, the unitary representation $\pi^\prime_q := \pi^\prime_{N_q}$ (given in the previous paragraph) has a non-zero Kazhdan constant $\kappa(S,\pi^\prime_{q})=\kappa>0$, given by
$$
\max_{s\in S}\|\pi_{q}(s)v_q - v_q\|\ge \kappa\|v_q\|
$$
for all non-zero vectors $v_q$ in $V_q$ the Hilbert space for $\pi^\prime_{q}$.
This property is often thought of as a spectral gap property as it implies that the random walk operator, with steps uniform in a fixed finite generating set $S$, has a uniform spectral gap in $q$ (see Lemma \ref{lemma:linfndecay}). 

The inverse limit carries a unitary representation $(\pi_\infty, V_\infty)$ of $\Gamma$ on the Hilbert space $V_\infty= L^2(\Gamma_\infty, \mathrm{Haar})$ given by $\pi_\infty(g) f(x) = f(g^{-1} x)$. We also have the decomposition $\pi_\infty = \mathds{1}_\Gamma \oplus \pi^\prime_\infty$ where $\pi^\prime_\infty$ is the restriction to $V^\prime_\infty$, the functions with zero integral $\int f dm = 0$. 
The inverse limit $\Gamma_\infty$ can be characterized by the existence of projection maps $P_q: V_{\infty}\to  V_{q}$ that are equivariant for $\pi_\infty, \pi_q$. We also have $P^*_q: V_{q}\to V_{\infty}$ with $P_qP^*_q$ equal to the identity on $V_{q}$. The representation $\pi_\infty$ contains each $\pi_q$ via $P_q^* \lambda_\infty P_q$. Moreover, the images $P^*_q V_{q}$ exhaust $V_\infty$: any element $v\in V_\infty$ is a limit of some $v_q\in P^*_q V_{q}$. In this way if $\kappa(S,\pi^\prime_q)\ge \kappa>0$ for all $q$ then also $\kappa(S,\pi^\prime_\infty)\ge \kappa>0$.

The unitary representations $\pi_q$ (and $\pi^\prime_q$) naturally appear when considering the $T_q$ correlations of functions, and it is for these that it is the most natural to impose the expander hypothesis on. However, whether these representations \textit{a priori} control the dynamics is unclear due to the problem that the monodromy groups need not be generating for all of $\Gamma$. The goal of the remainder of this section is to exhibit the unitary representations for which we are able to control the dynamics and to relate these to the expander hypothesis.
\begin{theorem}\label{thm:tau}
Let $H$ be a normal and co-amenable subgroup of $\Gamma$.
Suppose that $\Gamma$ has Property $(\tau)$ relative to $(N_q)_{q\in\mathbb{N}}$. Then there is a finite family $\mathcal{R}$ of unitary representations of $\Gamma$ such that for each $q\in\mathbb{N}\cup\left\{\infty\right\}$ we have the decomposition
$$
\pi_q = \theta_q + \rho_q
$$
for some $\theta_q\in \mathcal{R}$, and where $(\rho_q,\mathcal{H}_q)$ is a unitary representation of $\Gamma$; and such that $P_q$ equivariantly maps $(\rho_\infty,\mathcal{H}_\infty)$ to $(\rho_q,\mathcal{H}_q)$, and the collection of images $P_q^* \mathcal{H}_q$ is dense in $\mathcal{H}_\infty$. 

Moreover there is a finite $S\subset H$ and $\kappa^\prime<1$ with
$$
(\#S)^{-1}\left\|\sum_{s\in S}\rho_q(s)\right\|_{\mathcal{H}_q}<\kappa^\prime,
$$
for every $q\in\mathbb{N}\cup\left\{\infty\right\}$.
\end{theorem}

Typically one studies the Kazhdan constants $\kappa(S,\rho_q)$, for which we note one has a relation to
$$
\kappa^\prime(S^\prime,\rho_q):= (\#S)^{-1}\left\|\sum_{s\in S^\prime}\rho_q(s)\right\|_{\mathcal{H}_q},
$$ 
for $S^\prime=S\cup\left\{e\right\}$, via the parallelogram law. We denote $\kappa^\prime(S,(\rho_q)_{q\in\mathbb{N}})$ the supremum of $\kappa^\prime(S\cup\left\{ e\right\},\rho_q)$ over $q\in\mathbb{N}$. (Note the supremum is identical if we include $q=\infty$ in the supremum.)

We will use the following powerful consequence of \cite{Shalom}. (In the following one has $r(q)\le r$ if $H$ is of index $r$ in $\Gamma$.)
\begin{proposition}\label{prop:shalom}
Assume that $H$ is co-amenable in $\Gamma$ and that $\Gamma$ has Property $(\tau)$ relative to $(N_q)_{q\in\mathbb{N}}$. Then the index $r(q)$ of $HN_q/N_q$ in $\Gamma/N_q$ is bounded uniformly in $q$. If in addition $H$ is normal in $\Gamma$ then there is a finite set $F\subset H$ such that $(\mathrm{Cay}(HN_q/N_q,F))_{q\in\mathbb{N}}$ form an expander family.
\end{proposition} 

The following is well-known.
\begin{proposition}\label{prop:decomp}
Let $K\unlhd H\unlhd G$ with $[K:G],[H:G]<\infty$. Then we have the decomposition
$$
\lambda_{G/K} =\lambda_{G/H} + \mathrm{Ind}^{G/K}_{\lambda^\prime_{H/K}}.
$$
\end{proposition}
We consider $N_q \unlhd HN_q \unlhd \Gamma$. In terms of $\pi_q$ this says that 
$$
\pi_q = \mathrm{Ind}^{\Gamma}_{\mathds{1}_{N_q}} = \mathrm{Ind}^{\Gamma}_{\mathds{1}_{HN_q}} + \mathrm{Ind}^{\Gamma}_{(\mathrm{Ind}^{HN_q}_{\mathds{1}_{N_q}})^\prime} ,
$$
where $(\mathrm{Ind}^{HN_q}_{\mathds{1}_{N_q}})^\prime$ is the restriction to the orthogonal to the constant function.

\begin{proof}[Proof of Theorem \ref{thm:tau}]
For $q\in\mathbb{N}$, we apply Proposition \ref{prop:decomp} setting $\theta_q = \mathrm{Ind}^{\Gamma}_{\mathds{1}_{HN_q}} $ and $\rho_q =  \mathrm{Ind}^{\Gamma}_{(\mathrm{Ind}^{HN_q}_{\mathds{1}_{N_q}})^\prime} $. Clearly there are only finitely many $\theta_q$ since $r(q)=[\Gamma/N_q: HN_q/N_q]$ is bounded. First we check that $\rho_q$ has no non-trivial $\Gamma$ invariant vectors. Equivalently, $\mathrm{Ind}^{\Gamma_q}_{\lambda^\prime_{HN_q/N_q}}$ has no non-trivial $\Gamma_q$ invariant vectors. By Prop \ref{prop:shalom}, there is $S\subset H$ with $\kappa(SN_q,\lambda^\prime_{HN_q/N_q})\ge \kappa>0$. We have that
$\kappa(S,\rho_q) = \kappa(SN_q,\mathrm{Ind}^{\Gamma_q}_{\lambda^\prime_{HN_q/N_q}})$, and since $S\subset H$ we have for $\gamma\in S$, using the direct sum decomposition,
$$
\mathrm{Ind}^{\Gamma_q}_{\lambda_{HN_q/N_q}}(\gamma) \oplus_{i=1}^{r(q)} v_i = \oplus_{i=1}^{r(q)} \lambda_{HN_q/N_q}(\zeta_{i}^{-1} \gamma \zeta_i N_q) v_{\sigma(\gamma)(i)} .
$$
By passing to the equivalent supremum norm on coordinates, it follows that for some $\gamma\in S$ we have
$$
\|\mathrm{Ind}^{\Gamma_q}_{\lambda_{HN_q/N_q}}(\gamma) \oplus_{i=1}^{r(q)} v_i -  \oplus_{i=1}^{r(q)} v_i\| \ge \frac{\kappa}{r(q)} \|\oplus_{i=1}^{r(q)} v_i\| .
$$

The final step is to consider $q=\infty$. Denote $W_q$ the Hilbert space for $\theta_q$. Let $W_\infty$ be the union of all $P_q^* W_q$. Note that $W_\infty$ contains the constant function since each $W_q$ contains the respective constant vector. Since there are only finitely many $W_q$ it follows that $W_\infty$ is finite dimensional, and we let $\theta_\infty$ be the corresponding restriction of $\pi_\infty$. We let $\rho_\infty$ be the restriction of $\pi_\infty$ to the orthogonal complement $\mathcal{H}_\infty$ of $W_\infty$. It is easily seen that $P^*_q \mathcal{H}_q$, $q\in\mathbb{N}$, is dense in $\mathcal{H}_\infty$, and that $P_q$ equivariantly maps $(\rho_\infty, \mathcal{H}_\infty)$ to $(\rho_q, \mathcal{H}_q)$. Then $\rho_\infty$ has no non-trivial fixed $\Gamma$ vector (since this subspace is unique and in $W_\infty$), and since $P_q^*\mathcal{H}_q$ are dense it follows that $\kappa(S,\rho_\infty)\ge \kappa$ as in the previous. All that that remains is to use the parallelogram law.
\end{proof}

We conclude this section by stating an obvious result.
\begin{lemma}\label{lemma:linfndecay}
Let $H\unlhd \Gamma$ and let $(\rho_q,\mathcal{H}_q)_{q\in\mathbb{N}\cup\left\{ \infty\right\}}$ be unitary representations of $\Gamma$. 
Let $q\in\mathbb{N}\cup\left\{ \infty\right\}$. Let $S$ be a finite subset of $H$. Suppose that $\phi : (\rho_q(\Gamma),\|\cdot\|_{\mathcal{H}_q\to \mathcal{H}_q}) \to \mathbb{C}$ is a bounded linear functional with norm $C>0$ and for which $\phi(\rho_q(h))$ is real and bounded from below by $\epsilon(h)>0$ for $h\in S^n$. Then we have
$$
0<\sum_{h\in S^n}\epsilon(h) \le \phi\left((\# S)^{-n}\sum_{s_1,\ldots, s_n\in S} \rho_q(s_1\ldots s_n)\right) \le C(\kappa^\prime(S,(\rho_q)_{q\in\mathbb{N}}))^n.
$$ 
\end{lemma}


\section{The Transfer Operator, Banach spaces, and Operator Algebras}\label{section:vectorspaces}
Throughout $(X,T,m)$ is a measurable space equipped with measurable map $T$ and measure $m$. For brevity we denote $L^p(\nu,\mathbb{C})=L^\infty_\nu(X,\mathbb{C})$, where $\nu$ is any measure on $X$.
We denote $\nu_{|B}$ for the restriction of the measure $\nu$ to the measurable set $B$.
\begin{assumption}\label{assump1}
There is a collection of measurable subsets $\alpha=\left\{ [1],\ldots, [k] \right\}$ that form a partition mod $\mu$ and are a strong generator for $(X,T,\mu)$. 
We assume that for any pair $[i],[j]$ the following holds: $T([i]\cap T^{-1}[j])$ is measurable and there is a measurable map $T_{i,j}^\dag:T([i]\cap T^{-1}[j])\to [i]\cap T^{-1}[j]$ that is the inverse to the restriction $T_{i,j}$ of $T$ up to a set of measure zero. 

We assume that the pushforward of $m_{|T([i]\cap T^{-1}[j])}$ under $T_{i,j}^\dag$ is equivalent to $m_{|[i]\cap T^{-1}[j]}$ and that the Radon-Nikodym derivatives belong to $L^\infty(m,\mathbb{C})$.  

Given $i_0,\ldots, i_n$ we ask that $T^n([i_0]\cap \ldots \cap T^{-n}[i_n])$ is measurable.
\end{assumption}
We allow $[i]\cap T^{-1}[j]=\emptyset$ in the previous -- later we discuss in terms of functions and operators, whence this previous case reduces to a function or operator being the $0$ element. 
The inverse branches of $T$ play a large role in the theory. Henceforth we assume standing assumption \ref{assump1}. 
\begin{definition}
We let $r\in L^\infty(m,\mathbb{C})$ be the function on $X$ whose restriction to $[i]\cap T^{-1}[j]$ is the Radon-Nikodym derivative $dm_{|[i]\cap T^{-1}[j]}/d(T^\dag_{i,j})_*m_{|T([i]\cap T^{-1}[j])}$.
We let $r(i_0,\ldots, i_n)$ be the function supported on 
$$
T^n([i_0]\cap T^{-1}[i_1])\cap \cdots \cap T([i_{n-1}]\cap T^{-1}[i_n])
$$ 
taking values
$$
r(i_0,\ldots, i_n)(x) =\prod_{k=1}^n  r(T^\dag_{i_{k-1},i_k} \circ \cdots \circ T^\dag_{i_{n-1},i_n}x).
$$
\end{definition}
The following integration follow from the Radon-Nikodym derivate definition and are more easily verified using a duality relation between Koopman and transfer operators (see Section \ref{section:duality}).
\begin{lemma}\label{cor:measure1}
Given $i_0,\ldots, i_n$ and $B=[i_0]\cap T^{-1}[i_1]\cap \cdots T^{-n}[i_n]$ we have that for each $k=0,\ldots, n$,
$T^\dag_{i_{k},i_{k+1}} \circ \cdots \circ T^\dag_{i_{n-1},i_n}$ is well-defined on $T^n(B)$ and for $k=0$ the image is $B$ up to a set of measure zero. We have
$$
\int_{T^n B} r(i_{0},\ldots, i_n) \,dm = \int_{B} \,dm,
$$
and
$$
\int_{B} 1/r(i_{0},\ldots, i_n)\circ T^n \,d\mu = \int_{T^nB}\, h\circ  (T^\dag_{i_{n-1},i_n}\circ \cdots T^\dag_{i_{0},i_1}) \,dm .
$$
\end{lemma}

We let $\mathcal{A}\subset L^\infty_\mu(X,\mathbb{C})$ be the linear span of $\chi(i_0,\ldots, i_n)$, the indicator functions on the sets $[i_0]\cap T^{-1} [i_1] \cap \cdots \cap T^{-n} [i_{n}]$. For brevity we denote its closure in $L^\infty(m,\mathbb{C})$ by $\overline{\mathcal{A}}_{L^\infty}$. We also make reference to the collection of $\chi^\dag(i_0,\ldots, i_n)$ given by the indicator function on $T^n([i_0]\cap T^{-1}[i_1]\cap \ldots \cap T^{-n}[i_n])$.
\begin{assumption}\label{assump3} 
There are vector subspaces $\mathcal{F}_s\subseteq \mathcal{F}_w\subseteq L^p(m,\mathbb{C})$, with $p\in[1,\infty]$, and norms $\|\cdot \|_s,\|\cdot \|_w$ satisfying the following. 
\begin{enumerate}
\item $\|\cdot\|_w = \|\cdot\|_{L^p(m,\mathbb{C})}$ and $\mathcal{B}_w=(\mathcal{F}_w,\|\cdot \|_w)$ is complete and separable;
\item $\mathcal{A}\subseteq \mathcal{B}_w$;
\item $\mathcal{B}_s=(\mathcal{F}_s,\|\cdot \|_s)$ is complete and separable;
\item $\mathcal{F}_s$ is dense in $(\overline{\mathcal{A}}_{L^\infty},L^\infty(m,\mathbb{C}))$.
\item $\mathcal{F}_s$ is dense in $\mathcal{B}_w$.
\item\label{item:compact} The unit ball in $\mathcal{B}_s$ is precompact in $\mathcal{B}_w$.
\item Each $\chi^\dag(i_0,\ldots, i_n)$ is an element of $\mathcal{F}_s$, and any $f\in \mathcal{F}_s$ with $f$ bounded away from $0$ has $f^{-1}\in\mathcal{F}_s$;
\item For each $i,j=1,\ldots, k$, the operator $f\mapsto f\circ T^\dag_{i,j} \chi^\dag(i,j)$ preserves and is bounded on $\mathcal{B}_s$,
\item\label{item:mult2} there is $D_{s\cdot s}$ such that for each $g,f\in \mathcal{F}_s$ we have that $gf\in\mathcal{F}_s$ and with norm $\|fg\|_s\le D_{s\cdot s}\|f\|_{s} \|g\|_s$;
\item\label{item:upper2} we have the inequality $\|\cdot \|_{w}\le \|\cdot \|_s$ on $\mathcal{F}_s$.
\item\label{item:upper3} we have $\|f\|_{L^\infty(m,\mathbb{C})}\le \|f\|_s$.
\end{enumerate}
\end{assumption}

We refer to $\mathcal{B}_s$ as the \emph{strong space} and $\mathcal{B}_w$ as the \emph{weak space}. 
It follows immediately that $\mathcal{A}$ is dense in $\mathcal{B}_w$ and that for any $g\in \overline{\mathcal{A}}_{L^\infty}$ we have $gf\in \mathcal{B}_w$ for any $f\in \mathcal{B}_w$, and with norm $\|gf\|_{w}\le \|g\|_{L^\infty(m,\mathbb{C})}\|f\|_{w}$.

The Koopman operator $\mathbf{T}$ acts on functions by $f\mapsto f\circ T$. It follows from Assumption \ref{assump1} that $\mathbf{T}$ is bounded on $L^q(m,\mathbb{C})$ for any $q\in[1,\infty]$. 
The transfer operator $\mathbf{L}$ is implicitly defined by
$$
\int \mathbf{L}(f) \, g \, dm = \int f \, \mathbf{T}(g) \, dm,
$$
for $f\in L^p(m,\mathbb{C})$, and $g\in L^{q}(m,\mathbb{C})$ with $1/p + 1/q = 1$. Boundedness for $\mathcal{B}_w$ follows by duality, whereas boundedness for $\mathcal{B}_s$ follows by Assumption \ref{assump3}. 

A $T$-invariant probability measure is said to be an \textit{absolutely continuous invariant probability measure (ACIP)} if there is $h\in L^q(m,\mathbb{C})$ such that $\mu = h\,dm$ and $\mathbf{L}h=h$. We refer to $h$ as the density of the ACIP (which is generally defined up to measure zero but unique once understood in a particular class of functions).
\begin{assumption}\label{assump2}
We assume that there is an ACIP $\mu= h\, dm$ with $h\in \mathcal{F}_s$ and such that $h$ vanishes only on a set of zero $m$-measure. In the case $p=\infty$ we assume that there is a constant $D_{\mathrm{Gibbs}}>0$ such that
$$
\|r(i_0,i_1,\ldots, i_{n-1})\|_{w}\le D_{\mathrm{Gibbs}}\mu ([i_0]\cap T^{-1}[i_1]\cap \cdots T^{-n}[i_n]),
$$
and moreover we assume that $h=1$.
\end{assumption}
Note that one also has that the operators $\mathbf{T}^m, \mathbf{L}^n: \mathcal{B}_w \to L^p(\mu,\mathbb{C})$ are bounded uniformly in $m,n\in\mathbb{N}$ and with norm $1$.

Existence (and uniqueness) for the ACIP are given by a transfer operator approach. Two important ingredients of such an approach are the precompactness result of $\mathcal{B}_s$ in $\mathcal{B}_w$ (Assumption \ref{assump3}:\ref{item:compact}), and the use of a \emph{Doeblin--Fortet inequality:} an inequality of the form 
$$
\|\mathbf{L}^n f\|_s \le C_{DF}\alpha^n_{DF}\|f\|_w + C_{DF}\|f\|_w,
$$
for constants $C_{DF}$ and $\alpha_{DF}$ independent of $n\in\mathbb{N}$ and $f\in\mathcal{B}_s$.
(It should be noted that different names are given to this based on the history and context in which one works \cite{DF}, \cite{ITM}, \cite{LSY}.) We will use this notion for the extensions but don't require it here since we assume properties of the ACIP instead.

\begin{assumption}[Base]\label{assump4}
We say that $(X,T,\alpha,\mu,m)$, and $\mathcal{B}_s$, $\mathcal{B}_w\subseteq L^p(m,\mathbb{C})$ satisfies \textnormal{(Base)} if they fufill Assumptions \ref{assump1}, \ref{assump3}, \ref{assump2}, and if the ACIP $\mu$ is ergodic.
\end{assumption}
For shorthand we typically say that $(X,T,\alpha,\mu,m)$ and $\mathcal{B}_s$ satisfy \textnormal{(Base)} with $p\in[1,\infty]$ since $\mathcal{B}_w$ is necessarily the closure of $\mathcal{B}_s$ in $L^p(m,\mathbb{C})$.


\subsection{Operator families for the base dynamics}\label{subsection:ops}
We write $\mathrm{Op}(\mathcal{B},\mathcal{B})$ for the bounded linear operators from a Banach space $(\mathcal{B},\|\cdot\|_{\mathcal{B}})$ to itself. Then $\mathrm{Op}(\mathcal{B},\mathcal{B})$ is a Banach space when equipped with the norm
$$
\|b\|_{\mathcal{B}\to \mathcal{B}}:= \sup_{f\in\mathcal{B}, f\ne 0} \|bf\|_{\mathcal{B}}/\|f\|_{\mathcal{B}},
$$
and moreover is an algebra. If $\mathcal{B}^\prime$ is another Banach space that is also a sub-vector space of $\mathcal{B}$ then we may restrict any $b\in \mathrm{Op}(\mathcal{B})$ to $b\in\mathrm{Op}(\mathcal{B}^\prime)$ provided $b$ is bounded for the norm on $\mathcal{B}^\prime$. 
We write $\mathrm{Op}(\mathcal{B}_1, \mathcal{B}_2)$ for the bounded linear operators from the Banach spaces $(\mathcal{B}_1,\|\cdot\|_{1})$ to the Banach space $(\mathcal{B}_2,\|\cdot\|_{2})$, which is itself a Banach space when equipped with the norm
$$
\|b\|_{1\to 2}:= \sup_{f\in\mathcal{B}_1, f\ne 0} \|bf\|_{2}/\|f\|_{1}.
$$
Similarly we can restrict the domain to $\mathcal{B}^\prime_1$ or include the range to $\mathcal{B}^\prime_2$ provided the requisite bounds for the norms are satisfied. Inclusion maps are faithful by definition, and the restriction maps we consider later will also be: if the vector space $\mathcal{B}^\prime_1$ is dense in $\mathcal{B}_1$ then any collection of operators $\mathcal{V}\subseteq \mathrm{Op}(\mathcal{B}_1,\mathcal{B}_2)$ has that the restriction map to $\mathrm{Op}(\mathcal{B}^\prime_1,\mathcal{B}_2)$ is faithful.

The multiplication operators $\mathbf{M}(g)$ have pointwise formula $f(x)\mapsto g(x)f(x)$ on any function space. The restrictions and inclusions achieved in the previous paragraph are achieved by restricting the functions $f$ correspondingly.
Denote $\mathrm{Mult}(\overline{\mathcal{A}}_{L^\infty})\subseteq \mathrm{Op}(\mathcal{B}_w,\mathcal{B}_w)$ the collection of all $\mathbf{M}(g)$ with $g\in \overline{\mathcal{A}}_{L^\infty}$. Then $\mathrm{Mult}(\overline{\mathcal{A}}_{L^\infty})$ is an algebra and one has 
$$
\|\mathbf{M}(g)\|_{L^p(m,\mathbb{C})\to L^p(m,\mathbb{C})}= \|g\|_{L^\infty(m,\mathbb{C})}.
$$
(Recall we assume that the norm on $\mathcal{B}_w$ is $\|\cdot\|_{L^p(m,\mathbb{C})}$.)
We also have the restrictions $\mathrm{Mult}(\overline{\mathcal{A}}_{L^\infty})\subseteq \mathrm{Op}(\mathcal{B}_s,\mathcal{B}_w)$ and $\mathrm{Mult}(\overline{\mathcal{A}}_{L^\infty})\subseteq \mathrm{Op}(\mathcal{B}_s,\mathcal{B}_s)$. 
Denote $\mathrm{Mult}(\mathcal{F}_s)\subseteq \mathrm{Op}(\mathcal{B}_s,\mathcal{B}_s)$ the collection of all $\mathbf{M}(g)$ with $g\in \mathcal{F}_s$. Then $\mathrm{Mult}(\mathcal{F}_s)$ is an algebra (by Assumption \ref{assump2}) and one has
$$
\|g\|_s \le \|\mathbf{M}(g)\|_{s\to s}\le D_{s\cdot s}\|g\|_s.
$$ 
By Assumption \ref{assump3} one has the inclusions $\mathrm{Mult}(\mathcal{F}_s)\subseteq \mathrm{Op}(\mathcal{B}_s,\mathcal{B}_w)$, $\mathrm{Mult}(\mathcal{F}_s)\subseteq \mathrm{Op}(\mathcal{B}_w,\mathcal{B}_s)$.

We define linear maps $\mathbf{T}(i,j)\in \mathrm{Op}(\mathcal{B}_w,\mathcal{B}_w), \mathbf{T}^\dag(i,j)\in \mathrm{Op}(\mathcal{B}_w,\mathcal{B}_w)$ by
$$
\mathbf{T}(i,j) f = (f\chi(i,j))\circ T, \; \mathbf{T}^\dag(i,j) f = f\circ T^\dag_{i,j} \chi^\dag(i,j).
$$
That $\mathbf{T}(i,j)$ is bounded follows from boundedness for $\mathbf{T}$, and that $\mathbf{T}^\dag(i,j)$ is bounded follows from the  duality relation (Lemma \ref{lemma:duality}).
By Assumption \ref{assump3} the restriction $\mathbf{T}^\dag(i,j)\in\mathrm{Op}(\mathcal{B}_s,\mathcal{B}_s)$ is well-defined. Also one has the inclusion $\mathbf{T}^\dag(i,j)\in \mathrm{Op}(\mathcal{B}_s,\mathcal{B}_w)$. (As for the multiplication operators, the restrictions and inclusions achieved in the previous paragraph are achieved by restricting the functions $f$ correspondingly.)

It can be checked (see Sections \ref{section:duality} and \ref{section:algebra}) that we have the formulas
$$
\mathbf{L} = \sum_{i,j=1}^k \mathbf{T}_{i,j}^\dag\mathbf{M}(r), \,
\mathbf{T} = \sum_{i,j=1}^k \mathbf{T}_{i,j}.
$$
This tells that the transfer operator is composed of certain fundamental pieces: Koopman operators for dynamics, and multiplication operators. Our study of transfer operators is intimately connected to the families of operators constructed from the previous objects. Given $S\subset \mathrm{Op}(\mathcal{B},\mathcal{B})$, the algebra generated by $S$ is the smallest sub-algebra of $\mathrm{Op}(\mathcal{B})$ containing $S$. Given $S\subset \mathrm{Op}(\mathcal{B}_1\to \mathcal{B}_2)$ the span of $S$ is the smallest sub-vector space containing $S$. (In each case the smallest such object indeed exists.)
\begin{definition}

Denote
$$
\left\langle \mathrm{Mult}(\mathcal{F}_s), \mathbf{L}\right\rangle_{s,s}
$$ 
for the sub-algebra of $\mathrm{Op}(\mathcal{B}_s,\mathcal{B}_s)$ generated by $\mathrm{Mult}(\mathcal{F}_s)$ and $\mathbf{L}$. 
We also denote $\left\langle \mathrm{Mult}(\mathcal{F}_s), \left\{\mathbf{T}^\dag(i, j): i,j=1,\ldots, k\right\}\right\rangle_{s,w}$ for its inclusion to $\mathrm{Op}(\mathcal{B}_s,\mathcal{B}_w)$.
Denote 
$$
\left\langle \mathrm{Mult}(\overline{\mathcal{A}}_{L^\infty}), \mathbf{L} \right\rangle_{w,w}
$$ 
for the sub-algebra of  $\mathrm{Op}(\mathcal{B}_w,\mathcal{B}_w)$ generated by $\mathrm{Mult}(\overline{\mathcal{A}}_{L^\infty})$ and $\mathbf{L}$, and denote $\langle \mathrm{Mult}(\overline{\mathcal{A}}_{L^\infty}), \mathbf{L} \rangle_{s,w}$ the restriction to  $\mathrm{Op}(\mathcal{B}_s,\mathcal{B}_w)$ .
\end{definition}
In Section \ref{section:algebra} we check that $\langle \mathrm{Mult}(\overline{\mathcal{A}}_{L^\infty}), \mathbf{L} \rangle_{s,w}$ is contained in $\overline{\left\langle \mathrm{Mult}(\mathcal{F}_s), \mathbf{L}\right\rangle_{s,w}}$, the closure of $\left\langle \mathrm{Mult}(\mathcal{F}_s), \mathbf{L}\right\rangle_{s,w}$ in $\mathrm{Op}(\mathcal{B}_s,\mathcal{B}_w)$. 

A particular role is played by the following elements, which are all observed to belong in particular to $\langle \mathrm{Mult}(\overline{\mathcal{A}}_{L^\infty}), \mathbf{L}\rangle_{s,w}$ -- with the exception of the compositions of $\mathbf{T}(j_0,j_1)$ which belong to $\langle \mathrm{Mult}(\overline{\mathcal{A}}_{L^\infty}), \mathbf{L},\mathbf{T} \rangle_{w,w}$.  
\begin{definition}\label{def:niceops}
Denote
$$
\bm{\chi}^\dag(j_0,\ldots, j_m) := \mathbf{M}(\chi^\dag(j_0,\ldots, j_m)),\,\bm{\chi}(j_0,\ldots, j_m) := \mathbf{M}(\chi(j_0,\ldots, j_m) ),
$$
$$
\mathbf{T}^\dag(j_0,\ldots, j_m) = \mathbf{T}^\dag(j_{m-1},j_m) \cdots \mathbf{T}^\dag(j_0,j_1),
$$
$$
\mathbf{T}(j_0,\ldots, j_m) = \mathbf{T}(j_0,j_1)\cdots \mathbf{T}(j_{m-1},j_m),
$$
and
\begin{align*}
\mathbf{d}^\dag(i_0,\ldots, i_{k-1};j_k,\ldots, &j_n) =
\bm{\chi}^\dag(i_0,\ldots,i_{k-1})\mathbf{T}^\dag(i_0,\ldots,i_{k-1}) 
\\
&-\bm{\chi}^\dag(i_0,\ldots,i_{k-1}j_k,\ldots, j_n)\mathbf{T}^\dag(i_0,\ldots,i_{k-1}j_k,\ldots, j_n).
\end{align*}
\end{definition}


\subsection{Banach spaces and operators for the group extension}\label{subsection:gpext}
The goal of this section is extend the $\mathbb{C}$-valued function spaces $\mathcal{F}_s,\mathcal{F}_w$ to taking values in the Hilbert spaces $\mathcal{H}_q$, $q\in\mathbb{N}\cup\left\{\infty\right\}$. This is most obviously done by taking tensor products and tensor product norms. Since the transfer operator machinery is more adapted to ``pointwise evaluation" type arguments we use the vector valued functions viewpoint in addition to the tensor product viewpoint. For finite dimensional Hilbert spaces (such as $\mathcal{H}_q$, $q<\infty$) one easily (and carelessly) moves between these two perspectives, and moreover the picture is similarly nice for separable Hilbert spaces (such as $\mathcal{H}_\infty$).

Let $q\in\mathbb{N}\cup\left\{ \infty\right\}$. For any measure $\nu$ the space $L^p(\nu,\mathcal{H}_q)$ consist of (equivalence classes of) functions $F:X\to \mathcal{H}_q$ whose coordinate functions (taken in an orthonormal basis) are measurable and such that, in the case $p<\infty$,
$$
\|F\|_{L^p(\nu,\mathcal{H}_q)}:=\left(\int \|F(x)\|^p_{\mathcal{H}_q}\,d\mu \right)^{1/p} <\infty,
$$ 
and in the case $p=\infty$, $\|F\|_{L^\infty(\nu,\mathcal{H}_q)}$ is the essential supremum of $x\mapsto \|F(x)\|_{\mathcal{H}_q}$, which is asked to be finite. 
We make the convention that the inner product in our Hilbert spaces $\mathcal{H}_q$ is denoted ``$\cdot$" and extend this notation to the pairing of $\mathcal{H}_q$-valued functions as follows. Given $v\in\mathcal{H}_q$ one can define the projection (a $\mathbb{C}$-valued function) $F\cdot v$ by $(F\cdot v)(x) = F(x)\cdot v$. More generally, given $F,G$ we denote $F\cdot G$ for the function $x\mapsto F(x)\cdot G(x)$. When the dimension of $\mathcal{H}_q$ is finite, any $F\in L^p(\nu,\mathcal{H}_q)$ is a finite linear combination of simple elements $fv$ for $f\in L^p(\nu,\mathbb{C})$ and $v\in\mathcal{H}_q$ chosen in an orthonormal basis. The notation $fv$ may be viewed as shorthand for the function $x\mapsto f(x) v$ or as shorthand for the isomorphism from the tensor product, by mapping $f\otimes v$ to this function. For $q=\infty$, such functions are dense if $p\in [1,\infty)$. Let $\mathcal{A}_q$ be those functions $F = fv$ for $f\in\overline{\mathcal{A}}_{L^\infty}$ and $v\in\mathcal{H}_q$ with $\|v\|_{\mathcal{H}_q}=1$. It can be observed that $\mathcal{A}_q$ is dense in $L^p(m,\mathcal{H}_q)$ for $1\le p<\infty$ (and otherwise is dense in its closure in $L^\infty(m,\mathcal{H}_q)$).

We define $\mathcal{F}_{s,q}\subseteq L^p(m,\mathcal{H}_q)$ as the linear span of simple elements $fv$, $f\in \mathcal{F}_{s}$, $v\in\mathcal{H}_q$. This is isomorphic to the tensor product of $\mathcal{F}_{s}$ and $\mathcal{H}_q$ when $q<\infty$. We want to come up with a natural norm for the space. In our applications we encounter the following two examples.
\begin{definition}
Let $\bullet=s,w$ and $q\in\mathbb{N}\cup\left\{ \infty\right\}$. The \emph{injective norm} is given by
$$
\|F\|_{\epsilon, s,q} := \sup_{v\in\mathcal{H}_q: \|v\|_{\mathcal{H}_q}=1} \|F\cdot v\|_s <\infty.
$$
The \emph{projective norm} is given by
$$
\|F\|_{\pi,s,q} := \inf\left\{ \sum_{i=1}^m \|f_i\|_s \|v_i\|_{\mathcal{H}_q}: F=\sum_{i=1}^m f_i v_i \right\}.
$$
\end{definition}
It is not hard to see that $\mathcal{B}_{\bullet,q}=(\mathcal{F}_{\bullet,q},\|\cdot\|_{\bullet,q})$ is complete for each of $\|\cdot \|_{\bullet,q}=\|\cdot \|_{\epsilon, \bullet,q}, \|\cdot \|_{\pi, \bullet,q}$ and $q<\infty$.
Given $\mathcal{B}_{\bullet,q}$ as previously, and given an operator $b\in\mathrm{Op}(\mathcal{B}_\bullet)$ we denote $b\otimes I\in\mathrm{Op}(\mathcal{B}_{\bullet,q})$ for, equivalently, the tensor product operator, and operator with pointwise formula $b(fv) = b(f)v$ on simple elements. We have $\|b\otimes I\|_{\bullet,q\to \bullet,q} = \|b\|_{\bullet\to \bullet}$ in the general framework of reasonable cross-norms. (One says that $\|\cdot\|_{\bullet,q}$ is a reasonable cross norm if the norm of a simple element is the product of the corresponding norm and if the norm form any simple linear functional is the product of the corresponding dual norms.) Any reasonable cross norm $\|\cdot\|_{s,q}$ has $\|\cdot \|_{\epsilon, s,q}\le \|\cdot \|_{s,q}\le \|\cdot \|_{\pi, s,q}$. 

Recall that the Hilbert spaces have a tower property given by $P_q$. We define 
$$\mathbf{P}_q:L^p (m,L^2(\Gamma_\infty, \mathrm{Haar}))\to L^p (m,\mathbb{C}^{\Gamma_q})
$$ and 
$$
\mathbf{P}^*_q:L^p (m,\mathbb{C}^{\Gamma_q})\to L^p (m,L^2(\Gamma_\infty, \mathrm{Haar}))
$$
 via
$$
(\mathbf{P}_qF)(x) = P_q(F(x)),\; (\mathbf{P}^*_qG)(x) = P_q^*(G(x)).
$$
Then also $\mathbf{P}_q:\mathcal{F}_{w,\infty}\to \mathcal{F}_{w,q}$ and $\mathbf{P}^*_q:\mathcal{F}_{w,q}\to \mathcal{F}_{w,\infty}$ and we may realise $\mathbf{P}_q$ as $I\otimes P_q$, and $\mathbf{P}^*_q$ as $I\otimes P^*_q$.
We observe that the restrictions to $\mathcal{F}_{s,\infty}$ and $\mathcal{F}_{s,q}$ respectively are well-defined.

\begin{definition}[Tower norms]\label{def:towernorms}
We say that a family of reasonable cross-norms $\|\cdot\|_{\bullet,q}$, $q\in\mathbb{N}$, $\bullet=s,w$ are \emph{tower norms} if for each $q\in\mathbb{N}$, the map $\mathbf{P}_{q+1}\mathbf{P}_q^*$ is an isometric embedding from $(\mathcal{F}_{s,q},\|\cdot\|_{s,q})$ into $(\mathcal{F}_{s,q+1},\|\cdot\|_{s,q+1})$

We define $\|\cdot\|_{\bullet, \infty}$ by limit of pushforwards.

The \emph{strong spaces} are the normed vector spaces $\mathcal{B}_{s,q}=(\mathcal{F}_{s,q},\|\cdot\|_{s,q})$, $q<\infty$, and $\mathcal{B}_{s,\infty}$ the closure of $\mathcal{F}_{s,\infty}$ in $\|\cdot\|_{s,\infty}$.
\end{definition}
It can be checked that each of the injective (respectively, projective) norms are tower norms and that the limit construction $\|\cdot\|_{\bullet,\infty}$ coincides with the $q=\infty$ injective (respectively, projective) norm.

\begin{proposition}\label{prop:analogue}
For any $q\in\mathbb{N}\cup\left\{ \infty\right\}$, we have the inequalities
$$
\|\cdot\|_{w,q},\|\cdot\|_{L^\infty(m,\mathcal{H}_q)}\le \|\cdot\|_{s,q}.
$$
The inclusion of $\mathcal{B}_{s,q}$ is dense in $\mathcal{B}_{w,q}$. The union of inclusions $\mathbf{P}_q^*\mathcal{B}_{s,q}$ is dense in $\mathcal{B}_{w,\infty}$. For $q<\infty$, the unit ball in $\mathcal{B}_{s,q}$ is pre-compact in $\mathcal{B}_{w,q}$.
\end{proposition}
\begin{proof}
We only consider the final statement regarding pre-compactness since the others follow by considering each tensor factor separately.
Recalling that we assume $q<\infty$, any element of $\mathcal{B}_{s,q}$ may be written as $\sum_{i=1}^k f_i v_i$ with $v_i$ an orthonormal basis and $k= \dim \mathcal{H}_q$. Suppose that $F=\sum_{i=1}^k f_i v_i$ belongs to the unit ball of $\mathcal{B}_{s,q}$. Then also $\|F\|_{\epsilon,s,q}\le 1$, so for every unit $v$, we have that $\sum_{i=1}^k f_i v_i\cdot v$ belongs to the unit ball of $\mathcal{F}_s$. In particular for $v=v_i$ we conclude that $f_i$ belongs to the unit ball of $\mathcal{F}_s$. Now let $g_i\in \mathcal{F}_w$ with $\|f_i-g_i\|_w\le \epsilon/k$.
Now
\begin{align*}
\left\|\sum_{i=1}^k g_i v_i - f_iv_i\right\|_{w,q} &\le \left\|\sum_{i=1}^k g_i v_i - f_iv_i\right\|_{\pi ,w,q}
\\
&\le 
\sum_{i=1}^k \left\|(f_i-g_i)\right\|_w \le \epsilon .
\end{align*}
Supposing instead that we have a sequence $f_{i,n}$, we may take a common subsequence along which $f_{i,n}$ converge in $\|\cdot\|_w$ as $n\to\infty$ to $g_i$. We conclude that $\sum_{i=1}^k g_iv_i$ is an accumulation point of 
$\sum_{i=1}^k f_{i,n} v_{i,n}$.
\end{proof}

As we have already remarked, there are isometric inclusions of the operators in $R_\bullet(T)$, $\bullet=s,w$ into $\mathrm{Op}(\mathcal{B}_{\bullet,q})$ given by $b\mapsto b\otimes I$ (or equivalently by pointwise formulas on corresponding subspaces). Similarly, there is an isometric inclusion of $\rho_q(\Gamma)$, the linear span of $\rho_q(\gamma)$, $\gamma\in\Gamma$ into each space. For these inclusions the images of operators in $\rho_q(\Gamma)$ and $R_\bullet(T)$ commute.

\begin{definition}[Twisted transfer operator and twisted Koopman operator]
We write $\psi(i,j)$ for the constant value of $\psi$ on $x\in [i]\cap T^{-1} [j]$, and
$$
\bm{\rho}^*_q(i,j) = \rho_q(\psi(i,j)^{-1}),\,\, \bm{\rho}_q(i,j)=  \rho_q(\psi(i,j))
$$
$$
\mathbf{L}_q = \sum_{i,j=1}^k \mathbf{T}^\dag(i,j) \mathbf{M}(r) \otimes \bm{\rho}^*_q(i,j),
\;\; \mathbf{T}_q = \sum_{i,j} \mathbf{T}(i,j)\otimes \bm{\rho}_q(i,j).
$$
We call $\mathbf{L}_q$ the twisted transfer operator, and we call $\mathbf{T}_q$ the twisted Koopman operator. 
\end{definition}
We have that $\mathbf{T}_q$ is isomorphic to the Koopman operator for $T_q$ on a suitable subspace of integral zero functions (see Lemma \ref{lemma:isofunctions}).
The duality relation (which we prove in Lemma \ref{lemma:duality}) between $\mathbf{T}_q$ and $\mathbf{L}_q$ is expressed in the following:
\begin{equation}\label{eq:duality}
\int \mathbf{L}_q(F)\cdot G \,d m = \int F \cdot \mathbf{T}_q(G) \,d m.
\end{equation}

Our final assumption concerns the ability to lift the transfer operator machinery from the base to the extension. (For the applications we have in mind we assert that it is possible to deduce the assumption ($\mathcal{H}_q$-Ext) as a consequence of a stronger form of the data given in (Base).) Recall the definition of $\mathbf{d}^\dag(i_0,\ldots, i_{k-1};j_k,\ldots, j_n)$ in Definition \ref{def:niceops}. 

\begin{assumption}[$\mathcal{H}_q$-Ext]\label{assump5} We say that norms $\|\cdot\|_{\bullet,q}$, on $\mathcal{F}_{\bullet,q}$, for $\bullet=s,w$, $q\in\mathbb{N}$ satisfy \textnormal{($\mathcal{H}_q$-Ext)} if they are tower norms (Definition \ref{def:towernorms}) and additionally:
\begin{enumerate}
\item\label{5item2} \textnormal{($L^p(m,\mathcal{H}_q)$ weak norm:)} For each $q\in\mathbb{N}$ we have $\|\cdot\|_{w,q}=\|\cdot\|_{L^p(m,\mathcal{H}_q)}$. 
\item\label{5item1}  \textnormal{(Uniform Doeblin--Fortet inequality:)}
There are constants $C_{DF}>0$ and $\alpha_{DF}<1$ such that for any $q\in\mathbb{N}$ and any $F\in\mathcal{F}_{s,q}$ we have
$$
\|\mathbf{L}_q^n (F) \|_{s,q} \le C_{DF}\alpha_{DF}^n \|F \|_{s,q} + C_{DF}\| F \|_{w,q}.
$$
\item\label{5item4}  \textnormal{(Closeness criterion:)}
Given $F\in\mathcal{B}_{s,q}$, the semi-norm
$$
\|b\|^{(F,q)}= \frac{\| bF\|_{L^p(\mu,\mathcal{H}_q)}}{\|F\|_{s,q}},
$$
satisfies the following. Given $(\Sigma^{\prime}_k)_{k\in\mathbb{N}}$, there are $i_0,\ldots, i_{k-1}\in \Sigma^{\prime}_k$ such that for any $j_k,\ldots, j_n$ we have that for any $g\in \overline{\mathcal{A}}_{L^\infty}$,
\begin{align*}
&\|\mathbf{M}(g)\, \mathbf{d}^\dag(i_0,\ldots, i_{k-1};j_k,\ldots, j_n)\otimes I\|^{(F,q)}
\\
&\le\|g\chi^\dag(i_0,\ldots, i_{k-1})\chi^\dag(i_0,\ldots, i_{k-1}j_k,\ldots, j_n)\|_{L^p(\mu,\mathbb{C})}(\#\Sigma^{\prime}_k)^{-1}
\\
&+ \|g(\chi^\dag(i_0,\ldots, i_{k-1})-\chi^\dag(i_0,\ldots, i_{k-1}j_k,\ldots, j_n))\|_{L^p(\mu,\mathbb{C})}.
\end{align*}
\end{enumerate}
\end{assumption}

Henceforth we assume that Assumption \ref{assump5} holds in addition to the previous assumptions. Recall that $\|\cdot\|_{\bullet,\infty}$ is defined as the limit of the pushforwards of $\|\cdot\|_{\bullet,q}$ for $q\in\mathbb{N}$. Using Prop \ref{prop:analogue} we can upgrade the norm estimates in \textnormal{($\mathcal{H}_q$-Ext)}  to $q=\infty$.
\begin{corollary}
For the same $C_{DF},\alpha_{DF}$ as in \textnormal{($\mathcal{H}_q$-Ext)} we have 
$$
\|\mathbf{L}_\infty^n (F) \|_{s,\infty} \le C_{DF}\alpha_{DF}^n \|F \|_{s,\infty} + C_{DF}\| F \|_{w,\infty},
$$
for any $F\in\mathcal{B}_{s,\infty}$ and the semi-norm,
$$
\|b\|^{(F,\infty)}= \frac{\| bF\|_{L^p(\mu,\mathcal{H}_\infty)}}{\|F\|_{s,\infty}},
$$
satisfies the following: given $(\Sigma^{\prime}_k)_{k\in\mathbb{N}}$, there are $i_0,\ldots, i_{k-1}\in \Sigma^{\prime}_k$ such that for any $j_k,\ldots, j_n$ and any
 $g\in \overline{\mathcal{A}}_{L^\infty}$,
\begin{align*}
&\|\mathbf{M}(g)\, \mathbf{d}^\dag(i_0,\ldots, i_{k-1};j_k,\ldots, j_n)\otimes I\|^{(F,\infty)}\|h\|_s^{-1}
\\
&\le\|g\chi^\dag(i_0,\ldots, i_{k-1})\chi^\dag(i_0,\ldots, i_{k-1}j_k,\ldots, j_n)\|_{L^p(\mu,\mathbb{C})}(\#\Sigma^{\prime}_k)^{-1}
\\
&+ \|g(\chi^\dag(i_0,\ldots, i_{k-1})-\chi^\dag(i_0,\ldots, i_{k-1}j_k,\ldots, j_n))\|_{L^p(\mu,\mathbb{C})}.
\end{align*}
\end{corollary}

\begin{definition}\label{def:operatoralg2}
Let $q\in\mathbb{N}\cup\left\{\infty\right\}$. Denote
$$
\left\langle \mathrm{Mult}(\mathcal{F}_s)\otimes\rho_q(\Gamma), \mathbf{L}_q\right\rangle_{s,q,s,q}
$$ 
for the sub-algebra of $\mathrm{Op}(\mathcal{B}_{s,q},\mathcal{B}_{s,q})$ generated by $ \mathrm{Mult}(\mathcal{F}_s)\otimes\rho_q(\Gamma)$ and $\mathbf{L}_q$. 
We also denote $\left\langle \mathrm{Mult}(\mathcal{F}_s)\otimes\rho_q(\Gamma), \mathbf{L}_q\right\rangle_{s,q,w,q}$ for its inclusion to $\mathrm{Op}(\mathcal{B}_{s,q},\mathcal{B}_{w,q})$.
Denote 
$$
\left\langle \mathrm{Mult}(\overline{\mathcal{A}}_{L^\infty})\otimes\rho_q(\Gamma), \mathbf{L}_q \right\rangle_{w,q,w,q}
$$ 
for the sub-algebra of  $\mathrm{Op}(\mathcal{B}_{w,q},\mathcal{B}_{w,q})$ generated by $\mathrm{Mult}(\overline{\mathcal{A}}_{L^\infty})\otimes\rho_q(\Gamma)$ and $\mathbf{L}_q$, and denote $\langle \mathrm{Mult}(\overline{\mathcal{A}}_{L^\infty})\otimes\rho_q(\Gamma), \mathbf{L}_q \rangle_{s,q,w,q}$ the restriction to  $\mathrm{Op}(\mathcal{B}_{s,q},\mathcal{B}_{w,q})$ .
\end{definition}
In Section \ref{section:algebra} we check that the closure of $\left\langle \mathrm{Mult}(\mathcal{F}_s)\otimes\rho_q(\Gamma), \mathbf{L}_q\right\rangle_{s,q,s,q}$ in $\mathrm{Op}(\mathcal{B}_{s,q},\mathcal{B}_{w,q})$, contains $\left\langle \mathrm{Mult}(\overline{\mathcal{A}}_{L^\infty})\otimes\rho_q(\Gamma), \mathbf{L}_q \right\rangle_{s,q,w,q}$. And moreover that the closures admit a description as a linear span in terms of $\mathbf{T}^\dag(i_0,\ldots, i_n)$.


\section{Expanders and KMS states}\label{section:KMS}
We begin with a brief discussion of KMS states for Cuntz-Krieger algebras because we find synergies with this concept. (Indeed, those in the dynamical world have been interested in the concept for some time; see for instance \cite{kesse}.)
Given a subshift of finite type $X=\Sigma_A^+$ defined by $k\times k$ transition matrix $A$ one defines $\mathcal{O}_A$ abstractly by partial isometries $S_i$, $i=1,\ldots , k$, such that their support and range projections $S_i^*S_i$, $S_iS_i^*$, satisfy certain combinatorial identities in terms of $A$ (see \cite{cuntzkrieger}), and $\mathcal{O}_A$ is seen to contain an isomorphic copy of $C(X,\mathbb{C})$ (indeed it is the $C^*$-algebra $\mathcal{D}_A$ generated by range projections).
A linear functional $\omega$ on $\mathcal{O}_A$ (or more generally for a $C^*$-algebra) is a \textit{state} if $\omega(a a^*)$ is positive for any $a\in \mathcal{O}_A$, $\omega$ has norm $1$, and $\omega(I)=1$. There is a strong relation between equilibrium measures and a family of states called \emph{KMS states}. Given a one-parameter family of authomorphisms $\alpha_t$, a state $\omega$ is said to be a \textit{KMS state} at inverse temperature $t_0$ if it satisfies the \emph{KMS order-swap} $\omega(ab) = \omega( b \alpha_{-t_0}(a))$ for any analytic $a$. For the family of automorphims built from multiplication by potential functions $v:X\to\mathbb{R}$ satisfying certain expansiveness conditions (see \cite{Renault}) it is known that there exists (a unique) KMS state at $t=t_0$ precisely when $P(-t_0 v)=0$ and that the state is integration against the equilibrium state $\mu_{-t_0 v}$ on an isomorphic copy of $C(X,\mathbb{C})$. Moreover, the KMS order-swap is very closely related to a Radon-Nikodym property, and it is shown in \cite{Renault} how to reduce the KMS state to its values on this isomorphic copy of $C(X,\mathbb{C})$. 

A more simple-minded approach is to work the transformation of the algebra 
$$
\left\langle \mathrm{Mult}(\overline{\mathcal{A}}_{L^\infty})), \mathbf{L},\mathbf{T} \right\rangle_{w,w}
$$ 
given by fixing $\mathrm{Mult}(\overline{\mathcal{A}}_{L^\infty}))$ and letting
$$
\alpha_t(\mathbf{T}(i,j)) = \mathbf{M}(\exp tv) \mathbf{T}(i,j),
$$
$$
\alpha_t(\mathbf{T}^\dag(i,j)) = \mathbf{T}^\dag(i,j)\mathbf{M}(\exp -tv) ,
$$
for a given H\"{o}lder continuous $v:X\to \mathbb{R}$. 
We consider the validity of the KMS order-swap rules for the evaluation linear functionals $\phi_\mu(\mathbf{M}(f)) = \int \mathbf{M}(f)(1) d\mu$ where $\mu$ is the equilibrium state of $v$. We also write $\omega_\mu(f) = \int f d\mu$ for $f\in C(X,\mathbb{C})$.

The key observations that we take forward from this discussion are the following:
\begin{itemize}
\item
The identity
\begin{equation}\label{eq:KMS1}
\omega_\mu(f\circ T \,\chi(i,j))=\omega_\mu(f \exp( t_0v(i,j))\chi^\dag(i,j)).
\end{equation}
is equivalent to a KMS order-swap rule for $\phi_\mu$, namely via
$$
 \phi_\mu(\mathbf{T}(i,j) \mathbf{M}(f) \mathbf{T}^\dag(i,j))
  = \phi_\mu(\mathbf{M}(f)  \mathbf{T}^\dag(i,j)\alpha_{t_0}(\mathbf{T}(i,j)) ),
$$
which simplifies since
$$
\mathbf{T}(i,j) \mathbf{M}(f) \mathbf{T}^\dag(i,j)= \mathbf{M}(f\circ T)\bm{\chi}(i,j),
$$
$$
\mathbf{T}^\dag(i,j)\alpha_t(\mathbf{T}(i,j))=\mathbf{M}(\exp tv(i,j)))\bm{\chi}^\dag(i,j).
$$
\item
We have an approximate rule for the ``wrong" evaluation linear functionals: if $\nu$ is given by $d\nu= g d\mu$ then
$$
\omega_\nu(f\circ T \chi(i,j))\approx \omega_\nu(f \exp (t_0v(i,j))\chi^\dag(i,j)).
$$
In terms of $\phi_\mu$ this is:
\begin{align*}
\phi_\nu(\mathbf{T}(i,j) \mathbf{M}(f) \mathbf{T}^\dag(i,j)\mathbf{M}(g))&= \phi_\mu(\mathbf{T}(i,j) \mathbf{M}(f) \mathbf{T}^\dag(i,j)\mathbf{M}(g))
\\
&=  \phi_\mu(\mathbf{M}(f) \mathbf{T}^\dag(i,j) \mathbf{M}(g) \alpha_{t_0}(\mathbf{T}(i,j))) 
\\
&= \phi_\mu(\mathbf{M}(f) \mathbf{M}(\exp t_0v(i,j)))\bm{\chi}^\dag(i,j)\mathbf{M}(g\circ T^\dag_{ij}) )
\\
&= \phi_\nu(\mathbf{M}(f) \mathbf{M}(\exp t_0v(i,j)))\bm{\chi}^\dag(i,j)\mathbf{M}(g\circ T^\dag_{ij}/g) ).
\end{align*}
In this way the regularity of $g$ appears in the error.
\end{itemize}

Our reason for this digression to KMS states is the following. On one hand, a failure of the uniform decay of correlations lead us to construct a limit of evaluation linear functionals $\phi: \left\langle \mathrm{Mult}(\overline{\mathcal{A}}_{L^\infty})\otimes \rho_q(\Gamma), \mathbf{L}_q \right\rangle_{s,q,w,q}\to \mathbb{C}$ that satisfy an approximate KMS rule for $\overline{\mathcal{A}}_{L^\infty}\otimes \rho_q(\Gamma)$. And on the other hand, we can expect some kind of uniqueness result for KMS states (and anything resembling them) under the expanders hypothesis. 
Finally, the so-called KMS-order swap is a narrative that centers $\rho_q(\Gamma)$ -- and its interaction with ``measure-dynamics" -- rather than its representation space.

In the previous paragraphs we restricted ourselves to the examples of a subshift of finite type; we now relax this to the assumption $\textnormal{(Base)}$. 
For each $q\in\mathbb{N}\cup\left\{\infty\right\}$ and $\pi_q=\rho_q,\lambda_q$, let $\alpha[\pi_q]$ be the map 
$$
\alpha[\pi_q](\mathbf{T}(i,j)\otimes I) = \mathbf{M}(r(i,j)) \mathbf{T}^\dag(i,j)\otimes \bm{\pi}^*_q(i,j).
$$

Applying the same reasoning as for (\ref{eq:KMS1}) we should hope to find linear maps $\omega:\overline{\mathcal{A}}_{L^\infty}\otimes \rho_\infty(\Gamma)\to\mathbb{C}$ satisfying, for each $j_0,\cdots, j_{m}$, 
\begin{align*}
&\omega( f\circ T^m\,\chi(j_0,\ldots, j_m)\otimes I)
\\
&= \omega(f\,r(j_0,\ldots,j_{m}) \chi^\dag(j_0,\ldots ,j_{m})\otimes \pi_q(j_0,\ldots ,j_{m})).
\end{align*}
(In addition we note that in the case of a subshift of finite type we have $\chi^\dag(j_0,\ldots ,j_{m})=\chi(j_m)$.) Unfortunately we don't achieve this, but instead achieve a local quasi- KMS order-swap rule.  
\begin{definition}\label{def:KMSapprox}
Given a subset $\Sigma^\prime_{k}$ and a constant $C>1$ we say that there are \emph{local $(\Sigma^\prime_{k},C)$-quasi KMS order-swap} linear functionals if the following holds: there is $i_0,\ldots, i_k\in \Sigma^\prime_{k}$ and linear maps $\omega,\omega_{i_0,\ldots,i_k}:\overline{\mathcal{A}}_{L^\infty}\otimes \rho_q(\Gamma) \to\mathbb{C}$, bounded on pure tensors $f\otimes u$ by $C\|f\|_w\|u\|_{\mathcal{H}_\infty}$, and for which
 $\omega( f\circ T^m \,\chi(j_0,\ldots, j_m))$ is within
\begin{align*}
&C\|f\,r(j_0,\ldots,j_{m}) \chi^\dag(j_0,\ldots, j_m)\chi^\dag(i_0,\ldots, i_k) \|_{w}(\#\Sigma^\prime_k)^{-1}
\\
&+C\|f\,r(j_0,\ldots,j_{m})(\chi^\dag(j_0,\ldots, j_m)- \chi^\dag(i_0,\ldots, i_k))\|_{w},
\end{align*}
of $\omega_{i_0,\ldots, i_k}(f\, r(j_0,\ldots,j_{m}) \otimes  \bm{\rho}^*_q(j_0,\ldots ,j_{m}) )$, for any $j_0,\cdots, j_{m}$ with common prefix $i_0,\ldots, i_k$. 
\end{definition}

\begin{remark}
In the case of a subshift of finite type (Definition \ref{exampleA}) we expect that it suffices to consider $\omega=\omega_{i_0,\ldots,i_k}$. It would also be natural to consider further quasi- order swap rules, by extending the definition of $\alpha[\pi_q]$ to fix $\mathrm{Mult}(\overline{\mathcal{A}}_{L^\infty})\otimes \rho_q(\Gamma)$. In the case that $\omega(f\otimes I) = \int f\, d\mu$, a suitable quasi- order swap is achieved (however, the lack of semigroup structure for the mondoromy sets suggests we ought to restrict the algebra to something smaller).
\end{remark}

\begin{theorem}\label{theorem:KMS1}
Assume that $(X,T,\mu,\alpha)$ and $\mathcal{F}_s(T)$ satisfies \textnormal{(Base)} with $p\in[1,\infty]$. Let $(X_q,T_q,\mu_q)_{q\in\mathbb{N}}$ be a tower of finite sheeted covers and assume that $\mathcal{F}_s(T_q)$ is equipped with a norm satisfying Assumption \textnormal{($\mathbb{C}$-Ext)} with uniform Doeblin--Fortet constant $C_{DF}$.

Fix a co-amenable normal subgroup $\Gamma^\prime$ of $\Gamma$ and corresponding decomposition $\mathbb{C}^{\Gamma_q}=V_q\oplus \mathcal{H}_q$.
Assume that $(\Gamma_q)_{q\in\mathbb{N}}$ are an expander family and let $S\subseteq \Gamma^\prime$ be a finite subset with $\kappa^\prime(S,(\mathcal{H}_q)_{q\in\mathbb{N}})=\kappa^\prime<1$.
Assume that the monodromy $(k,C_{DF},L^p,\epsilon)$-sees $S^n$ for $n$ with $C_{DF} (\kappa^\prime)^n \le \epsilon$, and denote $\Sigma_k(S^n)$ the corresponding prefixes.

Then any local $(\Sigma_k(S^n),C_{DF})$-quasi KMS order-swap linear functionals of Definition \ref{def:KMSapprox} cannot satisfy $\omega(f\otimes I)=\int f\,d\mu$ for every $f\in \overline{\mathcal{A}}_{L^p(\mu,\mathbb{C})}$.
\end{theorem}
\begin{proof}
Aiming for a contradiction assume such linear functionals do exist. Let $i_0,\ldots, i_k$ be given as in Definition \ref{def:KMSapprox} and let $i_0^\gamma,\ldots, i^\gamma_{n_\gamma}$ be given realizing the embedding of $S^n$ into the based monodromy set $H(i_0,\ldots, i_k)$. Further, using the remaining part of the $(k,C_{DF},L^p,\epsilon)$-sees $S^n$-hypothesis, let $0<g\in \overline{\mathcal{A}}_{L^\infty}$ be given such that $\|g\|_{w}\le 1$,
$$
t_n(\gamma)=\int_{E\cap E_\gamma} g \, h\circ (T^\dag_{i_{0},i_1}\circ \cdots T^\dag_{i_{n-1},i_n})\,d m,
$$
and
$$
 t_n^{\prime\prime}(\gamma)=\left(\int_{E\cap E_\gamma} g^p\,d m \right)^{1/p}(\#\Sigma_k(S^n))^{-1}
+\left(\int_{E \Delta E_\gamma} g^p\,d m \right)^{1/p},
$$
with $E = T^k([i_0]\cap \ldots \cap T^{-k}[i_{k}])$, $E_\gamma = T^{n_\gamma}([i^\gamma_0]\cap \ldots \cap T^{-n_\gamma}[i^\gamma_{n_{\gamma}}])$, have that 
$$
t_n := \sum \theta(\gamma) t_n(\gamma), t^\prime_n := \sum \theta(\gamma) t^\prime_n(\gamma),
$$
for the probability $\theta(\gamma) = (\#S)^{-n}$, $\gamma\in S^n$, satisfies
$$
t_n > \epsilon + C_{DF}t_n^\prime.
$$ 

Given $j_0,\ldots, j_m=i^{\gamma}_0,\ldots, i^{\gamma}_{n_\gamma}$ we choose $f= g/ r(j_0,\ldots, j_m)\chi^\dag(i_0,\ldots, i_k)$ whence, using the second identity in Lemma \ref{cor:measure1} generalized to functions,
\begin{align*}
\int f\circ T^m \chi(j_0,\ldots,j_m)\, h dm &=  \int g\chi^\dag(j_0,\ldots,j_m) \chi^\dag(i_0,\ldots, i_k) h\circ (T^\dag_{i_{0},i_1}\circ \cdots T^\dag_{i_{n-1},i_n}) \,dm 
\\
&\ge t_n(\gamma).
\end{align*}
By hypothesis, this lower bound applies to $\omega(f\circ T^m\,\chi(j_0,\ldots, j_m)\otimes I)$. We also note that
$$
g/ r(j_0,\ldots, j_m)\chi^\dag(i_0,\ldots, i_k) r(j_0,\ldots, j_m) = g\chi^\dag(i_0,\ldots, i_k) ,
$$
since the support of $r(j_0,\ldots, j_m)$ contains the support of  $\chi^\dag(i_0,\ldots, i_k) $.
The local quasi- KMS order-swap gives that the former identity is close to
$$
\omega_{i_0,\ldots, i_k}(g\, \chi^\dag(i_0,\ldots, i_{k})\otimes \bm{\rho}^*_q(j_0,\ldots, j_m)),
$$
with error being at most $C_{DF} t_n^{\prime\prime}(\gamma)$. In particular
\begin{align*}
&\Re \omega(g\,\chi^\dag(i_0,\ldots, i_{k})\otimes \bm{\rho}^*_q(j_0,\ldots, j_m))  
\\
&\ge t_n(\gamma) -  C_{DF}  t_n^{\prime\prime}(\gamma).
\end{align*}
It follows that the linear functionals
$$
\phi(\rho_\infty(\gamma^{-1})):= \Re\, \omega(g\,\chi^\dag(i_0,\ldots, i_{k})\otimes \rho_\infty(\gamma^{-1}))
$$
have norm bounded by $C_{DF}\|g\|_{w}=C_{DF}$. Moreover, $\phi$ are positive and real, and bounded from below by 
$$
t_n(\gamma) -  C_{DF}  t_n^{\prime\prime}(\gamma)
$$
on any $\rho_\infty(\gamma)$. By Lemma \ref{lemma:linfndecay} we have
$$
\epsilon< \sum_{\gamma\in S^n} t_n(\gamma) -  C_{DF}  t_n^{\prime\prime}(\gamma)\le C_{DF} (\kappa^\prime)^n.
$$
This is a contradiction.
\end{proof}

We encounter such linear functionals through the following.
\begin{proposition}\label{prop:KMS}
Assume that $(X,T,\mu,\alpha)$ and $\mathcal{F}_s(T)$ satisfies \textnormal{(Base)} with $p\in[1,\infty]$. Let $(X_q,T_q,\mu_q)_{q\in\mathbb{N}}$ be a tower of finite sheeted covers and assume that $\mathcal{F}_s(T_q)$ is equipped with a norm satisfying Assumption \textnormal{($\mathbb{C}$-Ext)} with uniform Doeblin--Fortet constant $C_{DF}$.

Suppose that $\phi$ is the linear map satisfying the conclusion of Theorem \ref{theorem:constructlinearf} and let $i_0,\ldots, i_k\in \Sigma^\prime_k$ be given. Define $\omega,\omega_{i_0,\ldots, i_k}$ for $f\in \overline{\mathcal{A}}_{L^\infty}$ by
$$
\omega(f\otimes u) = \phi(\mathbf{M}(f)\otimes u)
$$
$$
\omega_{i_0,\ldots, i_k}(f\otimes u) = \phi(\mathbf{M}(f)\bm{\chi}^\dag(i_0,\ldots, i_k)\mathbf{T}^\dag(i_0,\ldots,i_k)\otimes u)
$$
Then $\omega$ and $\omega_{i_0,\ldots, i_k}$ verify the local $(\Sigma^\prime_k,C_{DF})$-quasi KMS order-swap. Moreover, $\omega(f\otimes I) = \int f\,d\mu$ for every $f\in \overline{\mathcal{A}}_{L^\infty}$.
\end{proposition}

\begin{proof}
Well definedness of the linear functionals follows from well-definedness of the map 
$$
\mathbf{M}(f)\otimes u \mapsto f\otimes u
$$
from $\left\langle \mathrm{Mult}(\overline{\mathcal{A}_{L^\infty}})\otimes \rho_\infty(\Gamma),\mathbf{L}_\infty \right\rangle_{s,\infty,w,\infty}$ to $\overline{\mathcal{A}_{L^\infty}}\otimes \rho_\infty(\Gamma)$
(which follow from faithfulness of the restriction map is discussed in subsection \ref{subsection:ops}), and boundedness follows from the definitions of the tensor norms in addition to the inequality $\|\mathbf{M}(g) \|_{s\to w}
\le \|g \|_{L^\infty(m,\mathbb{C})}$.

By the left $\mathbf{L}^n_\infty$-invariance, we have
\begin{align*}
&\phi(\mathbf{M}(f\circ T^m)\bm{\chi}(j_0,\ldots,j_m)\otimes I) = \omega(\mathbf{L}^{m}_\infty \,\mathbf{M}(f\circ T^m)\bm{\chi}(j_0,\ldots,j_m)\otimes I)
\\
= & \phi(\mathbf{M}(f)\mathbf{M}(r(j_0,\ldots,j_m))\bm{\chi}^\dag(j_0,\ldots,j_m)\mathbf{T}^\dag(j_0,\ldots,j_m)\otimes \bm{\rho}_\infty(j_0,\ldots, j_m)).
\end{align*}
The final property given in Theorem \ref{theorem:constructlinearf} combined with the previous identity yields the local quasi- KMS order-swap result for $f\in \overline{\mathcal{A}}_{L^\infty}$.
\end{proof}


\section{Correlations and norm divergence for transfer operators}\label{section:resolvent}
In this section we give sufficient conditions for our desired decay or correlations statements in terms of spectral data associated to transfer operators. Let $h\in\mathcal{F}_s$ with $\mathbf{L} h = h$.
\begin{definition}\label{def:correlations}
Given $F_q\in \mathcal{B}_{s,q}$, and $G_q$ assumed to be in either $L^{1}(\mu,\mathcal{H}_q)$ or $L^{p_*}(m,\mathcal{H}_q)$, $q\in\mathbb{N}\cup\left\{\infty\right\}$, the \emph{sequence of correlations} is
$$
w_{F_q,G_q}(n) = \int \mathbf{L}^n_q (\mathbf{M}(h)\otimes I)(F_q) \cdot  G_q \, d m.
$$
\end{definition}
Using duality of the transfer operator with the dynamics (Lemmas \ref{lemma:duality} and \ref{lemma:isofunctions}) we see that these are indeed ``correlations" in the usual sense. 
We can estimate the correlations using the H\"{o}lder inequality (Lemma \ref{lemma:dotupperbound}). There are two different cases according to whether we assume the ACIP density is bounded away from $0$ or not.
For $G_q\in L^{p_*}(m,\mathbb{C})$ we have,
$$
\left| w_{F_q,G_q}(n)\right| \le \left\| \mathbf{L}^n_q\right\|_{s,q\to w,q} D_{s\cdot s}\|h\|_s\| F_q\|_{s,q}\|G_q\|_{L^{p_*}(m,\mathcal{H}_q)},
$$
Supposing $G_q\in L^{1}(m,\mathcal{H}_q)$ and that the ACIP has density exceeding some $c_h>0$. Then
$$
\left| w_{F_q,G_q}(n)\right|\le |z^{-n}| \|\mathbf{L}_q^n\|_{s,q\to s,q} D_{s\cdot s}\|h\|_s\|F_q\|_{s,q}c_h^{-1}\|G_q\|_{L^{1}(\mu,\mathcal{H}_q)}.
$$
In either case we desire a bound on $\delta_{\mathrm{SG},q}$ the spectral radius of $\mathbf{L}_q$ on the space $\mathcal{B}_{s,q}$. We have $\delta_{\mathrm{SG},q}\le 1$ since $\|\mathbf{L}_q^n\|_{s,q\to s,q}\le 2C_{DF}$. Using Gelfand's spectral radius formula combined with our previous observations we deduce the following.

\begin{lemma}\label{lemma:sprimpliesdecay}
If there is $\delta<1$ with $\delta_{\mathrm{SG},q}<\delta$ for every $q\in\mathbb{N}$ then there are $C_q$ such that 
$$
|w_{F_q,G_q}(n)|\le C_q\delta^n \|F_q\|_{s,q}\|G_q\|_{L^1(\mu,\mathcal{H}_q)},
$$
for any $F_q\in\mathcal{B}_{s,q},G_q\in L^1(\mu,\mathcal{H}_q)$, provided $h$ is bounded away from zero; and otherwise
$$
|w_{F_q,G_q}(n)|\le C_q\delta^n \|F_q\|_{s,q}\|G_q\|_{L^{p_*}(\mu,\mathcal{H}_q)},
$$
for any $F_q\in\mathcal{B}_{s,q},G_q\in L^{p_*}(\mu,\mathcal{H}_q)$.
\end{lemma} 

For $z>\delta_{\mathrm{SG},q}$, we observe that the series $\sum_{n=0}^N z^{-1}\mathbf{L}_q^n$ converges in the Banach space $\mathrm{Op}(\mathcal{B}_{s,q},\mathcal{B}_{s,q})$, and its limit is denoted $\mathbf{R}_q(z) = \sum_{n=0}^\infty z^{-n}\mathbf{L}^n_q$. 

The uniform Doeblin-Fortet inequality implies that the essential spectrum for $(\mathbf{L}_q:\mathcal{B}_{s,q}\to \mathcal{B}_{s,q})$ is bounded by $\alpha_{DF}$ uniformly in $q\in\mathbb{N}$ (this powerful observation dates back to \cite{hennion} and is widely used in analysis of the transfer operator). We can use this to deduce the following.
\begin{proposition}\label{prop:hennion}
There are $z_n$ with $1<|z_n|$, $\lim_{n\to\infty}|z_n|= 1$, and
$$
\lim_{n\to\infty}\|\mathbf{R}_\infty(z_n)\|_{s,\infty\to w,\infty}=\infty,
$$
provided that $\limsup_{q\to\infty}\delta_{\mathrm{SG},q}= 1$.
\end{proposition}
\begin{proof}
Since the spectral radius is assumed to approach $1$ it follows that there is a point in the spectrum $\lambda_q$ with $|\lambda_q|\to 1$. Now, by the bound on essential spectrum, $\lambda_q$ is an eigenvalue along a subsequence of $q\in\mathbb{N}$. These $\lambda_q$ are also eigenvalues for $\mathbf{L}_\infty$. The result follows from the geometric series formula.
\end{proof}

Now we consider results for $q=\infty$.
\begin{lemma}\label{lemma:seriesimpliesdecay}
Assume that $h$ is bounded away from zero.
If for each $q\in\mathbb{N}$ we have that $\mathbf{R}_q(1)$ converges and
$$
\sup_{q\in\mathbb{N} }\|\mathbf{R}_q(1)\|_{s,q \to s,q}<\infty,
$$
then for each $F_\infty\in\mathcal{B}_{s,\infty}$ there is $W_\infty\in\mathcal{B}_{s,\infty}$ with
$$
\mathbf{L}_\infty\,(\mathbf{M}(h)\otimes I) (F_\infty + W_\infty-\mathbf{T}_\infty(W_\infty)) = 0.
$$
Moreover, the assignment $F_\infty\mapsto W_\infty$ is continuous in $\mathcal{B}_{s,\infty}$. 
\end{lemma}
\begin{proof}
We claim that the hypothesis
$$
\sup_{q\in\mathbb{N}} \|\mathbf{R}_q(1)\|_{s,q\to s,q}<\infty
$$
implies that $\sum_{n=0}^N\mathbf{L}^n_\infty$ converges to a bounded operator, which we denote $\mathbf{R}_\infty(1)$.
Using the claim we show the existence of the desired elements. Set $F^\prime_\infty  = \mathbf{M}(h)\otimes I(F_\infty)$ and  $W^\prime_\infty = \mathbf{R}_\infty(1)F^\prime_\infty - F^\prime_\infty$. Then $\mathbf{L}_\infty (W^\prime_\infty) - W^\prime_\infty= -\mathbf{L}_\infty(F_\infty)$, from
$$
\mathbf{L}_\infty (W^\prime_\infty) - W^\prime_\infty = \sum_{n=1}^\infty \mathbf{L}_\infty^n \mathbf{L}_\infty(F^\prime_\infty) -  \sum_{n=1}^\infty \mathbf{L}_\infty^n (F^\prime_\infty).
$$
Hence $\mathbf{L}_\infty (W^\prime_\infty) - W^\prime_\infty + \mathbf{L}_\infty(F^\prime_\infty) = 0$.

We now proceed with proving the claim. 
We have the identity
$$
\sum_{n=0}^{N-1} \mathbf{L}_q (F_q) = (I-\mathbf{L}_q^N)\mathbf{R}_q(1) (F_q).
$$
Applying the embedding $\mathbf{P}^*_q$ we get
$$
\sum_{n=0}^{N-1} \mathbf{L}_\infty (\mathbf{P}^*_q F_q) = (I-\mathbf{L}_\infty^N)(\mathbf{P}^*_q\mathbf{R}_q(1)(F_q)).
$$
We claim that we can use this to show that 
$$
\sum_{n=0}^{N-1} \mathbf{L}^n_\infty (\mathbf{P}_q^*F_q) \to \sum_{n=0}^{N-1} \mathbf{L}^n_\infty F_\infty
$$
as $q\to\infty$, uniformly in $N$. To see this first note that for each $N$ there is $q_N$ large enough with the difference being at most $\epsilon$. To complete the observation we'll show the Cauchy property in $q$, independently of $N$. Let $F^\prime_{q-m}=\mathbf{P}_{q}\mathbf{P}^*_{q-m} F_{q-m}$.
We have
\begin{align*}
&\left\|\sum_{n=0}^{N-1} \mathbf{L}_\infty  (\mathbf{P}_q^*F_q- \mathbf{P}_{q-m}^*F_{q-m})\right\|_{s,\infty}  = \left\|\sum_{n=0}^{N-1} \mathbf{L}_q (F_q-F^\prime_{q-m})\right\|_{s,q}
\\
&\le \|(I-\mathbf{L}_q^N)\|_{s,q\to s,q} \| \mathbf{R}_q(1) (F_q- F^\prime_{q-m})\|_{s,q\to s,q}
\\
&\le \|(I-\mathbf{L}_q^N)\|_{s,q\to s,q} \| \mathbf{R}_q(1)\|_{s,q\to s,q} \| \mathbf{P}_{q}^*F_q- \mathbf{P}_{q-m}^*F_{q-m}\|_{s,\infty}.
\end{align*}
So, given $\epsilon>0$ we choose $M$ large enough with the former bounded by $\epsilon$, for $q+m,q>M$, whence also
$$
\left\|\sum_{n=0}^{N-1} \mathbf{L}_\infty (\mathbf{P}^*_{q} F_{q})- \sum_{n=0}^{N-1} \mathbf{L}_\infty  (F_{\infty})\right\|_{s,\infty}
$$
is bounded by
$$
\left\|\sum_{n=0}^{N-1} \mathbf{L}_\infty (\mathbf{P}^*_{q_N} F_{q_N})- \sum_{n=0}^{N-1} \mathbf{L}_\infty  (F_{\infty})\right\|_{s,\infty} + \epsilon,
$$
where $q_N>M$ is chosen such that the former term is also less then $\epsilon$. The conclusion regarding uniform convergence follows.

The uniform convergence implies that $\sum_{n=0}^{N-1} \mathbf{L}_\infty  (F_{\infty})$ is Cauchy and implies boundedness of the limit whence the boundedness of operator $\mathbf{R}_\infty(1)$ follows. 
\end{proof}

\begin{lemma}\label{lemma:diverge1}
If there are $|z_N|\ge 1$ with $\lim_{N\to\infty} |z_N|= 1$, and sequences $(q_N)_{N\in\mathbb{N}}$ with $\mathbf{R}_{q_N}(z_N)$ convergent in $\mathrm{Op}(\mathcal{B}_{s,q_N},\mathcal{B}_{s,q_N})$ and
$$
\lim_{N\to\infty}\|\mathbf{R}_{q_N}(z_N)\|_{s,q_N\to w,q_N}=\infty, 
$$
then are $M_N\in\mathbb{N}$ such that 
$$
a_{N} = \sum_{n=0}^{M_N} z_N^{-n}\mathbf{L}_\infty^n 
$$
satisfies
\begin{equation}\label{eq:hypA5}
\lim_{N\to\infty} \|a_{N}\mathbf{P}_{q_N}^*\mathbf{P}_{q_N}\|_{s,\infty\to w,\infty}^{-1}\| a_{N}\mathbf{P}_{q_N}^*\mathbf{P}_{q_N}(1-\mathbf{L}_\infty)\|_{s,\infty\to s,\infty} = 0.
\end{equation}
\end{lemma} 
\begin{proof}
Since $\mathbf{R}_{q_N}(z_N)$ is convergent may choose $M_N$ large enough (and with $M_N\to\infty$ as $N\to\infty$) such that for all $M>M_N/2$ we have
\begin{align*}
&\left\|\mathbf{P}_{q_N}^*\mathbf{R}_{q_N}(z_N)\mathbf{P}_{q_N}- \sum_{n=0}^{M}z_N^{-n}\mathbf{L}_\infty^n\mathbf{P}_{q_N}^*\mathbf{P}_{q_N}\right\|_{s,\infty\to \bullet,\infty}
\\
&=\left\|\mathbf{R}_{q_N}(z_N)- \sum_{n=0}^{M}z_N^{-n}\mathbf{L}_{q_N}^n\right\|_{s,q_N\to \bullet,q_N} <1/N,
\end{align*}
for $\bullet=s,w$/
Then in particular $\|a_{N}\mathbf{P}_{q_N}^*\mathbf{P}_{q_N}\|_{s,\infty\to w,\infty}\to \infty$. 
We note that $a_{N}\mathbf{P}_{q_N}^*\mathbf{P}_{q_N}(1-\mathbf{L}_\infty)=a_{N}(1-\mathbf{L}_\infty) \mathbf{P}_{q_N}^*\mathbf{P}_{q_N}$.
We have,
\begin{align*}
&\left\|\sum_{n=0}^{M_N}z_N^{-n}\mathbf{L}_{\infty}^n\mathbf{P}_{q_N}^*\mathbf{P}_{q_N} - \sum_{n=1}^{M_N+1}z_N^{-n}\mathbf{L}_{\infty}^n\mathbf{P}_{q_N}^*\mathbf{P}_{q_N}  \right\|_{s,\infty\to s,\infty}
\\
&=\left\|\sum_{n=0}^{M_N}z_N^{-n}\mathbf{L}_{q_N}^n - \sum_{n=1}^{M_N+1}z_N^{-n}\mathbf{L}_{q_N}^n \right\|_{s,q_N\to s,q_N}<2/N + \|\mathbf{L}_{q_N}\|_{s,\infty \to s,\infty}.
\end{align*}
this yields the result.
\end{proof}

\begin{remark}
It would be natural to include the condition $q_N\to\infty$ but this is not a necessary hypothesis. Indeed we don't know that $\delta_{\mathrm{SG},q}<1$ and so the machinery in this work will deduce that result (but we expect a more standard argument involving precompactness and equicontinuity would yield the same).

In the case that $\limsup_{q\to\infty}\delta_{\mathrm{SG},q}= 1$ we have the simpler scenario that we may take $|z_N|>1$ and $q_N=\infty$ whence $\mathbf{P}_{q_N}^*\mathbf{P}_{q_N}=I$.
\end{remark}


\section{Construction of invariant linear functional}\label{section:functional}
The ideas in this section can be traced back to \cite{CoulonDougallSchapiraTapie}. (It is also worth mentioning that \cite{Dougall2} prototypes some parts of the argument in a somewhat easier context.)
The outcome of this section is the following theorem. (See Proposition \ref{prop:KMS} for how we use the linear map.)
\begin{theorem}\label{theorem:constructlinearf}
Suppose that $a_{N}, a_{N,m}$ are given as in Lemma \ref{lemma:diverge1}. Then for any choice of $\Sigma^\prime_{k}\subset\Sigma_{k}$ there is a linear map 
$$
\omega: \left\langle \mathrm{Mult}(\overline{\mathcal{A}_{L^\infty}})\otimes \rho_\infty(\Gamma),\mathbf{L}_\infty \right\rangle_{s,\infty,w,\infty}\to \mathbb{C}
$$ 
with norm at most $C_{DF}^2$ and with the following properties:
\begin{enumerate}
\item\label{item:linfnunit}
We have $\phi (\mathbf{M}(f)\otimes I)=\int f\, d\mu$, and
\item\label{item:linfninvariance} 
satisfies the identity 
$$
\phi = \phi\circ \mathbf{L}_\infty^n ,
$$
for any $n\in\mathbb{N}$.
\item\label{item:linfncloseness}
There is $i_0,\ldots ,i_k\in \Sigma^\prime_{k}$ such that for any $j_k,\ldots, j_n$, we have that
$$
|\phi( \mathbf{M}(g) \,\mathbf{d}^\dag(i_0,\ldots, i_{k-1};j_k,\ldots, j_n)\otimes u)|
$$ 
is bounded by $\|u\|_{\mathcal{H}_\infty\to \mathcal{H}_\infty}$ times
\begin{align*}
&C_{DF} \|g\chi^\dag(i_0,\ldots, i_{k-1})\chi^\dag(i_0,\ldots, i_{k-1}j_k,\ldots, j_n)\|_{w}(\#\Sigma^\prime_k)^{-1}
\\
&+ C_{DF}\|g(\chi^\dag(i_0,\ldots, i_{k-1})-\chi^\dag(i_0,\ldots, i_{k-1}j_k,\ldots, j_n))\|_{w},
\end{align*}
for any $g\in \overline{\mathcal{A}}_{L^\infty}$.
\end{enumerate}
\end{theorem}

Using Lemma \ref{lemma:lebesguediff}, we can choose unit norm $F_N\in \mathcal{B}_{s,\infty},W_N\in L^{p_*}(\mu,\mathcal{H}_\infty)$ such that the linear map $\phi^\prime_N:\mathrm{Op}(\mathcal{B}_{w,\infty},\mathcal{B}_{w,\infty})\to \mathbb{C}$
$$
b\mapsto \int b\,(F_N) \cdot W_N\, d\mu
$$
 satisfies
 $|\phi^\prime_N(a_N \mathbf{P}_{q_N}^*\mathbf{P}_{q_N})|\ge (1-1/N)\|a_N\mathbf{P}_{q_N}^*\mathbf{P}_{q_N}\|_{s,\infty \to w,\infty}$.
We denote 
$$
\phi_N= \|a_N\mathbf{P}_{q_N}^*\mathbf{P}_{q_N}\|_{s,\infty\to w,\infty}^{-1}\phi_N^\prime\circ a_N\mathbf{P}_{q_N}^*\mathbf{P}_{q_N},
$$ 
and restrict it to a linear functional $\left\langle \mathrm{Mult}(\mathcal{F}_s)\otimes \rho_\infty(\Gamma),\mathbf{L}_\infty \right\rangle_{s,\infty,s,\infty}\to\mathbb{C}$. (Recall we argue in subsection \ref{subsection:ops} that the (composition of) restrictions maps from $\left\langle \mathrm{Mult}(\mathcal{F}_s)\otimes \rho_\infty(\Gamma),\mathbf{L}_\infty \right\rangle_{w,\infty,w,\infty}$ to $\left\langle \mathrm{Mult}(\mathcal{F}_s)\otimes \rho_\infty(\Gamma),\mathbf{L}_\infty \right\rangle_{s,\infty,s,\infty}$ are morphisms.) One has the lazy upper bounds
\begin{equation}\label{eq:bound1}
|\phi_{N}(b)| \le \| b \mathbf{P}_{q_N}^*\mathbf{P}_{q_N}F_N\|_{s,\infty}\le \| b \mathbf{P}_{q_N}^*\mathbf{P}_{q_N}\|_{s,\infty\to s,\infty}\le \| b \|_{s,\infty\to s,\infty},
\end{equation}
for $b\in\left\langle \mathrm{Mult}(\mathcal{F}_s)\otimes \rho_\infty(\Gamma),\mathbf{L}_\infty \right\rangle_{s,\infty,s,\infty}$. 

Our first goal is to construct a weak limit for $\phi_N$ and to extend it to $\left\langle \mathrm{Mult}(\mathcal{F}_s)\otimes \rho_\infty(\Gamma),\mathbf{L}_\infty \right\rangle_{s,\infty,w,\infty}$.

\begin{proposition}\label{prop:phiinftyexists1}
Any subsequence of $\mathbb{N}$ has a further subsequence along which $\phi_{N}(b)$ converges for all $b\in \left\langle \mathrm{Mult}(\mathcal{F}_s)\otimes \rho_\infty(\Gamma),\mathbf{L}_\infty \right\rangle_{s,\infty,s,\infty}$. Along any convergent subsequence one has $\phi_{N}(I)\to 1$, and
$$
\phi_{N}((\mathbf{L}^n_\infty -I)b)\to 0,
$$ 
for any $n\in\mathbb{N}$ and $b\in \left\langle \mathrm{Mult}(\mathcal{F}_s)\otimes \rho_\infty(\Gamma),\mathbf{L}_\infty \right\rangle_{s,\infty,s,\infty}$.
\end{proposition}
\begin{proof}
The Banach--Alaoglu theorem tells us that the the closed unit ball in the dual of $\left\langle \mathrm{Mult}(\mathcal{F}_s)\otimes \rho_\infty(\Gamma),\mathbf{L}_\infty \right\rangle_{s,\infty,s,\infty}$ is compact in the weak-* topology. Since $\left\langle \mathrm{Mult}(\mathcal{F}_s)\otimes \rho_\infty(\Gamma),\mathbf{L}_\infty \right\rangle_{s,\infty,s,\infty}$ has a countable dense subset we can upgrade compactness to sequential compactness. Hence we let $\phi_\infty$ be a weak-* accumulation point of $\phi_N$. Using Lemma \ref{lemma:diverge1} (\ref{eq:hypA5}) we arrive at the latter claim.
\end{proof}

We show how to extend any such limit linear functional.
\begin{proposition}\label{prop:phiinftyexists2}
Suppose that $\phi_\infty:\left\langle \mathrm{Mult}(\mathcal{F}_s)\otimes \rho_\infty(\Gamma),\mathbf{L}_\infty \right\rangle_{s,\infty,s,\infty}\to\mathbb{C}$ is a bounded linear functional such that $\phi_\infty$ has unit norm, and with
$$
\phi_\infty \circ \mathbf{L}^n_\infty\,= \phi_\infty .
$$
Then $\phi_\infty$ satisfies the norm bounds
$$
|\phi_\infty(b)|\le C_{DF} \|b\|_{s,\infty\to w,\infty},
$$
for any $b\in \left\langle \mathrm{Mult}(\mathcal{F}_s)\otimes \rho_\infty(\Gamma),\mathbf{L}_\infty \right\rangle_{s,\infty,s,\infty}$.
\end{proposition}
\begin{proof}
Given $a\in \left\langle \mathrm{Mult}(\mathcal{F}_s)\otimes \rho_\infty(\Gamma),\mathbf{L}_\infty \right\rangle_{s,\infty,s,\infty}$, we have,
$$
|\phi_\infty(a)|=|\phi_\infty(\mathbf{L}_\infty^m a)|\le \|\mathbf{L}_\infty^m a\|_{s,\infty\to,s\infty}.
$$
Since $\|\mathbf{L}_\infty^m a\|_{s,\infty\to,\infty}\le C_{DF}\|a\|_{s,\infty\to w,\infty} + C_{DF}\alpha_{DF}^m \|a\|_{s,\infty\to s,\infty}$ and since $m\in\mathbb{N}$ is arbitrary the result follows.
\end{proof}
It follows from the Hahn-Banach theorem that such $\phi_\infty$ defined on $\left\langle \mathrm{Mult}(\mathcal{F}_s)\otimes \rho_\infty(\Gamma),\mathbf{L}_\infty \right\rangle_{s,\infty,w,\infty}$ extends to the closure, which we know contains $\left\langle \mathrm{Mult}(\overline{\mathcal{A}}_L^\infty)\otimes \rho_\infty(\Gamma),\mathbf{L}_\infty \right\rangle_{s,\infty,w,\infty}$. 

We shall say that $\phi:\left\langle \mathrm{Mult}(\overline{\mathcal{A}}_{L^\infty})\otimes \rho_q(\Gamma), \mathbf{L}_q \right\rangle_{s,q,w,q}\to \mathbb{C}$ is \emph{absolutely continuous} if the ergodic theorem holds for any $g\in\overline{\mathcal{A}}_{L^\infty}$: 
\begin{equation}\label{eq:abscts}
\left|N^{-1}\sum_{n=0}^{N-1}\phi(\mathbf{M}(g\circ T^n)\otimes I) - \int g\,d\mu\, \phi(I)\right|\to 0,\,\mathrm{as}\,N\to\infty.
\end{equation}
Note that we have the following norm inequalities: if $\phi$ has norm $\|\phi\|_{s,\infty\to w,\infty}$, it follows that for any $f\in \mathcal{F}_s$ and $u\in \rho_\infty(\Gamma)$, we have
$$
|\phi(\mathbf{M}(f)\otimes u )|\le \|\phi\|_{s,\infty\to w,\infty}\|u\|_{\mathcal{H}_\infty}\|f\|_{w}.
$$
\begin{lemma}
Any bounded linear functional 
$$
\phi:\left\langle \mathrm{Mult}(\overline{\mathcal{A}}_{L^\infty})\otimes \rho_q(\Gamma), \mathbf{L}_q \right\rangle_{s,q,w,q}\to \mathbb{C}
$$
 with
$$
\phi \circ \mathbf{L}^n_\infty = \phi ,
$$
is absolutely continuous
\end{lemma}
\begin{proof}
\noindent($p\in[1,\infty)$.)
We have
$$
N^{-1}\sum_{n=0}^{N-1} \phi(\mathbf{M}(f\circ T^n)\otimes I) = \phi\left(N^{-1}\sum_{n=0}^{N-1} \mathbf{M}(f\circ T^n)\otimes I\right).
$$
whence
\begin{align*}
&\left|\phi\left(N^{-1}\sum_{n=0}^{N-1} \mathbf{M}(f\circ T^n)\otimes I - \mathbf{M}((\textstyle{\int} f\,d\mu) )\otimes I\right)\right|
\\
&\le \|\phi\|_{s,\infty\to w,\infty} \left\|N^{-1}\sum_{n=0}^{N-1} f\circ T^n - (\textstyle{\int} f\,d\mu) 1\right\|_w,
\end{align*}
where $\|\phi\|_{s,\infty\to w,\infty}$ is the norm of the linear functional.
Since the ergodic theorem holds in $L^p(\mu,\mathbb{C})$, $p\in[1,\infty)$, and since $\mu$ and $m$ have the same sets of measure zero, the result follows.

\noindent($p=\infty$, $h=1$.) In this case $\|\cdot\|_w=\|\cdot\|_{L^\infty}$.
Using the left $\mathbf{L}_\infty^m$ identity, we have that
\begin{align*}
\phi(\bm{\chi}(j_0,\ldots,j_m)\otimes I) &= \omega(r(j_0,\ldots,j_m)\bm{\chi}^\dag(j_0,\ldots,j_m)\mathbf{T}^\dag(j_0,\ldots,j_m)\otimes I)
\\
& \le \|\phi\|_{s,\infty\to w,\infty}\| r(j_0,\ldots,j_m)\|_{L^\infty}.
\end{align*}
By the standing assumption, when $p=\infty$ we have 
$$
\| r(j_0,\ldots,j_m)\|_\infty\le D_{\mathrm{Gibbs}}\mu([j_0,\ldots, j_m]),
$$ 
where $[j_0,\ldots, j_m]= [j_0]\cap T^{-1}[j_1]\cap \ldots \cap T^{-m}[j_m]$. It follows by linearity that $|\phi(\mathbf{M}(g))|\le D_{\mathrm{Gibbs}}\|\phi\|_{s,\infty\to w,\infty}\int |g|d\mu$ for any simple function $g$ built from these.  By the DCT the conclusion is also true for any $g\in \overline{\mathcal{A}}_{L^\infty}$. The conclusion follows since the ergodic theorem holds in $L^1$.
\end{proof}

For brevity, given $F^\prime_N\in\mathcal{B}_{s,\infty}$, we shall say that a sequence $i_0,\ldots, i_{k}$ \emph{works for $F^\prime_N$} if it fulfills the closeness criterion in Assumption \ref{assump5} (\ref{5item4}).
\begin{proposition}\label{prop:subseq1}
Given $\Sigma^\prime_{k}$ there are infinite $\Lambda^\prime\Lambda^{\prime\prime}\subseteq \mathbb{N}$, and $i_0,\ldots,i_k\in \Sigma^\prime_{k}$ that works for $M^{-1} \sum_{n=0}^{M-1}\mathbf{L}_\infty^n \,F_N$ for every $M\in\Lambda^{\prime\prime}$ and every $N\in \Lambda^{\prime}$.
\end{proposition}
\begin{proof}
We now elaborate on the choice of $\Lambda^\prime$.
Create a table with columns indexed by $N$ and rows indexed by $M$, and entries the $i_0,\ldots, i_k\in \Sigma^\prime_k$ that works for $M^{-1} \sum_{n=0}^{M-1}\mathbf{L}_\infty^n\,F_N$. Going down a column to a fixed $M$, there are only finitely many possible $u_0,u_1,\ldots, u_N$ (with $u_i$ a length $k$ sequences) so there must be infinitely many $N$ that see some fixed $u_0,u_1,\ldots, u_M$; call this $\Lambda(u_1,\ldots, u_M)$. Moreover, we must have $\Lambda(u_1,\ldots, u_M)\subseteq \Lambda(u_1,\ldots, u_{M-1})$. This creates a partial order on sequences.
There must be some $v$ such that $\Lambda(u_1,\ldots, u_{M},v)\ne \emptyset$ and a countably infinite. This tells us that any finite totally ordered set has an upper bound. If there are no infinite chains then nevertheless we have found an upper bound for them and so there must be a maximum element -- but this is ludicrous.  Therefore there must be an infinite totally ordered set. Then this defines a sequence $(u_m)_{m\in\mathbb{N}}$ such that $\Lambda(u_1,\ldots, u_m)\ne \emptyset$ for every $m\in\mathbb{N}$. This gives infinite descending $\Lambda^{\prime}(M)$ of $N$. Now we choose $\Lambda^{\prime\prime}$ to be a subset of $m$ where $u_m$ is constant. Let $\Lambda^{\prime}$ be a diagonal sequence over those $M\in \Lambda^{\prime\prime}$. The result follows.\end{proof}

We are now ready to prove the main theorem of this section. 
\begin{proof}[Proof of Theorem \ref{theorem:constructlinearf}] 
We assume that $\phi_\infty$ is an accumulation point over any $\Lambda{^\prime\prime}$.
Let $\phi_\infty:\left\langle \mathrm{Mult}(\overline{\mathcal{A}_{L^\infty}})\otimes \rho_\infty(\Gamma),\mathbf{L}_\infty \right\rangle_{s,\infty,w,\infty}\to\mathbb{C}$ be the bounded linear functional given by Propositions \ref{prop:phiinftyexists1} and \ref{prop:phiinftyexists2}. 
We let $\omega$ be an accumulation point of 
$$
\omega_M (b)= M^{-1}\sum_{n=0}^{M-1}\phi_\infty(b\,\mathbf{L}_\infty^n),
$$ 
along $M\in \Lambda^{\prime\prime}$. It follows that $\omega$ is right-$\mathbf{L}_\infty$ invariant. From the left $\mathbf{L}_\infty$-invariance we have $\omega(I\otimes I)=1$. 
We have the identity 
$$
\phi_\infty(\mathbf{M}(f\circ T^n)\otimes I) = \phi_\infty(\mathbf{M}(f)\otimes I\,\mathbf{L}_\infty^n).
$$
The absolute continuity of $\phi_\infty$ therefore implies that $\omega(\mathbf{M}(f)\otimes I) = \int f\, d\mu$.

It only remains to show \ref{item:linfncloseness}.
Denote $b_M=M^{-1}\sum_{n=0}^{M-1}\mathbf{L}^n_\infty$.
The lazy upper bound (\ref{eq:bound1}) and left $\mathbf{L}_\infty$ identity gives (in the same way as Proposition \ref{prop:phiinftyexists2})
\begin{align*}
&|\phi_\infty(a_n\otimes u \, b_M)|
\\
&\le \limsup_{N\to\infty, N\in \Lambda^\prime}C_{DF}\|I\otimes u\|_{w,\infty\to w,\infty}\|  a_n\otimes I \,  b_M F_N\|_{w,\infty} .
\end{align*}
Suppose that $a_n\to \mathbf{d}^\dag(i_0,\ldots, i_{k-1};j_k,\ldots, j_n)$ in $\|\cdot\|_{s\to w}$. Then $a_n\otimes u\to \mathbf{d}^\dag(i_0,\ldots, i_{k-1};j_k,\ldots, j_n)\otimes u$ in $\|\cdot\|_{s,\infty\to w,\infty}$.
So $\lim_{n\to \infty}|\phi_\infty(a_n\otimes u \, b_M)| =|\phi_\infty(\mathbf{d}^\dag(i_0,\ldots, i_{k-1};j_k,\ldots, j_n)\otimes u \, b_M)|$. But also
\begin{align*}
\|  a_n\otimes I \,  b_M F_N\|_{w,\infty}&\le \|  (a_n-\mathbf{d}^\dag(i_0,\ldots, i_{k-1};j_k,\ldots, j_n))\otimes I \,  b_M F_N\|_{w,\infty}
\\
&+ \|  \mathbf{d}^\dag(i_0,\ldots, i_{k-1};j_k,\ldots, j_n)\otimes I \,  b_M F_N\|_{w,\infty}
\\
\le & 
\|  (a_n-\mathbf{d}^\dag(i_0,\ldots, i_{k-1};j_k,\ldots, j_n))\otimes I \|_{s,\infty\to w,\infty}  \|b_M F_N\|_{s,\infty}
\\
+& \|  \mathbf{d}^\dag(i_0,\ldots, i_{k-1};j_k,\ldots, j_n)\otimes I \,  b_M F_N\|_{w,\infty}
\end{align*}
We conclude that
\begin{align*}
&|\phi_\infty(\mathbf{d}^\dag(i_0,\ldots, i_{k-1};j_k,\ldots, j_n)\otimes u \, b_M)|
\\
&\le \limsup_{N\to\infty, N\in \Lambda^\prime}C_{DF}\|I\otimes u\|_{w,\infty\to w,\infty}\|  \mathbf{d}^\dag(i_0,\ldots, i_{k-1};j_k,\ldots, j_n)\otimes I \,  b_M F_N\|_{w,\infty} .
\end{align*}
We have a constant upper bound for $M\in\Lambda^{\prime\prime}$, $N\in \Lambda^\prime$. 
This says exactly that $|\omega_M(\mathbf{d}^\dag(i_0,\ldots, i_{k-1};j_k,\ldots, j_n)\otimes u)|$ has the desired upper bound, whence also $|\omega(\mathbf{d}^\dag(i_0,\ldots, i_{k-1};j_k,\ldots, j_n)\otimes u)|$ has the desired upper bound.
\end{proof}


\section{Proof of main theorems: $\mathcal{H}_q$-valued functions}\label{section:proofmainvector}
\begin{theorem}\label{theorem:onevector}
Assume that $(X,T,\mu,\alpha)$ and $\mathcal{F}_s(T)$ satisfies \textnormal{(Base)} with $p\in[1,\infty]$. Let $(X_q,T_q,\mu_q)_{q\in\mathbb{N}}$ be a tower of finite sheeted covers and assume that $\mathcal{F}_s(T_q)$ is equipped with a norm satisfying Assumption \textnormal{($\mathbb{C}$-Ext)} with uniform Doeblin--Fortet constant $C_{DF}$.

Fix a co-amenable normal subgroup $\Gamma^\prime$ of $\Gamma$ and corresponding decomposition $\mathbb{C}^{\Gamma_q}=V_q\oplus \mathcal{H}_q$.
Assume that $(\Gamma_q)_q$ are an expander family and let $S\subseteq \Gamma^\prime$ be a finite subset with $\kappa^\prime(S,(\mathcal{H}_q)_q)<1$ for every $q\in \mathbb{N}$.
Assume that the monodromy $(C_{DF},\epsilon,L^p)$-sees $S^n$ for $n\ge N$ with $C_{DF} (\kappa^\prime(S))^N\le \epsilon$.

Then there is $\delta<1$ and $C>0$ such that 
\begin{itemize}
\item 
for any $F_q\in\mathcal{B}_{s,q},G_q\in L^1(\mu,\mathcal{H}_q)$
$$
|w_{F_q,G_q}(n)|\le C \delta^n \|F_q\|_{s,q}\|G_q\|_{L^1(\mu,\mathcal{H}_q)},
$$
in the case that the density of the ACIP is bounded away from $0$; otherwise,
\item 
for any $F_q\in\mathcal{B}_{s,q},G_q\in L^{p_*}(m,\mathcal{H}_q)$
$$
|w_{F_q,G_q}(n)|\le C \delta^n \|F_q\|_{s,q}\|G_q\|_{L^{p_*}(\mu,\mathcal{H}_q)}.
$$
\end{itemize}
\end{theorem}
\begin{proof}
Aiming for a contradiction suppose that $\delta_{\mathrm{SG},\infty}=1$. Then (by Proposition \ref{prop:hennion}) we deduce the existence of the linear functional in Theorem \ref{theorem:constructlinearf}. Therefore we deduce from Proposition \ref{prop:KMS} the existence of the non-trivial local quasi- KMS order swap linear functional in Theorom \ref{theorem:KMS1}, which is a contradiction.
\end{proof}

Recalling the definition of variance given in (\ref{eq:variance}), the vector valued analogue is
$$
\sigma_{q}(F) := \lim_{n\to\infty} \frac{1}{n}\int \left\| \sum_{k=0}^{n-1} \mathbf{T}_q^k F \right\|_{\mathcal{H}_q}^2 d\mu,
$$
for $F\in \mathcal{B}_{s,q}$.
\begin{theorem}\label{theorem:threevector}
Assume that $(X,T,\mu,\alpha)$ and $\mathcal{F}_s(T)$ satisfies \textnormal{(Base)} with $p\in[1,\infty]$. Let $(X_q,T_q,\mu_q)_{q\in\mathbb{N}}$ be a tower of finite sheeted covers and assume that $\mathcal{F}_s(T_q)$ is equipped with a norm satisfying Assumption \textnormal{($\mathbb{C}$-Ext)} with uniform Doeblin--Fortet constant $C_{DF}$.

Fix a co-amenable normal subgroup $\Gamma^\prime$ of $\Gamma$ and corresponding decomposition $\mathbb{C}^{\Gamma_q}=V_q\oplus \mathcal{H}_q$.
Assume that $(\Gamma_q)_q$ are an expander family and let $S\subseteq \Gamma^\prime$ be a finite subset with $\kappa^\prime(S,(\mathcal{H}_q)_{q\in\mathbb{N}})<1$.
Assume that the monodromy $(C_{DF},\epsilon,L^p)$-sees $S^n$ for $n\ge N$ with $C_{DF} \kappa^\prime(S,(\mathcal{H}_q)_{q\in\mathbb{N}})^N\le \epsilon$.

For each $q\in\mathbb{N}\cup\left\{ \infty\right\}$, the limit defining for $\sigma_{q}(F)$ exists for any $F\in \mathcal{B}_{s,q}$. We have $\sigma_{q}: \mathcal{B}_{s,q}\to [0,\infty)$ is continuous in the $\mathcal{B}_{s,q}$ topology, and vanishes precisely on those functions that are coboundaries.
\end{theorem}
\begin{proof}
Let us first point out that we must fulfill the hypotheses of Lemma \ref{lemma:seriesimpliesdecay}, for otherwise as before there is a contradiction.
Firstly, it is standard to expand the formula for variance to get
\begin{equation}\label{eq:variance} 
\sigma_{q}(F)= \int F\cdot F\,d\mu+ \lim_{n\to\infty} \frac{1}{n} 2\sum_{k=1}^{n} (n-k) \int F\cdot \mathbf{T}^k_q F\,d\mu.
 \end{equation}
Using Lemma \ref{lemma:seriesimpliesdecay} we deduce the existence of $W\in \mathcal{B}_{s,\infty}$ with 
$$
\int (F + W - \mathbf{T}_\infty W)\cdot \mathbf{T}_\infty^k (F + W - \mathbf{T}_\infty W) \,d\mu = 0,
$$
for every $k\in\mathbb{N}$. Then by \ref{eq:variance} we have
$$
\sigma_{\infty}(F + W - \mathbf{T}_\infty W) = \int \|F + W - \mathbf{T}_\infty W\|_{\mathcal{H}_q}^2 d\mu.
$$
It is standard to see that $F$ has the same variance as $F + W - \mathbf{T}_\infty W$, using
\begin{align*}
\sum_{k=0}^{n-1} \mathbf{T}_\infty^k(F + W - \mathbf{T}_\infty W) &= \sum_{k=0}^{n-1} \mathbf{T}^k_\infty F+ \sum_{k=0}^{n-1}(W-\mathbf{T}^{k}_\infty W)
\\
&=\sum_{k=0}^{n-1} \mathbf{T}^k_\infty F + W-\mathbf{T}^{n}_\infty W
\end{align*}
\begin{align*}
& \int \sum_{k=0}^{n-1}\mathbf{T}^k_q (F + W - \mathbf{T}_\infty W)\cdot \sum_{j=0}^{n-1}\mathbf{T}^j_q (F + W - \mathbf{T}_\infty W)\,d\mu
\\
&=
\int \sum_{k=0}^{n-1}\mathbf{T}^k_q (F)\cdot \sum_{j=0}^{n-1}\mathbf{T}^j_q (F)\,d\mu
\\
&+
\int \sum_{k=0}^{n-1}\mathbf{T}^k_q (F)\cdot (W-\mathbf{T}^{n}_\infty W)\,d\mu
\\
&+
\int (W-\mathbf{T}^{n}_\infty W)\cdot \sum_{k=0}^{n-1}\mathbf{T}^k_q (F)\,d\mu
\\
&+
\int (W-\mathbf{T}^{n}_\infty W)\cdot (W-\mathbf{T}^{n}_\infty W)\,d\mu.
\end{align*}
Summability of the transfer operator tells us that the integral against $n^{-1}\sum_{k=0}^{n-1}\mathbf{T}^k_q (F)$ tends to zero. (Or equivalently one can rephrase this as Birkhoff's ergodic theorem.)

All that remains is to show continuity. We claim that
$$
\sigma_q(F+W-\mathbf{T}_\infty(W)) = \int (F+W-\mathbf{T}_\infty(W))\cdot (F+W-\mathbf{T}_\infty(W)) d\mu.
$$
Now, let $F_1,F_2\in \mathcal{B}_{s,\infty}$ and $W_1,W_2\in \mathcal{B}_{s,\infty}$ given by Lemma \ref{lemma:seriesimpliesdecay}. We have
\begin{align*}
|\sigma_\infty(F_1)-\sigma(F_2)|&=|\sigma_\infty(F_1+W_1-\mathbf{T}(W_1))-\sigma_\infty(F_2+W_2-\mathbf{T}(W_2))| 
\\
&= \int ((F_1-F_2)+(W_1-W_2)-\mathbf{T}_\infty(W_1-W_2)) \cdot G_1 d\mu 
\\
&+  \int G_2\cdot ((F_1-F_2)+(W_1-W_2)-\mathbf{T}_\infty(W_1-W_2)) \,d\mu,
\end{align*}
where $G_1 = F_1+W_1-\mathbf{T}_\infty(W_1)$ and $G_2 = F_2+W_2-\mathbf{T}_\infty(W_2)$.
Then, using Lemma \ref{lemma:dotupperbound}, and using that $L^1$ is a lower bound for $L^p$ since we are in a probability space, we have
\begin{align*}
&\left|\int (F_1-F_2)+(W_1-W_2)-\mathbf{T}_\infty(W_1-W_2) \cdot G_1 \,d\mu\right|
\\
&\le \|G_1\|_{L^{\infty}_\mu(X,\mathcal{H}_q)}(\| F_1-F_2\|_{w,q} + \|W_1-W_2\|_{w,q} + \|\mathbf{T}_\infty(W_1-W_2) \|_{w,q}).
\end{align*}
and similarly
\begin{align*}
&\left|\int G_2\cdot (F_1-F_2)+(W_1-W_2)-\mathbf{T}_\infty(W_1-W_2) \,d\mu\right|
\\
&\le \|G_2\|_{L^{\infty}_\mu(X,\mathcal{H}_q)}(\| F_1-F_2\|_{w,q} + \|W_1-W_2\|_{w,q} + \|\mathbf{T}_\infty(W_1-W_2) \|_{w,q}).
\end{align*}
Then it follows that for $i=1,2$ using that $L^\infty$ is a lower bound for the strong norm,
$$
\|G_i\|_{L^{\infty}_\mu(X,\mathcal{H}_q)}\le 2\|W_i\|_{L^{\infty}_\mu(X,\mathcal{H}_q)} + \|F_i\|_{L^{\infty}_\mu(X,\mathcal{H}_q)}\le 2\|W_i\|_{s,q} + \|F_i\|_{s,q}
$$
Whence,
$$
|\sigma(F_1)-\sigma(F_2)|\le 2(\| F_1-F_2\|_{s,q} + 2 \|W_1-W_2\|_{s,q})^2.
$$
This estimate implies continuity.
\end{proof}


\section{Proof of main theorems: $\mathbb{C}$-valued functions}\label{section:proofmain}
Our main theorems concern decay of correlations and variance of functions defined on $X\times \Gamma_q$, and therefore our first objective in this section is translate between these and vector valued functions. Throughout we assume we are given $\mathbb{C}^{\Gamma_q} = \mathcal{H}_q\oplus V_q$.

For $q\in\mathbb{N}$ we denote $\mathcal{F}_{\bullet}(T_q)$ those functions $f\in L^p_{m_q}(X_q,\mathbb{C})$ such that $f_\gamma\in\mathcal{F}_\bullet(T)$. We shall call $f$ \emph{simple} if there is some $h\in\mathcal{F}_\bullet(T)$ with $f(x,\cdot)/h(x)\in \mathbb{C}^{\Gamma_q}$ a constant unit vector almost surely in $x\in X$. Define, for $\bullet=s,w$, $\mathcal{F}_\bullet(T_q;\mathcal{H}_q)$ as the linear span of simple functions for which $f(x,\cdot)/h(x)\in \mathcal{H}_q$ and $\mathcal{F}_\bullet(T_q;\mathcal{H}_q)$ as the linear span of simple functions for which $f(x,\cdot)/h(x)\in V_q$.
\begin{lemma}
For $\bullet=s,w$, $\mathcal{F}_{\bullet}(T_q)$ is the linear span of simple functions.
We have that $\mathcal{F}_\bullet(T_q;\mathcal{H}_q) \cap \mathcal{F}_\bullet(T_q;V_q)=\left\{ 0\right\} $ and the joint linear span of $\mathcal{F}_\bullet(T_q;\mathcal{H}_q) $ and $\mathcal{F}_\bullet(T_q;V_q) $ is $\mathcal{F}_\bullet(T_q)$.
Moreover, one has $\mathcal{F}_{s}(T_q;\mathcal{H}_q)=\mathcal{F}_{s}(T_q)\cap \mathcal{F}_w(T_q;\mathcal{H}_q)$.
\end{lemma}
\begin{proof}
Let $v^{(i)}$, $i=1,\ldots, \dim(V_q)$ be an orthonormal basis of $V_q$ and $w^{(i)}$, $i=1,\ldots, \dim(\mathcal{H}_q)$ an orthonormal basis of $\mathcal{H}_q$. The first statements follows from the definition. For the remaining part it is useful to observe that $\mathcal{F}_\bullet(T_q;\mathcal{H}_q)$, respectively $\mathcal{F}_\bullet(T_q;V_q)$, are those $f\in \mathcal{F}_\bullet(T_q)$ in the kernel of the respective linear maps
$$
f\mapsto \left(x\mapsto \sum_{i=1}^{\dim(\mathcal{H}_q)}\sum_{\gamma\in\Gamma_q} f_\gamma(x) \delta_\gamma\cdot w^{(i)}\right),
$$
$$
f\mapsto \left(x\mapsto \sum_{i=1}^{\dim(V_q)}\sum_{\gamma\in\Gamma_q} f_\gamma(x)\delta_\gamma\cdot v^{(i)}\right).
$$
\end{proof}
Given a vector valued function $F(x) = h(x)v$ we may define $f(x,\gamma) = h(x)v\cdot \delta_\gamma$ for $\gamma\in \Gamma_q$. Then $\iota_q(F) = f$ extends to an isomorphism of $\mathcal{F}_{\bullet,q}$ and $\mathcal{F}_{\bullet}(T_q;\mathcal{H}_q)$, and an isomorphism of $L^{p_*}_{m}(X,\mathbb{C}^{\Gamma_q})$ and $L^{p_*}_{\overline{m}_q}(X_q,\mathbb{C})$, where $\overline{m}_q$ for the product of $m$ and the Haar measure on $\Gamma_q$.  This also allows us to more easily describe how to include $\mathcal{F}_{s}(T_q)$ into $L^2_{\overline{m}_\infty}(X_\infty,\mathbb{C})$. Given $f\in \mathcal{F}_{s}(T_q)$ we have $\iota_q^{-1}(f)\in L^{p_*}_{m}(X,\mathbb{C}^{\Gamma_q})\cap L^{2}_{m}(X,\mathbb{C}^{\Gamma_q})$, since $f(x,\gamma)$ is pointwise bounded. Now.
$$
F= \mathbf{P}^*_q\iota_q^{-1}(f)\in L^2_m(X, L^2(\Gamma_\infty, \mathrm{Haar})).
$$ 
Restricting to $f\in \mathcal{F}_{s}(T_q;\mathcal{H}_q)$ yields $F\in L^2_m(X, \mathcal{H}_\infty)$.
We define $\mathcal{F}_{s}(T_\infty;\mathcal{H}_\infty)$ as the linear span of $\mathbf{P}^*_q\iota_q^{-1}\mathcal{F}_{s}(T_q;\mathcal{H}_q)$, $q\in\mathbb{N}$.

\begin{assumption}[$\mathbb{C}$-EXT]\label{assump6}
We say that norms on $\mathcal{F}_{\bullet}(T_q)$, $\bullet=s,w$, $q\in\mathbb{N}$, satisfy \textnormal{($\mathbb{C}$-Ext)} if their pushforward norms under $\iota_q$
satisfy \textnormal{($\mathcal{H}_q$-Ext)} of Assumption \ref{assump5}.
\end{assumption}
We define the norm on $\mathcal{F}_{s}(T_\infty;\mathcal{H}_\infty)$ as the limit (via a supremum and nesting) of pushforward norms for $\mathcal{F}_{\bullet}(T_q;\mathcal{H}_q)$. We let $\mathcal{B}_{s}(T_\infty;\mathcal{H}_\infty)$ we the completion in this norm.
We define $\iota_\infty$ as the limit (via nesting property) of $\mathbf{P}_q^*\mathbf{P}_q\iota_q$ on a dense subspace. 

\begin{lemma}\label{lemma:isofunctions}
Let $q\in\mathbb{N}\cup\left\{infty\right\}$. We have that $\iota_q$ satisfies the following equivariance property: for any $F\in\mathcal{F}_{s,q},G\in L^{p_*}_{m}(X,\mathbb{C}^{\Gamma_q})$, $\iota_q(F)=f\in \mathcal{F}_s(T_q;\mathcal{H}_q), \iota_q(G)=g\in L^p_{\overline{m}_q}(X_q,\mathbb{C})$, we have,
$$
f\circ T_q = \iota_q(\mathbf{T}_q F),\; g\circ T_q = \iota(\mathbf{T}_q G),
$$
and moreover
$$
\int f \, g\circ T^k_q \,d\overline{\mu}_q = \int F \cdot \mathbf{T}^k_q G \,d\mu.
$$
\end{lemma}
\begin{proof}
Let $q\in\mathbb{N}$. To show equivariance we observe that 
$$
(\mathbf{T}_q F)(x) = (\rho_q(\psi(x))F)(Tx),
$$ 
and this maps to $f^\prime(x,g) = F(Tx)\cdot \delta_{\psi(x)^{-1}g}$.
On the other hand $F$ maps to $f(x,g) = F(x)\cdot \delta_g$ and precomposition with $T_q$ yields $f\circ T_q(x,g) = f(Tx,\psi(x)^{-1}g) = F(Tx)\cdot \delta_{\psi(x)^{-1}g}$. Therefore $f^\prime=f\circ T_q$ as required.

For the second part it suffices to check that 
$$
\int fg\, d\mu_q = \int F\cdot G\, d\mu.
$$
To see this, by definition of $\mu_q$ we have
$$
\int fg\, d\mu_q = (\#\Gamma_q)^{-1}\sum_{\gamma\in\Gamma_q} \int f(x,\gamma)g(x,\gamma) \,d\mu(x).
$$
On the other hand  $f(x,\gamma)g(x,\gamma) = (F(x)\cdot \delta_\gamma )(\delta_\gamma \cdot G(x))$ and so 
$$
\sum_{\gamma\in \Gamma_q}  (F(x)\cdot \delta_\gamma )(\delta_\gamma \cdot G(x)) =  (F(x)\cdot \sum_{\gamma\in \Gamma_q}(\delta_\gamma \cdot G(x)) \delta_\gamma ) = F(x)\cdot G(x).
$$
The conclusion follows for any $q\in\mathbb{N}$. The case $q=\infty$ follows by approximating $\iota_\infty,\iota_\infty\circ \mathbf{T}_\infty$ by $\iota_q,\iota_q\circ \mathbf{T}_q$.
\end{proof}

We are now ready to conclude the main theorems. 
\begin{proof}[Proof of Theorem \ref{theorem:one}]
Given $g:X\to \mathbb{C}$, we let $G=\iota_q(g)$ be the vector valued function.
Using the triangle inequality, and noting the normalization for $\mu_q$, we have
$$
\#\Gamma_q\|G\|_{L^1_\mu(X,\mathbb{C}^{\Gamma})}\le \|g\|_{L^1_{\mu_q}(X,\mathbb{C})}.
$$
More generally, for $p_*\in [1,\infty)$, using convexity of $x\mapsto x^{p_*}$ and using Jensen's inequality for the convex function $x\mapsto x^{1/p_*}$, we have, 
$$
(\#\Gamma_q)^{1/p_*}\|G\|_{L^{p_*}_m(X,\mathbb{C}^{\Gamma})}\le \|g\|_{L^{p_*}_m(X,\mathbb{C})}.
$$
On the other hand it is easily observed that
$$
\|G\|_{L^{\infty}_m(X,\mathbb{C}^{\Gamma})}\le \|g\|_{L^{\infty}_m(X,\mathbb{C})}.
$$

We note that since $V_q$ and $\mathcal{H}_q$ are orthogonal, we have that $G = G_1 + G_2$ with $G_1$ taking values in $\mathcal{H}_q$ and moreover 
$$
\|G\|_{L^{p}_m(X,\mathbb{C}^{\Gamma})}\ge \|G\|_{L^{p}_m(X,\mathcal{H}_q)}
$$
for any $p\in [1,\infty]$.

This statement of the theorem now follows from Theorem \ref{theorem:onevector}.
\end{proof}

\begin{proof}[Proof of Theorem \ref{theorem:three}]
We use Theorem \ref{theorem:threevector}.
By the equivariance of $\iota$ for any $f\in\mathcal{F}_s(T_q;\mathcal{H}_q)$ we can choose $F\in \mathcal{B}_{s,q}$ with $\sigma_{T_q}(F)=\sigma_{T_q}(f)$. Moreover, $F$ is a coboundary, i.e.\ $F= \mathbf{T}_q W - W$ for some $W\in \mathcal{B}_{s,q}$, if and only if $f$ is a coboundary, i.e\ $f= w\circ T_q - w$ for some $w\in\mathcal{F}_s(T_q)$ (whence $w$ is necessarily also in $\mathcal{F}_s(T_q;\mathcal{H}_q)$).
\end{proof}

\section{Preliminaries for $L^p(\nu,\mathcal{H}_q)$}
Throughout $\nu$ is an arbitrary probability measure on $X$. 
Recall that the function $F\cdot G\in  L^1(\nu,\mathbb{C})$ is well defined for any $F\in L^p(\nu,\mathcal{H}_q)$ and $G\in L^{p_*}(\nu,\mathcal{H}_q)$. We refer to pointwise representatives for such $F,G$ which are of course only valid up to a set of measure zero.
\begin{lemma}[Cauchy-Schwartz Lemma]\label{lemma:dotupperbound}
Let $q\in\mathbb{N}\cup\left\{\infty\right\}$. 
Let $p,p_*\in [1,\infty]$ with $1/p + 1/p_* = 1$. For $F\in L^p(\nu,\mathcal{H}_q)$ and $G\in L^{p_*}(\nu,\mathcal{H}_q)$ we have
$$
\left|\int F\cdot G \,d\nu \right| \le \|F\|_{L^{p}(\nu,\mathcal{H}_q)}\|G\|_{L^{p_*}(\nu,\mathcal{H}_q)}.
$$
\end{lemma}
\begin{proof}
Recall that $F\cdot G(x)=F(x)\cdot G(x)$ and so the usual Cauchy-Schwartz gives $|(F\cdot G)(x)|\le \|F(x)\|_{\mathcal{H}_q}\|G(x)\|_{\mathcal{H}_q}$, whence
$$
\left|\int F\cdot G \,d\nu \right|\le \int |(F\cdot G)(x)|\, d\nu(x)\le \int \|F(x)\|_{\mathcal{H}_q}\|G(x)\|_{\mathcal{H}_q}\, d\nu(x).
$$
Using H\"{o}lder's inequality for complex valued functions gives
\begin{align*}
&\int \|F(x)\|_{\mathcal{H}_q}\|G(x)\|_{\mathcal{H}_q}\, d\nu(x) 
\\
&\le \|(x\mapsto \|F(x)\|_{\mathcal{H}_q} )\|_{L^{p}(\nu,\mathbb{C})}\|(x\mapsto \|G(x)\|_{\mathcal{H}_q})\|_{L^{p_*}(\nu,\mathbb{C})}.
\end{align*}
\end{proof}

\begin{lemma}\label{lemma:lebesguediff} Let $q\in\mathbb{N}\cup\left\{\infty\right\}$. 
Let $p\in [1,\infty)$ and let $p_*$ with $1/p + 1/p_* = 1$. For any $F\in L^p (\nu,\mathcal{H}_q)$ there are unit norm $G\in L^{p_*}(\nu,\mathcal{H}_q)$ with
$$
\int F\cdot G \,d\mu = \|F\|_{L^{p}(\nu,\mathcal{H}_q)}\|G\|_{L^{p_*}(\nu,\mathcal{H}_q)}.
$$
For $p=\infty$, for any $F\in L^p(\nu,\mathcal{H}_q)$ there are unit norm $G_n\in L^{1}(\nu,\mathcal{H}_q)$ with
$$
\lim_{n\to\infty} \int F\cdot G_n \,d\nu  =\|F\|_{L^{\infty}(\nu,\mathcal{H}_q)}.
$$
\end{lemma}
\begin{proof}
Suppose $p=1$.
Choose $G(x)= \|F(x)\|^{-1}_{\mathcal{H}_q}  F(x)$. Then 
$$
\|x \mapsto \|G(x)\|_{\mathcal{H}_q} \|^p_{L^{\infty}(\nu,\mathbb{C})}=1,
$$ 
and
$$
\int F\cdot G \,d\mu  = \int \|F(x)\|_{\mathcal{H}_q} \,d\nu(x) = \|x \mapsto \|F(x)\|_{\mathcal{H}_q} \|_{L^{1}(\nu,\mathbb{C})}.
$$
Suppose $p\in(1,\infty)$. Note the identity
$$
1/p + 1/p_* = 1 \iff p_*/p + 1 = p_* \iff p_* + p = p_*p. 
$$
Choose $G(x)= \|F(x)\|^{-(2-p)}_{\mathcal{H}_q}  F(x)$. We claim $G\in L^{p_*}(\nu,\mathcal{H}_q)$. To see this it suffices to note the indentity,
$$
\|G(x)\|_{\mathcal{H}_q}^{p_*} = \|F(x)\|_{\mathcal{H}_q}^{p_*}\|F(x)\|^{-p_*(2-p)}_{\mathcal{H}_q} =  \|F(x)\|_{\mathcal{H}_q}^{-p_*}\|F(x)\|^{p_*p}_{\mathcal{H}_q}= \|F(x)\|^{p}_{\mathcal{H}_q}.
$$
The conclusion follows from the identity
$$
F(x)\cdot G(x) =  \|F(x)\|^2_{\mathcal{H}_q}\|F(x)\|^{-(2-p)}_{\mathcal{H}_q} =\|F(x)\|^{p}_{\mathcal{H}_q}.
$$
Suppose $p=\infty$. Let $F$ be given with $\|F\|_{L^\infty(\nu,\mathbb{C})}=1$, and let $O_n$ be a positive measure set with $\|F(x)\|_{\mathcal{H}_q}\ge (1-1/n)$ for $x\in O_n$. Set $G(x)=F(x)/\nu(O_n)$ on $O_n$. Then
$$
\int \|G(x)\|_{\mathcal{H}_q} \,d\nu(x) \ge (1-1/n),
$$
and
$$
\int F\cdot G(x)\,d\nu(x)= \mu(O_n)^{-1}\int_{O_n} \|F(x)\|^2_{\mathcal{H}_q} \,d\nu(x) \ge (1-1/n)^2.
$$
\end{proof}

\section{Duality results for operators and integration}\label{section:duality}
In section \ref{section:vectorspaces} we introduced the transfer operator $\mathbf{L}$ as dual to the Koopman operator for the dynamics and deduced a pointwise formula for it. We need to check the pointwise formula and to recover the duality property for our constructed $\mathbf{L}_q$. We collect some natural results concerning the continuity of integration in the respective Banach spaces.

In order to show the duality for $\mathbf{L}_q,\mathbf{T}_q$ we need to make note of certain ingredients behind the duality property for $\mathbf{L},\mathbf{T}$. For clarity recall that our starting point that $\mu$ is assumed to be a $T$-invariant probability measure, absolutely continuous with respect to $m$, and $r$ is defined on each $[i]\cap T^{-1}[j]$ as the Radon-Nikodym derivative of $m_{|[i]\cap T^{-1}[j]}$ with respect to $(T^\dag_{i,j})_*m_{|[i]}$ which is assumed to have $r,r^{-1}\in L^\infty(m,\mathbb{C})$.
\begin{lemma}\label{lemma:duality}
Let $p,p_*\in [1,\infty]$ with $1/p + 1/p_* = 1$. 
For any $f\in L^p(m,\mathbb{C})$ and $g\in L^{p_*} (m,\mathbb{C})$ we have
\begin{equation}\label{eq:dualitytoep}
\int f\, \mathbf{T}(i,j)(g) dm  = \int \mathbf{M}(r(i,j))\mathbf{T}^\dag(i,j) (f) \,g \,dm.
\end{equation}
More generally, we have
$$
\int I\otimes\bm{\rho}_q^*(i,j)(F)\cdot \mathbf{T}(i,j)\otimes I(G) \,dm = \int (\mathbf{M}(r(i,j))\mathbf{T}^\dag(i,j)\otimes \bm{\rho}(i,j)_q^*(F)\cdot G \,dm
$$
\end{lemma}
\begin{proof}
We have 
$$
\int (\mathbf{M}(r(i,j))\mathbf{T}^\dag(i,j)) (f) \,g \,dm=\int r(T^\dag_{i,j} x) f(T^\dag_{i,j} x) g(T_{i,j}T^\dag_{i,j}x) \chi^\dag(i,j)(x)\,dm(x).
$$
We observe that 
\begin{align*}
&\int r(T^\dag_{i,j} x) f(T^\dag_{i,j} x) g(T_{i,j}T^\dag_{i,j}x) \chi^\dag(i,j)(x)dm(x) 
\\
&= \int r(y)f(y) g(T_{i,j} y)\,d(T^\dag_{i,j})_*m_{|T([i]\cap T^{-1}[j])}(y).
\end{align*}
Recalling the definition of $r$ as a Radon-Nikodym derivative, we have
$$
\int h(y)r(y) \,d(T^\dag_{i,j})_*m_{|T([i]\cap T^{-1}[j])}(y) = \int h(y) \,dm_{|[i]\cap T^{-1}[j]},
$$
for any measurable $h$. Therefore
$$
\int (\mathbf{M}(r(i,j))\mathbf{T}^\dag(i,j)) (f) \,g \,dm = \int f(y) g(T_{i,j} y)\,dm_{|[i]\cap T^{-1}[j]}(y).
$$
Recalling the definition of $\chi(i,j)$ we have
$$
\int f(y) g(T_{i,j} y)\,dm_{|[i]\cap T^{-1}[j]}(y) = \int f(y) g(T_{i,j} y)\chi(i,j) \,dm(y), 
$$
which is easily seen to equals the left hand side of \ref{eq:dualitytoep}.

We proceed with the second part in the case that $q<\infty$. We observe
$$
\int F\cdot (\mathbf{T}(i,j)\otimes \bm{\rho}_q(i,j))(G) \,dm = \int (I\otimes \bm{\rho}_q^*(i,j))(F)\cdot \mathbf{T}(i,j)\otimes I(G) \,dm.
$$
Recalling that $I\otimes \bm{\rho}^*(i,j)(F)=F^\prime$ and $G$ are the linear sum of $f^\prime_\gamma, g_\gamma$, we can use \ref{eq:dualitytoep} on each $f_{\gamma_1}^\prime (x)\mathbf{T}(i,j)(g_{\gamma_2})(x)$ to get
$$
\int I\otimes \bm{\rho}_q^*(i,j)(F)\cdot \mathbf{T}(i,j)\otimes I(G) \,dm = \int (\mathbf{M}(r(i,j))\mathbf{T}^\dag(i,j)\otimes \bm{\rho}(i,j)^*_q(F)\cdot G \,dm.
$$
The case $q=\infty$ follows by continuity of the operators.
\end{proof}

We use the formula to deduce Lemma \ref{cor:measure1}.
\begin{proof}[Proof of Lemma \ref{cor:measure1}]
Identity (\ref{eq:dualitytoep}) implies the duality relation for $\mathbf{L},\mathbf{T}$ on $L^p(m,\mathbb{C})$ pairs.
We note that 
\begin{align*}
\int r(i_0,\ldots, i_{n}) \,\mathds{1}_{T^n B} \,d\mu &= \int r(i_0,\ldots, i_{n}) \,\mathds{1}_{T^n B} \, h \,dm
\\
&= \int \mathbf{L}^n( \mathds{1}_{B} \,h\circ T) \,\mathds{1}_{T^n B} \,dm
\\
&= \int  \mathds{1}_{B} \,h \,\mathds{1}_{T^n B}\circ T \,dm = \int \mathds{1}_{B} \,d\mu. 
\end{align*}

For the second, 
\begin{align*}
\int r^{-1}(i_0,\ldots, i_{n})\circ T\,\mathds{1}_{ B} \,h\, dm &= \int r^{-1}(i_0,\ldots, i_{n}) \,\mathbf{L}^n (\mathds{1}_{ B}\, h)\,dm
\\
&= \int \mathds{1}_{T^nB}\, h\circ  (T^\dag_{i_{n-1},i_n}\circ \cdots T^\dag_{i_{0},i_1}) \,dm 
\end{align*}
\end{proof}

\begin{lemma}[Duality Lemma with respect to $m$]\label{lemma:duality} 
Let $q\in\mathbb{N}\cup\left\{\infty\right\}$. 
Let $p$ be given and let $p_*$ be conjugate in the sense that $1/p + 1/p_* = 1$. 
For any $F\in L^{p}(m,\mathcal{H}_q)$, and any $G\in L^{p_*}(m,\mathcal{H}_q)$ we have
$$
\int \mathbf{L}_q(F)\cdot G \,dm = \int F \cdot \mathbf{T}_q(G) \,dm.
$$
\end{lemma}
\begin{proof}
Taking the identity
$$
\int I\otimes \bm{\rho}^*(i,j)(F)\cdot \mathbf{T}(i,j)\otimes I(G) \,dm = \int (\mathbf{M}(r(i,j))\mathbf{T}^\dag(i,j)\otimes \bm{\rho}(i,j)^*(F)\cdot G \,dm
$$
and summing over $i,j=1,\ldots, k$ yields the result. 
\end{proof}

\begin{corollary}[$\mathbf{T}_q$-invariance of integration with respect to $\mu$] 
Let $q\in\mathbb{N}\cup\left\{\infty\right\}$. 
Let $p$ be given and let $p_*$ be conjugate in the sense that $1/p + 1/p_* = 1$. 
For any $F\in L^{p}(\mu,\mathcal{H}_q)$, and any $G\in L^{p_*}(\mu,\mathcal{H}_q)$ we have
$$
\int \mathbf{T}_q(F)\cdot \mathbf{T}_q(G) \,d\mu = \int F \cdot G \,d\mu.
$$
\end{corollary}
\begin{proof}
Once we check the identity $\mathbf{L}_q \mathbf{M}(h)\mathbf{T}_q=\mathbf{M}(h)$ in Lemma \ref{lemma:triv} we simply take $F^\prime= \mathbf{M}(h)\mathbf{T}_q(F)$ whence $ \mathbf{L}_qF^\prime=\mathbf{M}(h)F$.
\end{proof}

\section{Algebra identities}\label{section:algebra}
We begin with the following. We use the notation $\delta_{x,y}$ for the value $0$ when $x\ne y$ and $1$ when $x=y$.
\begin{proposition}\label{prop:toeplitz1} 
For any $i_0,\ldots, i_n, j_0,\ldots, j_n$ we have
$$
\mathbf{T}^\dag(i_0,\ldots,i_n)\bm{\chi}(j_0,\ldots, j_n) = \delta_{(i_0,\ldots,i_n),(j_0,\ldots, j_n)}\bm{\chi}^\dag(i_0,\ldots,i_n)\mathbf{T}^\dag(i_0,\ldots,i_n),
$$
$$
\mathbf{T}^\dag(i_0,\ldots,i_n)\mathbf{T}(j_0,\ldots, j_n) = \delta_{(i_0,\ldots,i_n),(j_0,\ldots, j_n)}\bm{\chi}^\dag(i_0,\ldots,i_n).
$$
\end{proposition}
\begin{proof}
First we show that
$$
\mathbf{T}^\dag(i,j)\bm\chi(k,l) = \delta_{(i,j),(k,l)}\chi^\dag(i,j)\mathbf{T}^\dag(i,j).
$$
We have
$$
\mathbf{T}^\dag(i,j)\bm\chi(k,l)f(x) = \chi(k,l)(T_{i,j}^\dag x)f(T_{i,j}^\dag x)\chi^\dag(i,j)(x).
$$
If $\chi^\dag(i,j)(x)=1$ then $x\in T([i]\cap T^{-1}[j])\subseteq T([i])\cap [j]$. If $\chi(k,l)(T_{i,j}^\dag x)=1$ then $y=T_{i,j}^\dag x$ has $y\in [k]\cap T^{-1}[l]$, whence $x=Ty\in T([k]\cap T^{-1}[l])\subseteq T([k])\cap [l]$. This forces $j=l$. On the other hand $y=T_{i,j}^\dag x$ is always an element of $[i]\cap T^{-1}[j]$. So $i=k$.

Now we argue that 
\begin{align*}
&\bm{\chi}^\dag(i_0,\ldots, i_n)\mathbf{T}^\dag(i_0,\ldots, i_n)\bm\chi(j_0,\ldots, j_n) 
\\
&= \delta_{(i_0,\ldots, i_n),(j_0,\ldots, j_n)}\bm{\chi}^\dag(i_0,\ldots, i_n)\mathbf{T}^\dag(i_0,\ldots, i_n).
\end{align*}
We have
$$
\mathbf{T}^\dag(i_0,\ldots, i_n) = \mathbf{T}^\dag(i_1,\ldots, i_n)\mathbf{T}^\dag(i_0,i_1),
$$
$$
\bm{\chi} (i_0,\ldots, i_n)=\mathbf{M}(\chi(i_1,\ldots, i_n)\circ T^{n-1}) \bm{\chi}(i_0,i_1)
$$
Then
$$
\mathbf{T}^\dag(i_0,i_1)\mathbf{M}(\chi(i_1,\ldots, i_n)\circ T^{n-1}) \bm{\chi}(i_0,i_1) =\mathbf{M}(\chi(i_1,\ldots, i_n))\mathbf{T}^\dag(i_0,i_1).
$$
\end{proof}

\begin{lemma}\label{lemma:triv}
We have the identity $\mathbf{L}_q\,\mathbf{M}(h)\otimes I\, \mathbf{T}_q = \mathbf{M}(h) \otimes I$.
\end{lemma}
\begin{proof}
Recall that 
$$
\mathbf{L}_q = \sum_{i,j=1}^k \mathbf{T}^\dag(i,j) \mathbf{M}(r)\otimes \bm{\rho}_q^*(i,j), \;\; \mathbf{T}_q = \sum_{a,b=1}^k\mathbf{T}(a,b)\otimes \bm{\rho}_q(a,b).
$$ 
We therefore have, using that $I\otimes \bm{\rho}_q^*(i,j)$ (respectively, $\bm{\rho}_q(a,b)$) commutes with any $\mathbf{T}^\dag(i,j) \mathbf{M}(r)\otimes I$ (respectively, $\mathbf{T}(a,b)\otimes I$),
$$
\mathbf{L}_q \,\mathbf{M}(h)\otimes I\,  \mathbf{T}_q =  \sum_{i,j=1}^k \mathbf{T}^\dag(i,j) \mathbf{M}(r)\mathbf{M}(h) \mathbf{T}(a,b) \otimes \bm{\rho}^*(i,j) \sum_{a,b=1}^k  \bm{\rho}(a,b).
$$
It is easily seen that $\mathbf{T}^\dag(i,j) \mathbf{M}(r\,h)=\mathbf{M}(r(i,j))\mathbf{M}(h(i,j)) \mathbf{T}^\dag(i,j)$ whence
$$
\mathbf{T}^\dag(i,j) \mathbf{M}(r\,h)\mathbf{T}(a,b) = \mathbf{M}(r(i,j)h(i,j))\mathbf{T}^\dag(i,j)\mathbf{T}(a,b),
$$
which is non-zero only when $i=a,j=b$. Therefore
$$
\mathbf{L}_q \mathbf{T}_q =  \sum_{i,j=1}^k \mathbf{M}(r(i,j)h(i,j))\bm{\chi}^\dag(i,j)\otimes \bm{\rho}^*(i,j) \bm{\rho}(i,j) .
$$
The right hand side simplifies to $\mathbf{M}\left(\sum_{j=1}^k\sum_{i=1}^kr(i,j)h(i,j)\right)$. The term appearing inside the brackets is the function $\mathbf{L}h$ which we know equals $h$.
\end{proof}

The following identities are easily seen for any $g\in L^\infty_\mu(X,\mathbb{C})$ and $k\in\mathbb{N}_0$,
\begin{equation}\label{TLids}
\mathbf{M}(g)\otimes I\, \mathbf{L}_q^k  = \mathbf{L}_q^k \,\mathbf{M}(g\circ T^k)\otimes I,\;
\mathbf{T}_q^k \,\mathbf{M}(g) \otimes I= \mathbf{M}(g\circ T^k)\otimes I\,\mathbf{T}_q^k. 
\end{equation}
And further that for any $i_0,\ldots, i_n$ and any $f\in \mathcal{F}_s$ there is $f^\prime\in \mathcal{F}_s$ (depending on $i_0,\ldots, i_n$) with
\begin{equation}\label{Tdagids2}
 \mathbf{T}^\dag(i_0,\ldots, i_n)\mathbf{M}(f)  =\mathbf{M}(f^\prime)  \mathbf{T}^\dag(i_0,\ldots, i_n).
\end{equation}

\begin{proposition}\label{prop:alg2}
The closure of $\left\langle \mathrm{Mult}(\mathcal{F}_s)\otimes \rho_q(\Gamma),\mathbf{L}_q\right\rangle_{s,q,s,q}$ in $\|\cdot\|_{s,q\to w,q}$ contains $\left\langle \mathrm{Mult}(\overline{\mathcal{A}}_{L^\infty})\otimes \rho_q(\Gamma),\mathbf{L}_q\right\rangle_{s,q,w,q}$. 
\end{proposition}
\begin{proof}
From Assumption \ref{assump3} it follows that
 $$
\|\mathbf{M}(f)-\mathbf{M}(g_i)\|_{s\to w}\le \|f-g_i\|_{w}.
 $$
 The conclusion for the tensor products follows since $\|\cdot\|_{s,q},\|\cdot\|_{w,q}$ are reasonable cross norms. 
\end{proof}

\begin{proposition}\label{prop:alg1}
We have that $\left\langle \mathrm{Mult}(\mathcal{F}_s)\otimes \rho_q(\Gamma),\mathbf{L}_q \right\rangle_{s,q,s,q}$ is contained in the sub-vector space of $\mathrm{Op}(\mathcal{B}_{s,q},\mathcal{B}_{s,q})$ given by
$$
\mathrm{span}\left(\mathrm{Mult}(\mathcal{F}_s)\otimes \rho_q(\Gamma), \left\{\mathbf{T}^\dag(i_0,\ldots, i_m)\otimes I: i_j=1,\ldots,k, m\in\mathbb{N}\right\}\right),
$$
and moreover their closures in $\|\cdot\|_{s,q\to w,q}$ coincide. 
\end{proposition}

\begin{proof}
The containment follows from checking that the linear span is indeed an algebra: this follows from the rule (\ref{Tdagids2}) that
$$
 \mathbf{M}(f_1)\mathbf{T}^\dag(i_0,\ldots, i_m) \mathbf{M}(f_2)\mathbf{T}^\dag(j_0,\ldots, j_n) = \mathbf{M}(f_1 f_2^{\prime})\mathbf{T}^\dag(j_0,\ldots, j_n,i_0,\ldots, i_m).
$$
The claim regarding their closure follows upon observing that $\mathbf{T}^\dag(i,j)\otimes I$ is the $\|\cdot\|_{s,q\to w,q}$ limit of $\mathbf{L}_q \,\mathbf{M}(f_n)\otimes I$ with $f_n\to \chi(i,j)$ in $\mathcal{B}_w$.
\end{proof}


\section{The case of a Subshift of Finite Type and Lipschitz functions}\label{section:ssft}
The goal of this section is the following.
\begin{proposition}\label{prop:exampleA}
Let $(X,T,\mu)$, $\left\{ [1],\ldots, [k]\right\}$ and $\mathcal{F}_w:=\mathcal{F}_w(T), \mathcal{F}_s:=\mathcal{F}_s(T)$ be the Banach spaces given in Definition \ref{exampleA}. These satisfy \textnormal{(Base)} with $p=\infty$. The norms $\|\cdot\|_{s,q},\|\cdot\|_{w,q}$ given in Definition \ref{def:exampleA2} satisfies  \textnormal{($\mathbb{C}$-Ext)}.
\end{proposition}
The verifications of the assumptions pertaining to the base dynamics are standard. The main work is to show that the vector valued Lipschitz semi-norm coincides with an injective tensor semi-norm, and that one has the following \emph{branch-wise Doeblin--Fortet inequality}: there are $\alpha<1$ and $C>0$ such that for arbitrary $j_0,\ldots, j_m$ and arbitrary $f\in\mathcal{F}_s$ we have
$$
\|\mathbf{T}^\dag(j_{m-1},j_m)\mathbf{M}(r) \cdots \mathbf{T}^\dag(j_0,j_1)\mathbf{M}(r) (f)\|_s\le C\mu([j_0\ldots, j_m])( \alpha^n \|f\|_s + \|f\|_{w}),
$$
where $[j_0,\ldots, j_m]:=[j_0]\cap T^{-1}[j_1]\cap\cdots \cap T^{-m}[j_m]$. 

Recall that in Definition \ref{exampleA} we have $\mu=m$. We recall that there is an easy isomorphism between functions on $X_q$ and vector valued functions on $X$. In the case of a subshift of finite type we have a natural notion of Lipschitz $\mathcal{H}_q$-valued functions and we translated this into a norm on $X\times \Gamma_q$ valued functions in \ref{eq:lipschitzgpext}.
\begin{proposition}\label{prop:lipequivalent}
The norm $\|f\|_{\mathcal{F}_w(T_q)}$ for $f\in\mathcal{F}_w(T_q;\mathcal{H}_q)$ coincides with $\|\iota(f)\|_{\epsilon,w,q}$.
The norm $|f|_{L,\Gamma_q} + \|f\|_{\mathcal{F}_w(T_q)}$ for $f\in\mathcal{F}_s(T_q;\mathcal{H}_q)$ coincides with $\|\iota(f)\|_{\epsilon,s,q}$. 
\end{proposition}
\begin{proof}
First we evaluate the pushforward norms as
$$
|F|_{L^\infty,\mathbb{C}^{\Gamma_q}}:= \sup_{x\in X} \|F(x)\|_{\mathbb{C}^{\Gamma_q}},
$$
$$
|F|_{L,\mathbb{C}^{\Gamma_q}}:= \sup_{x\ne y} \|F(x)-F(y)\|_{\mathbb{C}^{\Gamma_q}}/ d_{\Sigma^+}(x,y).
$$
Since we only see $\mathcal{H}_q$ valued functions we have
$$
\|F(x)\|_{\mathbb{C}^{\Gamma_q}}=\|F(x)\|_{\mathcal{H}_q},
$$
$$
\|F(x)-F(y)\|_{\mathbb{C}^{\Gamma_q}}= \|F(x)-F(y)\|_{\mathcal{H}_q},
$$
in the previous. It can be checked that 
$$
 \sup_{x\in X} \|F(x)\|_{\mathbb{C}^{\Gamma_q}} = \sup_{v\in\mathcal{H}_q, \|v\|_{\mathcal{H}_q}} \|F\cdot v\|_{L^\infty}
$$
which agrees with the form of the injective norm we have given. The fact that the Lipschitz semi-norm coincides with the injective semi-norm is also standard.
\end{proof}

Our reference for the following is book of Parry and Pollicott \cite{PP}. 
\begin{proof}[Proof of Proposition \ref{prop:exampleA}] Recall that \textnormal{(Base)} is composed of Standing assumption \ref{assump1}, \ref{assump2}, \ref{assump3}.

(\textit{Standing assumption \ref{assump1}:}) It is given that $\alpha$ is the partition into cylinder sets, and indeed this is a strong generator, and the weak covering property follows from transitivity of the map. Moreover in this case one has $T([i]\cap T^{-1}[j]) = [j]$ provided $[i]\cap T^{-1}[j]\ne \emptyset$. The inverse branches $T^\dag_{i,j}$ are defined on $[j]$ by $(x_0,x_1,\ldots)\mapsto (i,x_0,x_1,\ldots)$. We assume that $\mu$ is an equilibrium measure for a Lipshcitz continuous function $h:\Sigma^+\to\mathbb{R}$. We may assume that the pressure of $h$ is zero, whence the Radon-Nikodym derivative $r$ is precisely $\log h$, which is indeed bounded.

(\textit{Standing assumption \ref{assump2}:})
We observe that $\mathcal{A}$ are continuous and their closure in $\|\cdot\|_{L^\infty}$ coincides with $C(X,\mathbb{C})$ equipped with the supremum norm on continuous functions. Hence $\mathcal{B}_w = (C(X,\mathbb{C}),\|\cdot\|_{L^\infty})$ satisfies the standing assumption. 

(\textit{Standing assumption \ref{assump3}:})
We take $\mathcal{B}_s$ the complex Lipschitz functions equipped with norm $\|\cdot\|_s=|\cdot |_L+\|\cdot\|_\infty$. It is standard that $\mathcal{B}_s$ is complete and separable, that the unit ball is pre-compact in the space of continuous functions, and that the complex Lipschitz functions are dense in $C(X,\mathbb{C})$ with respect to the supremum norm. It follows easily that the inverse branches preserve $\mathcal{B}_s$. Recalling that $r$ is the exponential of a Lipschitz function it follows that $r,r^{-1}$ are again Lipschitz. We have that $\chi^\dag(i,j)$ is the indicator function on $[j]$ which is indeed Lipschitz. Products of Lipschitz functions are Lipschitz, and the triangle inequality gives that
$$
|fg|_L \le |f|_L\|g\|_\infty + |g|_L\|f\|_\infty,
$$
which simplifies the the desired inequality $\|fg\|_s \le 2\|f\|_s\|g\|_s$ upon noting that by contruction each of $|\cdot |_L,\|\cdot\|_\infty$ is a lower bound for $\|\cdot \|_s$.

To check \textnormal{($\mathbb{C}$-Base)} it suffices to check \textnormal{($\mathbb{H}_q$-Base)} for the norms given in Prop \ref{prop:lipequivalent}.  

\noindent(\textit{Form of weak norm:}) Follows from the injective cross-norm. 

\noindent(\textit{Uniform Doeblin--Fortet inequality:}) Let $n\in\mathbb{N}$ be given. We have
\begin{align*}
&\| \bm{\chi}(i_0,i)\otimes I\, \mathbf{L}_q^n (F)\|_{\epsilon,s,q}
\\
& \le  \sum_{j_0,\ldots, j_{n-1}} \| \mathbf{T}^\dag(j_0,\ldots, j_{n-1},i_0) \mathbf{M}(r) \otimes \bm{\rho}^*(j_0,\ldots, j_{n-1},i_0)(F)\|_{\epsilon, s,q}.
\end{align*}
Using the definition of the norm in addition to the unitary property of the representation $\rho$ we have
\begin{align*}
& \| \mathbf{T}^\dag(j_0,\ldots, j_{n-1},i_0) \mathbf{M}(r) \otimes \bm{\rho}^*(j_0,\ldots, j_{n-1},i_0)F)\|_{\epsilon,s,q} 
\\
&= \sup_{v\in\mathcal{H}_q}\| (\mathbf{T}^\dag(j_0,\ldots, j_{n-1},i_0) \mathbf{M}(r) \otimes \bm{\rho}^*(j_0,\ldots, j_{n-1},i_0)F)\cdot v\|_{s}
 \\
 &=\sup_{v^\prime\in\mathcal{H}_q}\| \mathbf{T}^\dag(j_0,\ldots, j_{n-1},i_0) \mathbf{M}(r)(F\cdot v^\prime)\|_{s}
 \\
 &\le
C\mu([j_0,\ldots, j_{n-1}])\alpha^n( \sup_{v^\prime\in\mathcal{H}_q}\|F\cdot v^\prime\|_{s} + \sup_{v^\prime\in\mathcal{H}_q}\| F\cdot v^\prime\|_{w})
\\
&=
C\mu([j_0,\ldots, j_{n-1}])\alpha^n(\|(F\|_{\epsilon,s,q} + \| F\|_{\epsilon,w,q})
\end{align*}
Summing over $j_0,\ldots, j_{n-1}$ we deduce that
$$
\|\mathbf{L}_q^n (F)\|_{\epsilon,s,q}\le C\alpha^n\|F\|_{\epsilon,s,q} + C \|F\|_{\epsilon,w,q},
$$
as required.

\noindent(\textit{Closeness criterion:}) Let $i_0,\ldots, i_{k-1}, j_k,\ldots, j_n$.
We have, if $j_n= i_{k-1}$,
$$
\| \mathbf{d}(i_0,\ldots, i_{k-1};j_k\ldots, j_n)\otimes I\, F\|_{w,q} \le \theta^{k-1}\|F\|_{L,q}.
$$
Including the term $\mathbf{M}(g)$ increases the estimate by at most $\|g\chi(i_{k-1})\|_{L^\infty(\mu,\mathcal{H}_q)}$ as required.

If $j_n \ne  i_{k-1}$ then the supremum norm is attained either by
$$
\mathbf{M}(g)\bm{\chi}^\dag(i_0,\ldots,i_{k-1})\mathbf{T}^\dag(i_0,\ldots,i_{k-1})\otimes I \,F
$$
or
$$
\mathbf{M}(g)\bm{\chi}^\dag(i_0,\ldots,i_{k-1}j_k,\ldots, j_n)\mathbf{T}^\dag(i_0,\ldots,i_{k-1}j_k,\ldots, j_n)\otimes I\,F.
$$
Each is bounded by the supremum norm of $g$ on $[i_{k-1}]$ and $[j_n]$ respectively, the largest of which coincides with
$$
\|g(\chi^\dag(i_0,\ldots, i_{k-1})-\chi^\dag(i_0,\ldots, i_{k-1}j_k,\ldots, j_n))\|_{L^\infty(\mu,\mathbb{C})}.
$$
\end{proof}


\section{The case of an expanding interval map and BV functions}\label{section:intervalmap}
The goal of this section is the following.
\begin{proposition}\label{prop:exampleB}
Let $(X,T,\mu)$, $\left\{ [1],\ldots, [k]\right\}$ and $\mathcal{F}_w:=\mathcal{F}_w(T), \mathcal{F}_s:=\mathcal{F}_s(T)$ be the Banach spaces given in Definition \ref{exampleB} and let $m$ be Lebesgue measure. These satisfy \textnormal{(Base)} with $p=1$. The norms $\|\cdot\|_{s,q},\|\cdot\|_{w,q}$ given in Definition \ref{def:exampleB2} satisfies \textnormal{($\mathbb{C}$-Ext)}.
\end{proposition}
We recall that there is an easy isomorphism between functions on $X_q$ and vector valued functions on $X$. In the case of expanding interval maps we have a natural notion of BV $\mathcal{H}_q$-valued functions and we translated this into a norm on $X\times \Gamma_q$ valued functions in \ref{eq:bvgpext}.
\begin{proposition}\label{prop:bvequivalent}
The norm $\|f\|_{\mathcal{F}_w(T_q)}$ for $f\in\mathcal{F}_w(T_q;\mathcal{H}_q)$ coincides with $\|\iota(f)\|_{\pi,w,q}$. 
The norm $|f|_{\mathrm{BV},\Gamma_q} + \|f\|_{\mathcal{F}_w(T_q)}$ for $f\in\mathcal{F}_s(T_q;\mathcal{H}_q)$ coincides with $\|\iota(f)\|_{\pi,s,q}$. 
\end{proposition}
\begin{proof}
First we evaluate that the pushforward norm is given as follows.
We define $|F|_{\mathrm{BV},\mathbb{C}^{\Gamma_q}}$ by
$$
|F|_{\mathrm{BV},\mathbb{C}^{\Gamma_q}} = \sup_{0=x_0<\cdots< x_{m+1}=1}\sum_{i=1}^m \| F(z_i) -F(z_{i+1})\|_{\mathbb{C}^{\Gamma_q}}
$$
We define
$$
\|F\|_{w,q}= \int \|F(x)\|_{\mathbb{C}^{\Gamma_q}}\,d\mathrm{Leb}.
$$
We notice that since we consider functions taking values in $\mathcal{H}_q$ then it follows that in the previous
$$
\| F(z_i) -F(z_{i+1})\|_{\mathbb{C}^{\Gamma_q}}=\| F(z_i) -F(z_{i+1})\|_{\mathcal{H}_q},
$$
$$
 \|F(x)\|_{\mathbb{C}^{\Gamma_q}}= \|F(x)\|_{\mathcal{H}_q}.
$$
The fact that the $L^1(X,\mathcal{H}_q)$ norm coincides with the projective tensor norm (noting that $\mathcal{H}_q$ is finite dimensional) is standard. Similarly the fact that the vector valued BV norm coincides with the projective tensor norm is also standard.
\end{proof}

We use the article of Keller and Hofbauer \cite{KellerH} as a general reference, but one should look at the historical attributions therein. We directly use the predecessor work of Wong \cite{Wong}.
\begin{proof}[Proof of Prop \ref{prop:exampleB}] Recall that \textnormal{(Base)} is composed of Standing assumption \ref{assump1}, \ref{assump2}, \ref{assump3}.
(\textit{Standing assumption \ref{assump1}:}) It is given that elements of $\alpha$ are the closures of intervals of monotonicity for $T$, and hence form a partition modulo end-points which have Lebesgue and hence $\mu$ measure zero. The strong generator property is observed in \cite{KellerH}. The weak covering property is an assumption we make on $T$.

 We take $T_{i,j}^{\dag}$ the continuous extension of the inverse branch defined on the interior of the given interval. 
 We denote the ACIP density by $h$ in what follows.
 The definition of $r$ in addition to properties of Radon-Nikodym derivatives gives
\begin{align*}
&\frac{d\mu_{|[i]\cap T^{-1}[j]}}{d(T_{i,j}^\dag)_* \mu_{|T([i]\cap T^{-1}[j])} }
 \\
 &= \frac{d\mu_{|[i]\cap T^{-1}[j]}}{d\mathrm{Leb}_{|[i]\cap T^{-1}[j]}}
 \frac{d\mathrm{Leb}_{|[i]\cap T^{-1}[j]}}{d(T_{i,j}^\dag)_*\mathrm{Leb}_{|T([i]\cap T^{-1}[j])}}
 \frac{d(T_{i,j}^\dag)_*\mathrm{Leb}_{|T([i]\cap T^{-1}[j])}}{d(T_{i,j}^\dag)_* \mu_{|T([i]\cap T^{-1}[j])}}.
\end{align*}
 We have
\begin{align*}
& \frac{d\mu_{|[i]\cap T^{-1}[j]}}{d\mathrm{Leb}_{|[i]\cap T^{-1}[j]}} = h,\; 
 \frac{d\mathrm{Leb}_{|[i]\cap T^{-1}[j]}}{d(T_{i,j}^\dag)_*\mathrm{Leb}_{|T([i]\cap T^{-1}[j])}}=\frac{1}{T^\prime_{i,j}},
 \\
 &
 \frac{d(T_{i,j}^\dag)_*\mathrm{Leb}_{|T([i]\cap T^{-1}[j])}}{d(T_{i,j}^\dag)_* \mu_{|T([i]\cap T^{-1}[j])}}= \frac{1}{h\circ T_{i,j}}.
\end{align*}
Therefore $r$ is indeed bounded since $h$ is bounded away from $0$. Notice that, since $1/T^\prime\circ T^\dag_{i,j} = (T^\dag_{i,j} )^\prime$, we have
$$
r(i,j)=\frac{h\circ T_{i,j}^\dag}{h}\chi^\dag(i,j)(T^\dag_{i,j} )^\prime = \frac{h\circ T_{i,j}^\dag}{h}\frac{1}{T^\prime\circ T^\dag_{i,j} }
$$ 
and more generally, 
\begin{align*}
r(i_0,\ldots, i_n)&= \frac{h\circ T^\dag_{i_0,i_1}\circ \cdots \circ T_{i_{n-1},i_{n}}^\dag}{h}\chi^\dag(i_0,\ldots, i_n)(T^\dag_{i_{0},i_1}\circ \cdots \circ T^\dag_{i_{n-1},i_n} )^\prime \\
&=  \frac{h\circ T^\dag_{i_0,i_1}\circ \cdots \circ T_{i_{n-1},i_{n}}^\dag}{h} \frac{1}{(T^n)^\prime\circ T^\dag_{i_0,i_1}\circ \cdots \circ T_{i_{n-1},i_{n}}^\dag}.
\end{align*} 

(\textit{Standing assumption \ref{assump2}:}) We are given that $(\mathcal{F}_w,\|\cdot\|_{w})$ is $(L^1_\mu(X,\mathbb{C}),\|\cdot\|_{L^1_\mu(X,\mathbb{C})})$ which is one of the cases allowed in the assumption.

(\textit{Standing assumption \ref{assump3}:}) We take $\mathcal{F}_s$ the ($L^1$ equivalence class of) bounded variation (BV) functions equipped, and norm $\|\cdot\|_s=|\cdot |_{\mathrm{BV}}+\|\cdot\|_{L^1}$ (with the infimum taken over $L^1$ equivalence classes). It is standard that $\mathcal{B}_s$ is complete and separable, and that $\mathcal{F}_s$ is dense in $L^1_\mu(X,\mathbb{C})$. That the inverse branches preserve $\mathcal{B}_s$ is argued in the next part.

Recalling that $h$ is bounded variation, and that the derivative of $T$ is BV, it follows that $h\circ T$ is BV. Since $h$ is bounded from below it follows that $1/h$ is BV. Therefore $r,r^{-1}$ are again BV (since products are BV). We have that $\chi^\dag(i,j)$ is the indicator function on an interval which is indeed BV. 

Products of functions in BV are BV, and the triangle inequality gives that
$$
|fg|_{\mathrm{BV}} \le 2|f|_{\mathrm{BV}}|g|_{\mathrm{BV}}+ |f|_{\mathrm{BV}}\|g\|_{L^1} + |g|_{\mathrm{BV}}\|f\|_{L^1}.
$$
To see this,
\begin{align*}
&\sum_{i=1}^m |f(z_i)g(z_i) - f(z_{i-1})g(z_{i-1})|\\
&\le \sum_{i=1}^m |f(z_i)g(z_i) - f(z_{i-1})g(z_{i})|+\sum_{i=1}^m |f(z_i)g(z_i) - f(z_{i})g(z_{i-1})|.
\end{align*}
Then taking $z^*$ to infimize $|f|$ we have
$$
|f(z_i)| \le |f(z_i)-f(z_*)| + |f(z_*)| \le |f|_{\mathrm{BV}} +  \int |f|d\mu.
$$
Hence
$$
\sum_{i=1}^m |f(z_i)||g(z_i) - g(z_{i-1})|\le |f|_{\mathrm{BV}} |g|_{\mathrm{BV}}  + \|f\|_{L^1} |g|_{\mathrm{BV}}.
$$
The other term is bounded similarly.
It will be useful shortly to note that $|\cdot |_{\mathrm{BV},E}$, the variation restricted to an interval $E=[y_-,y_+]$, satisfies the inequality
$$
|fg|_{\mathrm{BV},E} \le 2|f|_{\mathrm{BV},E}|g|_{\mathrm{BV},E}+ (y_+-y_-)^{-1}(|f|_{\mathrm{BV},E}\|g\|_{L^1(E)} + |g|_{\mathrm{BV},E}\|f\|_{L^1(E)}).
$$

We also have $\|fg\|_{L^1}\le |fg|_{\mathrm{BV}}$ and therefore we obtain the inequality $\|fg\|_s \le 5\|f\|_s\|g\|_s$ upon noting that by construction each of $|\cdot |_{\mathrm{BV}},\|\cdot\|_{L^1}$ is a lower bound for $\|\cdot \|_s$.

Finally, any BV function is bounded and we shall show that the strong norm is an upper bound for some fixed multiple of the essential sup norm. Fix $\gamma=3/4$. Given any $f$, let $x$ be given with $|f(x)|\ge \gamma\,\mathrm{ess.}\sup_{z\in X} |f(z)|$ and let $y$ be given with $\gamma|f(x)|\le \mathrm{ess.}\inf_{z\in X} |f(z)|$. We have 
$$
|f(x)| - |f(y)| \le |f(x) - f(y)|\le |f|_{\mathrm{BV}}.
$$
If $\mathrm{ess.}\sup_{z\in X} |f(z)|\le 2\, \mathrm{ess.}\inf_{z\in X} |f(z)|$ then $\|f\|_{L^1_\mu(X,\mathbb{C})}\ge \frac{1}{2}\|f\|_{L^\infty_\mu(X,\mathbb{C})}$ as required. Otherwise, $\mathrm{ess.}\sup_{z\in X} |f(z)|\ge 2\, \mathrm{ess.}\inf_{z\in X} |f(z)|$ and so
\begin{align*}
|f(x)| - |f(y)|&\ge  \gamma\,\mathrm{ess.}\sup_{z\in X} |f(z)| - \gamma^{-1}\mathrm{ess.}\inf_{z\in X} |f(z)|
\\
&\ge \gamma\,\mathrm{ess.}\sup_{z\in X} |f(z)| - \gamma^{-1}2^{-1}\mathrm{ess.}\sup_{z\in X} |f(z)|=(3/4-2/3)\|f\|_{L^\infty_\mu(X,\mathbb{C})}.
\end{align*}
Since $3/4>2/3$ the conclusuion follows.

To check \textnormal{($\mathbb{C}$-Base)} it suffices to check \textnormal{($\mathcal{H}_q$-Base)} for the norms given in Prop \ref{prop:bvequivalent}.  

\noindent(\textit{Form of weak norm:})  
This is by definition.

\noindent(\textit{Closeness criterion:})
Let $F$ be given. Let be given. Note that
\begin{align*}
&\bm{\chi}^\dag(i_0,\ldots, i_n)\mathbf{T}^\dag(i_0,\ldots, i_n)-\bm{\chi}^\dag(i_0,\ldots, i_k)\mathbf{T}^\dag(i_0,\ldots, i_k)
\\
&=\bm{\chi}^\dag(i_0,\ldots, i_n)\bm{\chi}^\dag(i_0,\ldots, i_k)(\mathbf{T}^\dag(i_0,\ldots, i_n)-\mathbf{T}^\dag(i_0,\ldots, i_k))
\\
&+(\bm{\chi}^\dag(i_0,\ldots, i_n)-\bm{\chi}^\dag(i_0,\ldots, i_k))\bm{\chi}^\dag(i_0,\ldots, i_n)\mathbf{T}^\dag(i_0,\ldots, i_n)
\\
&+(\bm{\chi}^\dag(i_0,\ldots, i_n)-\bm{\chi}^\dag(i_0,\ldots, i_k))\bm{\chi}^\dag(i_0,\ldots, i_k)\mathbf{T}^\dag(i_0,\ldots, i_k)
\end{align*}
For brevity denote these $a=a_1 + a_2 + a_3$.
For any $x\in X$ either $a_1 F(x)=0$ or there are $y_1,y_2\in [i_0,\ldots, i_k]\cap [i_0,\ldots, i_n]\subseteq [i_0,\ldots, i_k]$ with $a_1F(x) = F(y_1)-F(y_2)$. For such $x$ we have
$$
\|a_1F(x)\|_{\mathcal{H}_q} \le \|F\|_{\mathrm{BV},[i_0,\ldots, i_k],\mathcal{H}_q} .
$$
Noting the convex sum property one has
$$
\sum_{i_0,\ldots, i_k\in \Sigma^\prime_k} \|F\|_{\mathrm{BV},[i_0,\ldots, i_k],\mathcal{H}_q}\le \|F\|_{s,q}
$$
and therefore the existence of some $i_0,\ldots, i_k$ with 
$$
\|F\|_{\mathrm{BV},[i_0,\ldots, i_k],\mathcal{H}_q}\le (\#\Sigma^\prime_k)^{-1} \|F\|_{s,q}.
$$
Using $\|F(y)\|_{\mathcal{H}_q}\le \|F\|_{s,q}$ we have the respective pointwise bounds
$$
\|a_2F(x)\|_{\mathcal{H}_q}\le |\chi^\dag(i_0,\ldots, i_n)(x)-\chi^\dag(i_0,\ldots, i_k)(x)| \chi^\dag(i_0,\ldots, i_n)(x)\|F\|_{s,q},
$$
$$
\|a_3F(x)\|_{\mathcal{H}_q}\le |\chi^\dag(i_0,\ldots, i_n)(x)-\chi^\dag(i_0,\ldots, i_k)(x)| \chi^\dag(i_0,\ldots, i_k)(x)\|F\|_{s,q}.
$$
For $\mathbf{M}(g)a$ the pointwise bounds increase by $|g(x)|$. Integrating, it follows that $\|\mathbf{M}(g)aF\|_{L^1(\mu,\mathcal{H}_q)}/\|F\|_{s,q}$ is bounded by
\begin{align*}
&\|g\chi^\dag(i_0,\ldots, i_k)\chi^\dag(i_0,\ldots, i_n)\|_{L^1(\mu,\mathbb{C})} (\#\Sigma^\prime_k)^{-1} 
\\
&+\|g(\chi^\dag(i_0,\ldots, i_k)-\chi^\dag(i_0,\ldots, i_n))\|_{L^1(\mu,\mathbb{C})}.
\end{align*}

\noindent(\textit{Uniform Doeblin--Fortet inequality:})
We aim to find $m$, $\alpha<1$, $C>0$ such that for every $q\in\mathbb{N}$ and $F\in\mathcal{B}_{s,q}$ we have
\begin{align*}
&\sum_{j_0,\ldots, j_m}\|\mathbf{M}(r(j_0,\ldots, j_m)) \mathbf{T}^\dag(j_0,\ldots, j_{m})\otimes \bm{\rho}^*(j_0,\ldots, j_{m})\,(F)\|_{s,q}
\\
&\le \alpha \|F\|_{s,q} + C\|F\|_{w,q}.
\end{align*}
We follow the ideas in the proof of Theorem 1 of \cite{Wong}. 
Recall we assume that $T^\prime,1/T^\prime$ are of bounded variation and we assume that $T^\prime>\gamma^{-1}>1$ uniformly. One also has $T^\prime<\eta^{-1}<\infty$ uniformly. Let $n$ be given.
Let $i_0,\ldots, i_n$ be given. For brevity write $E=[i_0]\cap \ldots \cap T^{-n}[i_n]$. 
It follows that $(T^n)^\prime$ has finitely many jump discontinuities, which means that $(T^n)^\prime$ has a well-defined limit on either side of a discontinuity. We partition $E$ such that on $U^\prime$, the interior of a given sub-interval $E^\prime$, we have that $(T^n)^\prime$ is continuous and has a continuous extension. We write $|\cdot |_{\mathrm{BV},\mathbb{C}^{\Gamma_q},U^\prime}$ for the total variation of the continuous extension to $E^\prime$. Then we have
$$
|F |_{\mathrm{BV},q,E^\prime}\le |F |_{\mathrm{BV},\mathbb{C}^{\Gamma_q},U^\prime} + \|F(z_-)\|_{\mathbb{C}^{\Gamma_q}} + \|F(z_+)\|_{\mathbb{C}^{\Gamma_q}},
$$ 
where $z_-,z_+$ are the end-points of $E^\prime$. Recall that we assume that $1/(T^n)^\prime$ is of bounded variation. It follows that we can further subdivide into intervals $U^\prime_1,\ldots, U^\prime_N$ with
$$
|1/(T^n)^\prime |_{\mathrm{BV},\mathbb{C}^{\Gamma_q},U^\prime_j}\le\gamma^n.
$$
(We use uniform continuity in addition with the representation of $1/(T^n)^\prime$ as the difference of two monotone functions.) Note that the number $N=N_n$ will be irrelevant, and similarly the reciprocal of $L=L_{n,N}$ the infimum of the lengths $|U^\prime_j|$ will be irrelevant.
This subdivision does not create new singularities, i.e.\ we have
$$
|F |_{\mathrm{BV},\mathbb{C}^{\Gamma_q},E^\prime}\le  \|F(z_-)\|_{\mathbb{C}^{\Gamma_q}} + \|F(z_+)\|_{\mathbb{C}^{\Gamma_q}}+\sum_{j=1}^N |F |_{\mathrm{BV},\mathbb{C}^{\Gamma_q},U^\prime_j}.
$$
Therefore
\begin{align*}
&\left| \mathbf{M}(r(i_0,\ldots, i_n))\mathbf{T}^\dag(i_0,\ldots, i_n)\otimes I\,(F)\right|_{\mathrm{BV},\mathbb{C}^{\Gamma_q}}
\\
&\le \sum_{y=y_-,y_+}2|r(i_0,\ldots, i_n)(y)|\|\mathbf{T}^\dag(i_0,\ldots, i_n)\otimes I\,(F)(y)\|_{\mathbb{C}^{\Gamma_q}} 
\\
&+ \sum_{j=1}^N\left| \mathbf{M}(r(i_0,\ldots, i_n))\mathbf{T}^\dag(i_0,\ldots, i_n)\otimes I\,(F)\right|_{\mathrm{BV},\mathbb{C}^{\Gamma_q},T^n U^\prime_j} 
\end{align*}

For the variation of $ \mathbf{M}(r(i_0,\ldots, i_n))\mathbf{T}^\dag(i_0,\ldots, i_n)\otimes I\,(F)$, it suffices to bound, over any partition $(z_i)_{i=1}^m$ of $T^nU^\prime_j$,
\begin{align*}
&\sum_{i=1}^m \| \mathbf{T}^\dag(i_0,\ldots, i_n)\otimes I\,(F)(z_i) -\mathbf{T}^\dag(i_0,\ldots, i_n)(F)\otimes I\,(z_{i+1})\|_{\mathbb{C}^{\Gamma_q}} |r(i_0,\ldots, i_n)(z_i)|
\\
&+ \sum_{i=1}^m |r(i_0,\ldots, i_n) (z_i) -r(i_0,\ldots, i_n) (z_{i+1})| \|\mathbf{T}^\dag(i_0,\ldots, i_n)\otimes I\,(f)(z_i)\|_{\mathbb{C}^{\Gamma_q}}
\end{align*}

Using the fact that the composition of inverse branches is monotonic on the given domain, partitions for $T^n U^\prime_j$ are in one-to-one correspondence with partitions of $U^\prime_j$. This gives
$$
\sum_{i=1}^m \| \mathbf{T}^\dag(i_0,\ldots, i_n)\otimes I\,(F)(z_i) -\mathbf{T}^\dag(i_0,\ldots, i_n)\otimes I\,(F)(z_{i+1})\|_{\mathbb{C}^{\Gamma_q}}\le |F|_{\mathrm{BV},\mathbb{C}^{\Gamma_q},U^\prime_j}.
$$

Using the formula for $r(i_0,\ldots, i_n)$, we have for any $u_j$, 
\begin{align*}
& \sum_{i=1}^m |r(i_0,\ldots, i_n) (z_i) -r(i_0,\ldots, i_n) (z_{i+1})|
\\
&= \sum_{i=1}^m |h(z_i)\mathbf{T}^\dag(i_0,\ldots, i_n)\otimes I\,(h/(T^n)^\prime) (z_i)
\\
& -h(z_{i+1})\mathbf{T}^\dag(i_0,\ldots, i_n)\otimes I\,(h/(T^n)^\prime) (z_{i+1}))|
\\
&\le \sum_{i=1}^m |h(z_i)|\mathbf{T}^\dag(i_0,\ldots, i_n)\otimes I\,(h/(T^n)^\prime) (z_i) -\mathbf{T}^\dag(i_0,\ldots, i_n)\otimes I\,(h/(T^n)^\prime) (z_{i+1}))|
\\
&+
\sum_{i=1}^m |h(z_i)-h(z_{i+1})|\mathbf{T}^\dag(i_0,\ldots, i_n)\otimes I\,(h/(T^n)^\prime) (z_i) |
\\
&\le |h/(T^n)^\prime|_{\mathrm{BV},U^\prime_j}\|h\|_{L^\infty} + |h|_{\mathrm{BV},T^nU^\prime_j}(h/(T^n)^\prime (u_j) + |h/(T^n)^\prime|_{\mathrm{BV},U^\prime_j})
\\
&\le \|h\|_\infty 1/(T^n)^\prime(u_j)|/2\|h\|_{L^\infty} + |h|_{\mathrm{BV},T^nU^\prime_j}(3/2 \|h\|_\infty 1/(T^n)^\prime (u_j)) 
\\
&= \gamma^n(\|h\|_\infty^2  + 2|h|_{\mathrm{BV},T^nU^\prime_j}\|h\|_\infty ).
\end{align*}
We let $v_j\in T^nE$ with 
$$
\|F(v_j)\|_{\mathbb{C}^{\Gamma_q}}\le |U^\prime_j|^{-1}\int_{U^\prime_j} \|F\|_{\mathbb{C}^{\Gamma_q}}\,d\mathrm{Leb}\le L^{-1}\int_{U^\prime_j} \|F\|_{\mathbb{C}^{\Gamma_q}}\,d\mathrm{Leb}.
$$
Then
\begin{align*}
&\sum_{j=1}^N\sum_{i=1}^{m_j} |r(i_0,\ldots, i_n) (z_i) -r(i_0,\ldots, i_n) (z_{i+1})| \|\mathbf{T}^\dag(i_0,\ldots, i_n)\otimes I\,(F)(z_i)\|_{\mathbb{C}^{\Gamma_q}}
\\
&\le \gamma^n(\|h\|_\infty^2  + 2|h|_{\mathrm{BV},T^nU^\prime_j}\|h\|_\infty )\times
\\
&\left(\eta^{-n}|U^\prime|^{-1}\sum_{j=1}^N\int_{U^\prime_j} \|F\|_{\mathbb{C}^{\Gamma_q}}\,d\mathrm{Leb}+ \sum_{j=1}^N|F|_{\mathrm{BV},\mathbb{C}^{\Gamma_q},U_j^\prime}\right)
\\
&\le \gamma^n(\|h\|_\infty^2  + 2|h|_{\mathrm{BV},T^nU^\prime_j}\|h\|_\infty )\times
\left(L^{-1}\int_{E^\prime} \|F\|_{\mathbb{C}^{\Gamma_q}}\,d\mathrm{Leb}+ |F|_{\mathrm{BV},\mathbb{C}^{\Gamma_q},E^\prime}\right)
\\
&\le \gamma^n C |F|_{\mathrm{BV},\mathbb{C}^{\Gamma_q},E^\prime} + C \gamma^n L^{-1}\int_{E^\prime} \|f\|_{\mathbb{C}^{\Gamma_q}}\,d\mathrm{Leb}.
\end{align*}
Finally, we bound, for $y=y_-,y_+$,
\begin{align*}
&\|\mathbf{T}^\dag(i_0,\ldots, i_n)\otimes I\,(F)(y)\|_{\mathbb{C}^{\Gamma_q}}|r(i_0,\ldots, i_n)(y)|
\\
&\le \gamma^n|F|_{\mathrm{BV},E^\prime} + \gamma^n|E^\prime|^{-1}\int_{E^\prime} \|F\|_{\mathbb{C}^{\Gamma_q}}\,d\mathrm{Leb}.
\end{align*}
We conclude that given $\alpha<1$ there is $n$ is large enough with
\begin{align*}
&
|\mathbf{M}(r(i_0,\ldots, i_n))\mathbf{T}^\dag(i_0,\ldots, i_n)\otimes I\,(F)|_{\mathrm{BV},\mathbb{C}^{\Gamma_q}} 
\\
&\le \alpha|F|_{\mathrm{BV},\mathbb{C}^{\Gamma_q},E^\prime} + C\|F\|_{L^1(E^\prime,\mathbb{C}^{\Gamma_q})}.
\end{align*}
Summing over $E^\prime$ (which we recall are defined as intervals corresponding to some partition of $E$) we have
\begin{align*}
&
|\mathbf{M}(r(i_0,\ldots, i_n))\mathbf{T}^\dag(i_0,\ldots, i_n)\otimes I\,(F)|_{\mathrm{BV},\mathbb{C}^{\Gamma_q}} 
\\
& \le \alpha|F|_{\mathrm{BV},\mathbb{C}^{\Gamma_q},E} + C\|F\|_{L^1(E,\mathbb{C}^{\Gamma_q})} .
\end{align*}
Finally, summing over $i_0,\ldots, i_n$ (and recalling $E=[i_0]\cap \ldots \cap T^{-n}[i_n]$ ) we have
\begin{align*}
&
\sum_{i_0,\ldots, i_n}|\mathbf{M}(r(i_0,\ldots, i_n))\mathbf{T}^\dag(i_0,\ldots, i_n)\otimes I(F)|_{\mathrm{BV},\mathbb{C}^{\Gamma_q}}
\\
&  \le \alpha|F|_{\mathrm{BV},\mathbb{C}^{\Gamma_q}} + C\|F\|_{L^1(X,\mathbb{C}^{\Gamma_q})} .
\end{align*}
Recall that $\left\|\cdot \right\|_{s,q} =|\cdot |_{\mathrm{BV},\mathbb{C}^{\Gamma_q}} + |\cdot |_{L^1(X,\mathbb{C}^{\Gamma_q})}$. We have already taken care of the $|\cdot |_{\mathrm{BV},\mathbb{C}^{\Gamma_q}}$ part. We simple add the constant coming from the $ |\cdot |_{L^1(X,\mathbb{C}^{\Gamma_q})}$ norm.
\end{proof}


\section{Examples of extensions and monodromy}\label{section:monodromyexamples}
An expanding interval map need not admit a Markov partition, and it is the goal of this section to give an example of such a map that also fulfills the monodromy requirements in Definition \ref{def:monodromy3}. We begin with an intermediate case where of a Markov map equipped with the ``wrong" partition. The hypotheses of the following appear in the context of infinite graph extensions \cite{SJR} (but without the dimension restriction).
\begin{lemma}
Let $T$ be a full-branch expanding interval map with a partition $\alpha$ whose boundaries satisfy that $\cup_{n\in\mathbb{N}}\partial [i]$ is nowhere dense. Let $\psi:X\to\Gamma$ constant on each element of $\alpha$. If the skew product of $T$ and $\psi$ is transitive then, up to refining the partition $\alpha$ by finitely many intervals, the following holds. There is a positive measure $[u]\in\alpha$ such that for each $\gamma\in \Gamma$ there are $i_0^\gamma,\ldots, i^{\gamma}_m$ with  $i_0^\gamma=i_m^\gamma = u$, such that the monodromy of $[i_0^\gamma]\cap \cdots T^{-m}[i^{\gamma}_m]$ is $\gamma$, and such that
$$
[u]= T^m([i_0^\gamma]\cap \cdots \cap T^{-m}[i^{\gamma}_m]).
$$
\end{lemma}
\begin{proof}
The inducing argument of \cite{SJR} yields that there is a positive measure set $U_0$ (a cylinder for an underlying Markov partition) such that inducing on $U_0$ yields a full-branch Gibbs-Markov map; and moreover the corresponding (Markov) partition $\beta$ of $U_0$ has that the monodromy is constant along each cylinder in $B\in \beta$. Transitivity implies that there are $B_s\in\beta$ whose monodromy is $s\in S$. Since $X$ is an interval, and since $U_0$ is an interval contained entirely in some element of $\alpha$, we may refine the partition $\alpha$ to a partition containing $U_0$. Now, given $\gamma=s_1\cdots s_n\in S^n$ we may construct $i_0^\gamma,\ldots, i^{\gamma}_m$, with $m=k_1+\ldots + k_n$, with $i_0^\gamma=i_m^\gamma = U_0$, and
$$
B_{s_1}\cap\cdots \cap T^{-\sum_{j=1}^{n-1} k_j}B_{s_n} \subseteq [i_0^\gamma]\cap \cdots \cap T^{-m}[i^{\gamma}_m]\subseteq U_0\cap T^{-m} U_0,
$$ 
where $k_1,\ldots, k_n$ are the smallest integers with $T^{k_i}B_{s_i}=U_0$. It follows that
$$
U_0 = T^m(B_{s_1}\cap\cdots \cap T^{-k_n}B_{s_n}) \subseteq T^m([i_0^\gamma]\cap \cdots \cap T^{-m}[i^{\gamma}_m])\subseteq U_0.
$$
\end{proof}

Let $X=[0,1]$ and, for $\beta\in (1,3),\sigma\in(0,1)$, let $T_{\beta,\sigma}:[0,1]\to [0,1]$ be the map given by $T_{\beta,\sigma}(x) = \left\{ \beta (x-\sigma) \right\}+\sigma$, where $\left\{  \cdot \right\}$ denotes the fractional part of a number. We equip $(X,T_{\beta,\sigma})$ with the absolutely continuous (with respect to Lebesgue) invariant probability measure $\mu_{\beta,\sigma}$. Let $\alpha=\alpha_{\beta,\sigma}$ be the partition into intervals of continuity for $T_{\beta,\sigma}$. Let $\Gamma = C_{r_1}(a)\ast C_{r_2}(b)$, where $C_r(x)$ is the cyclic group generated by $x$ whose order is $r\in \mathbb{N}\cup\left\{\infty\right\}$ and $\ast$ denotes the free product. We use same notation for each $\Gamma$ and denote $S=\left\{e,a,b\right\}\subseteq \Gamma$ where $e$ denotes the identity in $\Gamma$. Let $\psi:X\to \Gamma$ be the map taking the values: $a$ for the element of $\alpha$ containing $0$, $b$ for the element of $\alpha$ containing $1$, and the identity $e$ otherwise. For $\epsilon<1/2$, let $\beta_0\in  (1,2),\sigma_0\in (0,1)$ such that for any $\beta\in (\beta_0,2),\sigma\in (0,\sigma_0)$ we have that $\mu_{\beta,\sigma}([0])\ge \epsilon$.

\begin{lemma}\label{lemma:beta}
Let $\kappa^\prime<1$ and let $\Gamma=C_{r_1}(a)\ast C_{r_2}(b)$. For all $\beta\in (\beta_0,2),\sigma\in (0,\sigma_0)$ sufficiently close to $(2,0)$ (depending on $r_1,r_2$), there is $\ell$ large enough such that the extension $(\ell,D, L^1, (\kappa^\prime)^n)$-sees $S^n\subset C_{r_1}(a)\ast C_{r_2}(b)$ for all sufficiently large $n\in\mathbb{N}$. 

Let $\kappa^\prime<1$, let $\Gamma = F_{a,b}$ and fix $N$ with $(\kappa^\prime)^N<\epsilon/2$.  Given $D$, there is $\ell$ sufficiently large such that for $N\in\mathbb{N}$ and all $\beta\in (\beta_0,2),\sigma\in (0,\sigma_0)$ sufficiently close to $(2,0)$ (depending on $N$) the monodromy $(\ell, D,L^1, (\kappa^\prime)^N)$-sees $S^N\subset F_{a,b}$.
\end{lemma}

\begin{proof}
If $\sigma$ is sufficiently close to $0$ then we have $[1]\subseteq T[0]$ and $[0]\subseteq T[1]$.
We can write any word $w=w_1\ldots w_n$, $w_i\in\left\{0,1\right\}$ as a word in $\cup_{k,r\in\mathbb{N}}\left\{0_k,1_r\right\}$, where $0_k$ denotes $k$ consecutive $0$s and $1_r$ denotes $r$ consecutive $1$s; and every word starting with $0$ can be expressed as $u_1v_2\ldots u_n$ or $u_1v_2\ldots u_nv_n$ where $u_i=0_{k_i}$, $v_i=1_{k_i}$. When $\sigma>0$ there is $k_{\beta,\sigma}(0)$ such that $[0_k]$ is empty for $k>k_{\beta,\sigma}(0)$ and has $[1]\subset T^k([0_k])$ for all $k<k_{\beta,\sigma}(0)$. Similarly one defines $k_{\beta,\sigma}(1)$

It follows that any words $w^\prime, w$ of the form $u_1v_2\ldots u_n$ or $u_1v_2\ldots u_nv_n$, with corresponding $k_{2i+j}<k_{\beta,\sigma}(j)/2$, one has
\begin{equation}\label{eq:useful}
[1]\subset T^m [w^\prime w0],
\end{equation}
with $m$ the length of $w^\prime w$. 
To see this it suffices to show that $[1]\subset T^m [w0]$ for such $w$ ending in $1$. Indeed, let $y\in [1]$ and take $x_n\in [v_n]$ such that $T^{|k_n|}x_n =y$. Now as $x_n\in [0]$ let $x_{n-1}\in [u_{n-1}]$ be given with $T^{|k_{n-1}|}x_{n-1}=x_n$. We repeat this until we reach $x_1\in u_1$ which by construction satisfies $x_1\in [u_1v_2\ldots v_n]$ and $T^m x_1=y$.

Hence, using (\ref{eq:useful}), if we fix some $w^\prime$ then all such $w$ give a surjection to $C_{r_0}(a)\ast C_{r_1}(b)$ since $k_{\beta,\sigma}(i)\ge 2r_i$ -- this gives the existence of embeddings. Moreover, the number of $w^\prime$ of a given length $\ell$ is increasing. Now the role of $\ell$ is just to make the error $D/\#\Sigma_{\ell}(S^n)$ sufficiently small in order to apply Proposition \ref{prop:monodromy}, as we have already obtained a uniform lower bound for $t_{n,\ell}$ in (\ref{eq:useful}). 

On the other hand, given $N$ we choose $r_i$ large enough such that the $N$-ball in $F_{a,b}$ includes to the $N$-ball in  $C_{r_0}(a)\ast C_{r_1}(b)$. Now, given $r_i$, $i=0,1$, one can find neighbourhoods for which $k_{\beta,\sigma}(i)\ge 2r_i$. The argument follows as before.
\end{proof}

\begin{corollary}\label{cor:example}
Let $\Gamma=F_{a,b}$.
If $(\Gamma_q)_q$ are an expander family then for all $\beta\in (\beta_0,2),\sigma\in (0,\sigma_0)$ sufficiently close to $(2,0)$, the map $(X,T,\mu)=([0,1],T_{\beta,\sigma},\mu_{\beta,\sigma})$, and extension $\psi$ as in the previous, has a uniform rate of decay of correlations: there is $\delta<1$ such that for all $q\in\mathbb{N}$ there is $C_q$ with
$$
\left|\int f_q \cdot g_q\circ T^n_q \,d\mu_q - \int f_q  \,d\mu_q \int g_q \,d\mu_q \right| \le C_qe^{ -n\beta}\|f_q\|_{\mathcal{F}_s(T_q)}\|g_q\|_{L^{1}_{\mu_q}(X_q,\mathbb{C})} 
$$
for any $f_q\in  \mathcal{F}_s(T_q)$ and $g_q\in L^{1}_{\mu_q}(X_q,\mathbb{C})$.
\end{corollary}

\begin{proof}
One has $\mathcal{H}_q$ is the orthogonal complement to the constant function in $\mathbb{C}^{\Gamma_q}$. Further, $\kappa(\left\{ a,b\right\},(\mathcal{H}_q)_q) = \kappa(\left\{ a,a^{-1},b,b^{-1}\right\},(\mathcal{H}_q)_q)>0$. Using the parallelogram law it follows that $\kappa^\prime(S,(\mathcal{H}_q)_q)<1$. All that remains is to see that $(X,T,\mu)$ satisfy Definition \ref{exampleB}. The existence of the bounded variation density for the ACIP is standard. For $\beta,\sigma$ given in Lemma \ref{lemma:beta} we have $T^m[0]=T^m[1]=X$ for some $m\in\mathbb{N}$ which says exactly that $(X,T,\mu)$ satisfy the covering property and this implies that $h$ is bounded away from $0$ \cite{Liverani} and moreover implies that $(X,T,\mu)$ is mixing. We may now apply Theorem \ref{theorem:one}.
\end{proof}

\begin{remark}
We expect that the leading constant $C_q$ in \ref{cor:example} can be made uniform in $q\in\mathbb{N}$ but leave this to a separate work.
\end{remark}


\end{document}